\documentclass{article}
\usepackage[preprint]{icml2026}
\icmlpreprintfalse
\usepackage[T1]{fontenc}
\usepackage{amsmath,amssymb,amsthm,mathtools,booktabs,graphicx,microtype}
\usepackage{enumitem}
\usepackage[hidelinks]{hyperref}
\usepackage{xurl}

\usepackage{caption}
\setcitestyle{round}

\newtheorem{theorem}{Theorem}[section]
\newtheorem{lemma}[theorem]{Lemma}
\newtheorem{proposition}[theorem]{Proposition}
\newtheorem{corollary}[theorem]{Corollary}
\theoremstyle{definition}

\theoremstyle{remark}
\newtheorem{remark}[theorem]{Remark}
\allowdisplaybreaks[2]
\hypersetup{pdfauthor={},pdftitle={Information Limits of Low-Rank Approximation Certification}}

\title{Information Limits of Low-Rank Approximation Certification}

\author{
  \mbox{Kang Liu\textsuperscript{1}}
  \qquad
  \mbox{Bohao Qu\textsuperscript{2}}
}

\begin{document}

\maketitle

\footnotetext[1]{School of Future Technology, Xi'an Jiaotong University, Xi'an, China. Corresponding author: Kang Liu (\texttt{kanyo@foxmail.com}). \(^2\)Centre for Frontier AI Research (CFAR), A*STAR, Singapore.}

\begin{abstract}
Low-rank approximation can require additional matrix--vector products to verify that its error meets a prescribed tolerance. We characterize this certification cost for both relative matrix error and mean-square output error. For a single approximation matrix candidate, we determine the exact dimension-uniform minimax query constant as the allowed failure probability vanishes. Our main result concerns reusing validation responses as the approximation space expands. For a candidate family constructed independently of validation, one batch supports an entire nested path without increasing the query budget with the number of checks. Across \(W\) paths, a concentration bound exploiting shared residual energy yields a \(\sqrt{\log(W+1)}\) dependence. A matching lower bound establishes its optimality for fixed interior error targets and sufficiently small separation gaps. Finally, we compare two uniformly valid certificates on the same dispersed-spectrum family. Optimizing the validation budget within each rule family yields costs of orders \(N^{1/3}\) and \(N^{2/3}\) for validation and construction beyond the true target. Code is available at \url{https://anonymous.4open.science/r/Low-rank-approximation-1275/}.
\end{abstract}
\section{Introduction}

Low-rank approximations reduce computation and storage by
representing high-dimensional linear operators through a small
set of basis directions and their coefficients
\citep{halko2011finding}, with applications in model compression
\citep{denton2014linear}, dimensionality reduction, and large
Jacobian approximation.
Since the appropriate rank is often unknown, algorithms typically
expand the approximation space until a prescribed accuracy is
reached \citep{yu2018efficient,liu2025adaptive}.

For large operators that can only accessed through matrix--vector products
(MVPs), such as neural-network Jacobians, determining whether
this accuracy has been reached can itself be costly.
An approximation may already meet the target while the algorithm
continues working because it lacks sufficient evidence to stop.
We study how to reliably decide whether a computed low-rank
approximation is accurate enough to stop.
The original operator \(A\) is accessed through forward MVPs,
and its current approximation is \(PA\), where \(P\) orthogonally
projects onto the approximation space.
Certification must reject approximations whose error is at or
above the target and accept those below it by a prescribed gap,
with both types of incorrect decision controlled by a given
failure probability.
We consider relative squared Frobenius error and relative
mean-square output error, assuming a known input second moment
for the latter.
The question is how many MVPs are sufficient for
these guarantees, and how many are unavoidable regardless
of the certification procedure.

For a given approximation, Gaussian probing provides a concrete
way to obtain such guarantees
\citep{halko2011finding,epperly2024error}.
The squared norms of the residual responses estimate the
Frobenius error, while transforming the probes using the known
input second moment gives the corresponding output-error
estimates.
We analyze how often these estimates lead to an incorrect
decision, obtaining a sufficient query budget.
To establish a lower bound, we construct hard families of
operators that require opposite decisions but produce
statistically similar observations.
This argument covers adaptive procedures that choose new queries
using previous responses.
In the worst case over matrices of arbitrary dimensions,
the upper and lower bounds have the same leading constant
as the allowed failure probability tends to zero.

As the rank increases, checking each new approximation with
fresh probes could repeatedly incur this validation cost.
We therefore study whether the previously computed responses
of \(A\) can support the subsequent checks.
For a family of approximations constructed independently of
validation, one batch suffices along an entire nested expansion
path, with no increase in query budget as more approximations
are checked.
For simultaneous certification across \(W\) such paths,
we obtain an optimal \(\sqrt{\log(W+1)}\) dependence on \(W\),
for fixed interior error targets and sufficiently small gaps.
The proof uses the ordering of residuals within each path
and a concentration bound exploiting their shared energy
constraint across paths.
A matching lower bound shows that adaptive forward queries
cannot improve this dependence in the worst case.

Finally, we study how the validation budget affects the work
performed before stopping.
Fewer probes reduce the immediate validation cost, but the
resulting uncertainty may cause further rank expansion after
the approximation has already met its target.
We account for this trade-off by adding validation queries
to the expected construction cost beyond the true target.
On a common dispersed-spectrum matrix family indexed by \(N\),
we compare an energy-ratio test and a block-sign certificate
along the same expansion path.
Both provide uniform reliability guarantees, but optimizing
the validation budget within each rule family yields excess
costs of orders \(N^{2/3}\) and \(N^{1/3}\), respectively.
This separation quantifies how the choice of certificate
affects the balance between collecting more evidence and
continuing construction.

Our main contributions are summarized as:
(1) an exact minimax high-confidence query constant;
(2) a shared-energy concentration bound and optimal
\(\sqrt{\log(W+1)}\) dependence for simultaneous certification;
and (3) a stopping-cost separation between two uniformly
valid certification rules.

\section{Related Work}
\label{sec:related_work}

\paragraph{Adaptive low-rank approximation.}
Randomized range finding constructs approximation spaces from operator
applications with probabilistic error guarantees
\citep{halko2011finding}.
Blocked factorizations and incremental QB schemes support adaptive
stopping and fixed-precision approximation
\citep{martinsson2016randomized,yu2018efficient}, while randomized
numerical-rank estimation helps select suitable approximation
dimensions \citep{meier2024fast}.
Recent work develops blockwise deflation with spectral and Frobenius
guarantees \citep{liu2025adaptive}, and matrix-free residual indicators
with rank pruning and adjoint-free basis construction
\citep{smith2026adaptive}.
Leave-one-out and jackknife diagnostics additionally assess errors
and variances using existing construction data
\citep{epperly2024error}.
We study the information required to certify a supplied approximation
and the additional construction incurred while waiting for a reliable
stopping decision.

\paragraph{Error estimation and query complexity.}
The identity $\|R\|_F^2=\operatorname{tr}(R^\top R)$ connects
residual-energy estimation to stochastic trace estimation.
Classical analyses give relative-error sample bounds with prescribed
failure probabilities \citep{avron2011randomized,roosta2015improved},
with further concentration results for indefinite matrices
\citep{cortinovis2022randomized}.
Hutch++ and its adaptive variants combine low-rank approximation
with residual sampling \citep{meyer2021hutchpp,persson2022improved};
XTrace exploits exchangeability and spectral structure
\citep{epperly2024xtrace}.
Query lower bounds address scalar quadratic-form access
\citep{wimmer2014optimal} and matrix--vector access
\citep{woodruff2022dynamic}.
The oracle distinction prevents direct transfer: our oracle returns
the full vector $Av$, revealing $v^\top A^\top Av=\|Av\|_2^2$,
but does not directly provide $A^\top Av$.
Our bounds concern relative-error under this oracle,
including an exact high-confidence minimax constant.

\paragraph{Shared information and simultaneous guarantees.}
Dynamic trace estimation exploits small matrix updates
\citep{dharangutte2021dynamic,woodruff2022dynamic}, while
parameter-dependent estimators reuse probes and operator products
through polynomial approximation \citep{matti2026parameter}.
More generally, bounds for suprema of matrix-indexed quadratic
processes capture dependencies that separate concentration bounds
ignore \citep{krahmer2014suprema,dirksen2015tail}.
Here the relevant structure is the common envelope
$0\preceq E_i\preceq T$, together with ordered residuals along each
path.
Proposition~\ref{prop:shared_envelope} controls the normalized energy
estimates jointly; boundary reduction and a matching lower bound
then give the optimal $\sqrt{\log(W+1)}$ width dependence in
Theorem~\ref{thm:simultaneous_certification}.
The complete candidate family is independent of validation,
although its traversal and final selection may be adaptive.

\section{Optimal Certification of a Single Approximation}
\label{sec:single_certification}

We begin with the query cost of certifying an approximation, which establishes the information requirement for decisions before
considering repeated checks and stopping along the expansion path.

\subsection{Problem formulation}
\label{sec:certification_problem}

Let \(A\in\mathbb R^{m\times n}\) be nonzero.
An orthonormal basis \(U\in\mathbb R^{m\times r}\) defines the known
orthogonal projector \(P=UU^\top\) and approximation \(PA\).
Its relative squared Frobenius error is
\(q(P)=\|(I-P)A\|_F^2/\|A\|_F^2\), the fraction of energy outside
the retained output space.
Given \(0<g<\rho<1\) and \(0<\delta<1/2\), certification requires
\[
\begin{aligned}
q(P)\ge\rho
&~\Longrightarrow~
\mathbb P\{\mathrm{accept}\}\le\delta,\\
q(P)\le\rho-g
&~\Longrightarrow~
\mathbb P\{\mathrm{reject}\}\le\delta.
\end{aligned}
\]
Either decision is allowed in the intervening region.
Thus, \(g\) specifies the accuracy margin at which acceptance
must become reliable.

Each forward query \(v\mapsto Av\) returns the full output vector.
Query directions may depend on previous responses and private
randomness, but every run uses at most \(s\) queries.
Let \(R_s(\rho,g)\) be the smallest achievable worst-case error
probability over the two decision regions, with the worst case
taken over all dimensions, nonzero matrices, and supplied projections.
Procedures may use the dimensions and projection.
The optimal query budget is
\(s^\star(\delta;\rho,g)=\min\{s:R_s(\rho,g)\le\delta\}\).

Upper bounds also apply conditionally on a construction stage
independent of the validation probes, counting only additional
validation queries. Lower bounds, by contrast, charge for all
operator observations available to the procedure, including any
supplied construction responses; no construction transcript is
provided free of charge.

\subsection{Gaussian certification and optimality}
\label{sec:gaussian_optimality}

Following Gaussian residual estimation \citep{halko2011finding},
draw independent \(g_i\sim\mathcal N(0,I_n)\) and cache
\(Y=AG\), where \(G=[g_1,\ldots,g_s]\).
These \(s\) queries give the estimate
\(\widehat q=\|(I-P)Y\|_F^2/\|Y\|_F^2\).
Define the log-odds map \(\ell(x)=\log[x/(1-x)]\) and the transformed
gap \(\Delta=\ell(\rho)-\ell(\rho-g)\).
The certificate accepts when
\(
\ell(\widehat q)
\le ({\ell(\rho)+\ell(\rho-g)})/{2},
\)
with the natural endpoint conventions.
This midpoint gives the same deviation margin to the two
decision regions.

\begin{theorem}[Information limits for a single approximation]
\label{thm:single_certification}
Fix \(0<g<\rho<1\) and
\(\Delta=\ell(\rho)-\ell(\rho-g)\).
Let \(p_s(u)=\mathbb P\{F_{s,s}\le e^{-u}\}\), where \(F_{s,s}\)
is the ratio of two independent \(\chi_s^2\) variables.
For every integer \(s\ge1\),
\(
p_s(\Delta/2)
\le R_s(\rho,g)
\le \exp[-s\log\cosh(\Delta/4)].
\)
Consequently,
\[
\lim_{\delta\downarrow0}
[{s^\star(\delta;\rho,g)}/{\log(1/\delta)}]
={1}/{\log\cosh(\Delta/4)}.
\]
\end{theorem}
\vspace{-1pt}
Since \(-s^{-1}\log p_s(u)\to\log\cosh(u/2)\), the two bounds share
the same exponential rate.
The upper bound follows by analyzing residual and captured energy
jointly.
Write \(T=A^\top A\), \(E=A^\top(I-P)A\), \(C=T-E\), and
\(Z(M)=s^{-1}\sum_i g_i^\top Mg_i\).
For \(q=q(P)\in(0,1)\),
\[
L:=\ell(\widehat q)-\ell(q)
=\log({Z(E)}/{Z(C)})
-\log({\operatorname{tr}E}/{\operatorname{tr}C}).
\]
The common denominator cancels, leaving a comparison between
two energy estimates computed from the same probes.
The cases \(q=0,1\) are exact almost surely.

To control the lower tail, normalize
\(\bar E=E/\operatorname{tr}E\) and
\(\bar C=C/\operatorname{tr}C\).
For \(u>0\), the event \(L\le-u\) is equivalent to \(Z(H)\le0\),
where \(H=\bar E-e^{-u}\bar C\).
If \(p\) and \(r\) denote the total positive and absolute negative
eigenvalue masses of \(H\), then
\(p\le1\), \(r\le e^{-u}\), and \(p-r=1-e^{-u}\).
Grouping the positive and negative eigenvalue factors gives,
for \(0<t<e^u/2\),
\[
\det(I_n+2tH)
\ge (1+2tp)(1-2tr)
\ge (1+2t)(1-2te^{-u}).
\]
The second inequality uses \(pr\le e^{-u}\).
For a standard Gaussian vector \(z\), the identity
\(\mathbb E e^{-tz^\top Hz}=\det(I_n+2tH)^{-1/2}\)
therefore bounds the exponential moment without assuming
independence between the two energy estimates.
Exponential Markov with \(t=(e^u-1)/4\), followed by interchanging
\(E\) and \(C\), yields
\[
\mathbb P\{L\le-u\},~
\mathbb P\{L\ge u\}
\le \exp[-s\log\cosh(u/2)].
\]
Each incorrect decision requires a deviation of at least
\(\Delta/2\), proving the upper bound.
Appendix~\ref{app:single:joint_moment} gives the complete
moment and tail arguments.

For the lower bound, consider two Gaussian matrix priors with
\(P=\operatorname{diag}(0,1)\).
Under hypothesis \(h\in\{0,1\}\), the columns of \(A\) are independent
\(\mathcal N(0,\Sigma_h/n)\) vectors.
The diagonal, determinant-one covariances are chosen so that the
error log-odds concentrate at \(\ell(\rho)+\eta\) and
\(\ell(\rho-g)-\eta\), respectively, for an auxiliary \(\eta>0\).
Therefore, it is necessary to distinguish these priors for reliable
certification as \(n\) grows.

To account for adaptive queries, decompose each new direction as
\(v=v_\parallel+v_\perp\) relative to the span of previous queries.
The parallel response is already known, whereas input isotropy
gives the conditional representation
\begin{equation}
Av=Av_\parallel
+\frac{\|v_\perp\|_2}{\sqrt n}\,\xi,
~~
\xi\sim\mathcal N(0,\Sigma_h).
\label{eq:gaussian_innovation}
\end{equation}
The innovation is independent of the past.
Consequently, at most \(s\) independent covariance observations
simulate the complete adaptive transcript through a rule
independent of \(h\).
Lemma~\ref{app:single:innovation} proves this representation
and the transcript reduction.

For the chosen covariances, the optimal equal-prior test on these
observations has error \(p_s(\Delta/2+\eta)\).
The probability that a prior draw lies outside its intended
decision region vanishes as \(n\to\infty\).
Taking this limit and then \(\eta\downarrow0\) gives the lower bound;
Proposition~\ref{app:single:finite_lower} supplies the explicit
test and finite-dimensional correction.
Finally,
\(-s^{-1}\log p_s(u)\to\log\cosh(u/2)\)
by Lemma~\ref{app:single:f_asymptotics}, so the risk bounds have
the same exponential rate and yield the stated query constant.

For fixed interior \(\rho\), taking the high-confidence limit
before \(g\downarrow0\) gives the leading cost
\(32\rho^2(1-\rho)^2g^{-2}\log(1/\delta)\).
Appendix~\ref{app:single:minimax} derives this expansion.
Its \(g^{-2}\) dependence provides the precision scale for
the multiple-approximation problem studied next.

The same framework also certifies mean-square output error
by incorporating the input second moment into the probes.

\begin{corollary}[Mean-square output certification]
\label{cor:output_certification}
Suppose \(K=\mathbb E[xx^\top]\succeq0\) is known and
\(\operatorname{tr}(AKA^\top)>0\).
With \(B=AK^{1/2}\),
\[
q_K(P):=
\frac{\mathbb E\|Ax-PAx\|_2^2}
     {\mathbb E\|Ax\|_2^2}
=
\frac{\|(I-P)B\|_F^2}{\|B\|_F^2}.
\]
For every given \(K\), probes \(K^{1/2}g_i\) give the same
Gaussian upper bounds.
When the worst case also ranges over all admissible known
second moments, the minimax risk equals \(R_s(\rho,g)\),
with the same exact high-confidence query constant.
\end{corollary}
The input distribution need not be Gaussian.
The reduction and the minimax statement are proved in
Appendix~\ref{app:single:functional}.

\section{Reusing Validation Across Approximation Paths}
\label{sec:validation_reuse}

Once the responses \(Y=AG\) have been computed, another approximation
\(PA\) can be checked by projecting the same \(Y\), without further
queries to \(A\).
The question is whether these repeated decisions remain jointly
reliable.
Throughout this section, the complete approximation family is finite
or countable and constructed independently of the validation probes
\(G\), with \(\rho,g\) fixed in advance.
Validation outcomes may guide visits, switches, and final selection
within this family, but cannot be used to construct new members.
Our guarantees concern the simultaneous correctness of all required
decisions, which also protects a selection made using those outcomes.

\subsection{Repeated checks along a nested path}
\label{sec:nested_validation}

Suppose each expansion retains the existing approximation space:
\(\operatorname{range}(P_j)\subseteq\operatorname{range}(P_{j+1})\).
The residual energy matrices
\(E_j=A^\top(I-P_j)A\) then satisfy
\(E_{j+1}\preceq E_j\preceq T\).
Consequently, both the true errors \(q_j\) and their estimates
\(\widehat q_j=Z(E_j)/Z(T)\) decrease along the path.

This ordering reduces repeated certification to two boundary decisions.
On a finite path, approximations with \(q_j\ge\rho\) form a prefix.
Under a common acceptance threshold, accepting any member of this
prefix implies accepting its last member.
Similarly, approximations with \(q_j\le\rho-g\) form a suffix,
and rejecting any member implies rejecting its first member.
Thus, every required decision is correct if these two boundary
decisions are correct.

The boundaries depend on the true errors and are used only for
analysis; the procedure need not identify them.
Applying Theorem~\ref{thm:single_certification} to at most two
boundaries gives a query budget independent of the number of checks.
Lemma~\ref{app:reuse:boundary-lemma} formalizes this reduction.

\subsection{Optimal cost of comparing multiple paths}
\label{sec:multiple_paths}

An approximation family has \emph{path width} at most \(W\) if its
residual energy matrices can be covered by \(W\) chains in the
positive semidefinite order.
This includes \(W\) nested expansion paths, regardless of their
lengths.
Let \(s^\star_{\mathrm{sim}}(\delta;\rho,g,W)\) be the least query
budget ensuring that all required decisions are correct with
probability at least \(1-\delta\), uniformly over dimensions,
operators, and known families of width at most \(W\).
We use the query and information-accounting model of
Section~\ref{sec:certification_problem}.
The chain cover describes the instance class; the certification
procedure does not need to compute it.

\begin{theorem}[Information limits for simultaneous certification]
\label{thm:simultaneous_certification}
Fix \(\rho\in(0,1)\), and set
\(w=\min\{\rho,1-\rho\}\).
For integers \(W\ge1\), \(0<g\le w/32\), and
\(0<\delta\le1/4\), under the independent-family model above,
\[
s^\star_{\mathrm{sim}}(\delta;\rho,g,W)
\asymp_\rho
\frac{\sqrt{\log(W+1)}+\log(1/\delta)}{g^2}.
\]
The comparison constants depend only on \(\rho\).
Independent Gaussian probes and the rule
\(\widehat q_\lambda\le\rho-g/2\) attain the upper bound.
The lower bound allows arbitrary adaptive forward queries
with a hard budget.
\end{theorem}

The arithmetic midpoint used here targets the additive separation
gap.
For fixed \(W,\rho,g\), the log-odds midpoint from the preceding
section also preserves the exact high-confidence constant
\(1/\log\cosh(\Delta/4)\);
Proposition~\ref{app:reuse:fixed-width-proposition} gives this
consequence of the boundary reduction.
To obtain the smaller dependence on a growing \(W\), we use the
additional fact that all residual energies satisfy
\(0\preceq E_\lambda\preceq T\).

\begin{proposition}[Concentration under a shared energy envelope]
\label{prop:shared_envelope}
Let \(T\succeq0\) satisfy \(\operatorname{tr}T>0\), and let
\(0\preceq E_i\preceq T\), \(1\le i\le M\), be fixed independently
of the Gaussian probes.
Set
\(q_i=\operatorname{tr}(E_i)/\operatorname{tr}(T)\) and
\(\widehat q_i=Z(E_i)/Z(T)\).
For every integer \(s\ge1\) and \(0<\alpha<1\),
\begin{equation}
\mathbb P\!\left\{
\max_{i\le M}|\widehat q_i-q_i|
>
16\sqrt{
\frac{\sqrt{\log(2M)}+\log(5/\alpha)}{s}
}
\right\}
\le\alpha.
\label{eq:shared_envelope_concentration}
\end{equation}
\end{proposition}

To explain the dependence on \(M\), normalize
\(\operatorname{tr}T=1\) and write
\(E_i=T^{1/2}K_iT^{1/2}\), where \(0\preceq K_i\preceq I\).
Centering the ratio gives
\(
F_i=E_i-q_iT,~
\operatorname{tr}F_i=0,~
\widehat q_i-q_i={Z(F_i)}/{Z(T)}.
\)
The shared envelope constrains how these centered matrices can
vary across directions.

Split \(T\) into a head containing its largest \(r\) eigenvalues
and a remaining tail.
For the head, \(\|F_{i,HH}\|_*\le1\), where
\(\|\cdot\|_*\) denotes the nuclear norm.
Every centered head contribution is therefore bounded by the
operator norm of the same centered sample covariance.
Gaussian singular-value concentration
\citep[Corollary~5.35]{vershynin2012nonasymptotic}
controls all heads together at leading scale \(\sqrt{r/s}\),
without a union bound over \(i\).

For the tail, trace normalization gives
\(\|T_L\|_{\mathrm{op}}\le1/r\) and
\(\|F_{i,LL}\|_F^2\le1/r\).
A union bound over the tail quadratic forms then gives the leading
scale \(\sqrt{\log(2M)/(rs)}\).
The interaction between head and tail is controlled by
\[
K_{i,HL}K_{i,LH}
\preceq K_{i,HH}-K_{i,HH}^2
\preceq I/4,
\]
which follows from \(K_i^2\preceq K_i\).
It yields \(\|F_{i,HL}\|_F^2\le1/(4r)\); conditioning on the
Gaussian head controls the cross contribution at the same
leading scale as the tail.
Thus, the argument also covers residual matrices that do not
commute.

Choosing \(r=\lceil\sqrt{\log(2M)}\rceil\) balances the two scales:
\[
\sqrt{\frac rs}
\asymp
\sqrt{\frac{\log(2M)}{rs}}
\asymp
\left(\frac{\sqrt{\log(2M)}}{s}\right)^{1/2}.
\]
If the ambient dimension is smaller, using the entire space as
the head gives an equal or smaller bound.
The estimation error therefore grows as
\((\log M)^{1/4}/\sqrt{s}\), producing a
\(\sqrt{\log M}\) query factor when accuracy proportional to \(g\)
is required.
Appendix~\ref{app:reuse:proof} supplies the complete block estimates,
confidence terms, and control of the common denominator \(Z(T)\).
The spectral split is used only in the proof and requires no
spectral information from the procedure.

Applying Proposition~\ref{prop:shared_envelope} once to the at most
\(2W\) boundary approximations, with \(\alpha=\delta\), proves
the upper bound in Theorem~\ref{thm:simultaneous_certification}.
Indeed, estimation error at most \(g/3\) puts both decision regions
strictly on the correct sides of \(\rho-g/2\).
For a countable family, each finite subset still has at most
\(2W\) boundaries, so the same budget applies uniformly;
increasing finite subsets then give the full-family guarantee.
Appendix~\ref{app:reuse:width-application} gives this application
and its explicit budget.

For the lower bound, begin with a covariance-testing problem.
A packing of output subspaces yields \(M\) trace-one covariances
\(\Sigma_i\) and fixed residual projectors \(R_j\), with
\(\log M\asymp d^2\), such that
\begin{equation}
\operatorname{tr}(R_i\Sigma_i)=\rho-2g,
~~
\operatorname{tr}(R_j\Sigma_i)\ge\rho+2g
~(j\ne i).
\label{eq:hard_family_separation}
\end{equation}
At the same time, the divergence between any two Gaussian
observation laws satisfies
\(\mathrm{KL}(\mathcal N(0,\Sigma_i)\|
\mathcal N(0,\Sigma_j))\lesssim_\rho g^2d\).
Lemma~\ref{app:width:packing} gives the packing, and
Appendix~\ref{app:width:covariance-section} constructs the
covariances and projectors.

For these energy ratios, simultaneous correctness forces the
accepted set to be exactly \(\{i\}\), identifying the hidden
covariance label.
Thus, decoding error is no larger than certification error.
Fano's inequality \citep{scarlett2019fano} requires, for a fixed
small error probability,
\(
s g^2d\gtrsim_\rho\log M\asymp d^2,
\text{hence}~
s\gtrsim_\rho{\sqrt{\log M}}/{g^2}.
\)

To transfer this bound to the operator oracle, draw the columns
of \(A\) independently from \(\mathcal N(0,\Sigma_i/n)\).
The Gaussian innovation representation
\eqref{eq:gaussian_innovation} reduces every adaptive transcript
to at most \(s\) covariance observations.
As \(n\to\infty\), the actual approximation errors concentrate
around the separated ratios in
\eqref{eq:hard_family_separation}, so the probability of violating
the intended decision regions vanishes.
Appendices~\ref{app:width:adaptive-section}--\ref{app:width:fano-section}
provide the transcript, leakage, and decoding arguments.

Taking \(M=W\) singleton paths gives the width term.
The single-approximation lower bound supplies
\(g^{-2}\log(1/\delta)\) and handles bounded \(W\);
the maximum of these two lower bounds is equivalent to their sum.
The uniform treatment of confidence levels and integer widths
is given in Appendix~\ref{app:width:confidence-section}.

Finally, the transformation \(B=AK^{1/2}\) from
Corollary~\ref{cor:output_certification} preserves the residual
ordering and implements each probe with one forward MVP.
The reuse upper bounds therefore extend to mean-square output
error for every given \(K\).
The corresponding minimax lower bounds hold when the worst case
also ranges over known second moments, since \(K=I_n\) recovers
the Frobenius problem; see Appendix~\ref{app:reuse:mean-square}.

\section{Certification Delay in Iterative Approximation}
\label{sec:certification_delay}

We now combine the cost of validation with the construction performed
while waiting for acceptance.
A smaller validation batch costs less, but may require further rank
expansions before there is enough evidence to stop.
The resulting trade-off depends on both the certification rule
and the spectrum of the operator.
The algorithm cannot observe the true error directly; it must certify
it from finitely many random probes. To control false acceptance,
the acceptance threshold must lie below the target, so reaching the
target need not produce immediate acceptance.

Fix a nested expansion path constructed independently of validation,
and reuse the same \(s\) probes throughout.
Let \(r_\star=\min\{r:q(P_r)\le\rho\}\) be the first position reaching
the true target, and let \(\tau_s\) be the first position accepted
by the certificate.
Define the additional construction and its combined cost by
\(
D_s:=(\tau_s-r_\star)_+,
\mathcal C(s):=c_{\mathrm v}s+c_{\mathrm b}\mathbb E D_s,
\)
where \(c_{\mathrm v},c_{\mathrm b}>0\) are fixed costs per validation
query and construction step.
Validation is paid for once, while construction before \(r_\star\)
is excluded.
Early acceptance contributes zero delay and is controlled separately
by the probability of false stopping.

\subsection{Uniform safety and spectrum-dependent power}
\label{sec:spectral_power}

The log-ratio bound from Section~\ref{sec:single_certification}
gives the energy-ratio stopping rule
\[
\ell(\widehat q)\le\ell(\rho)-u_s(\delta),
~~
u_s(\delta)=2\operatorname{arcosh}(\delta^{-1/s}).
\]
For every approximation with \(q\ge\rho\), the probability of
acceptance is at most \(\delta\).
The boundary reduction from Section~\ref{sec:nested_validation}
extends this guarantee to false stopping anywhere along the path.
For fixed \(\rho,\delta\), the corresponding threshold
\(\theta_s=\ell^{-1}(\ell(\rho)-u_s(\delta))\) satisfies
\(\rho-\theta_s=\Theta(s^{-1/2})\) as \(s\to\infty\).
This margin is calibrated uniformly over spectra.

To examine how another statistic responds to spectral dispersion,
consider a block-sign certificate.
Set \(H=E-\rho T\), partition \(s=Lb\) probes into \(L\)
predetermined blocks of size \(b\ge2\), and compute
\[
X_\ell(H)=\frac1b\sum_{i\in\mathcal I_\ell}g_i^\top Hg_i,
~~
V_L(H)=\sum_{\ell=1}^L\mathbf1\{X_\ell(H)\ge0\}.
\]
These statistics use the cached responses because
\(g_i^\top Hg_i=\|(I-P)Ag_i\|_2^2-\rho\|Ag_i\|_2^2\).
A negative block statistic supports acceptance; the rule stops
when sufficiently few blocks remain nonnegative.
Specifically, define
\[
\begin{aligned}
k_\delta(L)
&=
\max\left\{
k\in\{0,\ldots,L-1\}:\right.\\
&\qquad\left.
\mathbb P\{\operatorname{Bin}(L,e^{-1})\le k\}\le\delta
\right\},
\end{aligned}
\]
and accept when \(V_L(H)\le k_\delta(L)\).
If the defining set is empty, set \(k_\delta(L)=-1\), so the
rule always rejects.

\begin{proposition}[Block-sign safety]
\label{prop:block_sign}
For \(b\ge2\), fixed \(0\preceq E\preceq T\), and
\(\operatorname{tr}T>0\), the block-sign rule accepts with
probability at most \(\delta\) whenever \(q\ge\rho\).
The same bound controls acceptance at any such position along
a nested path constructed independently of validation.
\end{proposition}

\begin{proof}
When \(q\ge\rho\), we have \(\operatorname{tr}H\ge0\).
Lemma~\ref{app:stopping:quadratic-sign} establishes
\(\mathbb P\{X_\ell(H)\ge0\}\ge e^{-1}\) through a log-concavity
argument.
The independent blocks therefore give
\(V_L(H)\sim\operatorname{Bin}(L,p)\) for some \(p\ge e^{-1}\),
and the definition of \(k_\delta(L)\) bounds acceptance by
\(\delta\).
Along a nested path, every block statistic and hence \(V_L\)
is nonincreasing, so the unsafe-prefix boundary argument applies.
Appendix~\ref{app:stopping:sign-calibration} gives the complete
sign-probability and pathwise arguments.
\end{proof}

The cutoff depends only on \(L,\delta\), whereas the probability
of acceptance below the target depends on the fluctuations of
the block statistics.
Appendix~\ref{app:stopping:power} quantifies this dependence through
their variance relative to their squared mean.
Thus, spectral structure affects power without being an input
to the calibration.

\subsection{A separation on the same spectral family}
\label{sec:delay_separation}

Consider the family
\[
A_{M,N}=
\operatorname{diag}\left(
\sqrt{\frac{1-\rho}{M}}\,I_M,\,
\sqrt{\frac{\rho}{N}}\,I_N
\right).
\]
A coordinate expansion path first captures the initial \(M\)
directions and then adds the remaining directions one at a time.
Consequently, \(r_\star=M\) and
\(q(P_{M+j})=\rho(1-j/N)\).
Both certificates use this same operator, path, and reliability
level.
When \(k_\delta(L)\ge0\), both accept almost surely at full rank,
so their stopping costs are finite.

\begin{theorem}[Diffuse-spectrum delay and allocation]
\label{thm:diffuse_delay}
Fix \(\rho\in(0,1)\), \(\delta\in(0,1/2)\), and an integer \(L\)
such that \(k=k_\delta(L)\ge0\).
Suppose \(M/N\to\kappa\in(0,\infty)\) as \(N\to\infty\).
For \(s=Lb\to\infty\), with integer \(b\ge2\) and \(b=o(N)\),
\begin{equation}
\mathbb E D_s^{\mathrm{sign}}
\sim C_{\mathrm{sign}}\sqrt{N/s},
~~
\mathbb E D_s^{\mathrm{ratio}}
\sim C_{\mathrm{ratio}}N/\sqrt{s},
\label{eq:diffuse_delay_asymptotics}
\end{equation}
where the positive constants depend only on the fixed parameters.
For fixed positive cost weights \(c_{\mathrm v},c_{\mathrm b}\),
\(
\min_{\substack{b\in\mathbb N, b\ge2}}
\mathcal C_N^{\mathrm{sign}}(Lb)
=\Theta(N^{1/3}),~
\min_{\substack{s\in\mathbb N, s\ge1}}
\mathcal C_N^{\mathrm{ratio}}(s)
=\Theta(N^{2/3}).
\)
The optimal budgets have the same respective orders.
Optimization is within the two specified rule families,
keeping \(L,k\) fixed for block signs and excluding construction
of the initial \(M\) directions.
\end{theorem}

The delay scales follow by comparing the initial fluctuations
with the error decrease produced by each additional direction.
Let \(X_\ell(j)\) denote a block statistic after \(j\) tail
directions have been captured.
When \(M/N\to\kappa\in(0,\infty)\),
\(
\operatorname{sd}(X_\ell(0))
\asymp \frac1{\sqrt{bN}},~
\mathbb E X_\ell(j)=-\frac{\rho j}{N}.
\)
The accumulated negative drift reaches the initial fluctuation
scale after order \(\sqrt{N/b}\) steps.
For fixed \(L,k\), acceptance requires a fixed number of blocks
to cross zero, changing the constant but preserving this scale.
Since \(s=Lb\), the resulting delay is of order \(\sqrt{N/s}\).

For the energy-ratio rule, the threshold remains order
\(s^{-1/2}\) below \(\rho\), while the estimation fluctuations
on this family have the smaller scale \((sN)^{-1/2}\).
Each additional direction reduces the true error by \(\rho/N\),
so covering the calibrated margin requires order \(N/\sqrt{s}\)
further steps.
Propositions~\ref{app:stopping:sign-limit}
and~\ref{app:stopping:ratio-limit} establish these crossing limits,
prove convergence of the expectations, give the exact constants
in \eqref{eq:diffuse_delay_asymptotics}.

Within these asymptotic regimes, the combined costs therefore
have the balances
\[
\mathcal C_N^{\mathrm{sign}}(s)
\asymp s+\sqrt{N/s},~~
\mathcal C_N^{\mathrm{ratio}}(s)
\asymp s+\frac{N}{\sqrt{s}},
\]
up to fixed positive coefficients.
Balancing validation against additional construction produces
the respective \(N^{1/3}\) and \(N^{2/3}\) budgets and costs.
Appendix~\ref{app:stopping:allocation-proof} handles the remaining
budget regimes and proves global optimality over the allowed
integer budgets.

For mean-square output error, both rules use the transformed
responses \(A(K^{1/2}g_i)\) from
Corollary~\ref{cor:output_certification}, retaining the same
safety guarantees.
The delay comparison applies when the effective operator
\(B=AK^{1/2}\) and its expansion path have the spectral structure
above; see Appendix~\ref{app:stopping:functional}.
\section{Numerical Experiments}
\label{sec:experiments}

The experiments examine three consequences of the theory:
the single-candidate confidence exponent, the width-dependent
cost of shared validation, and the effect of spectral structure
on stopping cost.
Synthetic operators provide known errors and controlled candidate
families. Gaussian response statistics are sampled from their exact
finite-dimensional distributions.
Full configurations, sampling checks, and additional results
are provided in the appendix.

\begin{figure*}[t]
    \centering
    \includegraphics[width=0.8\linewidth]{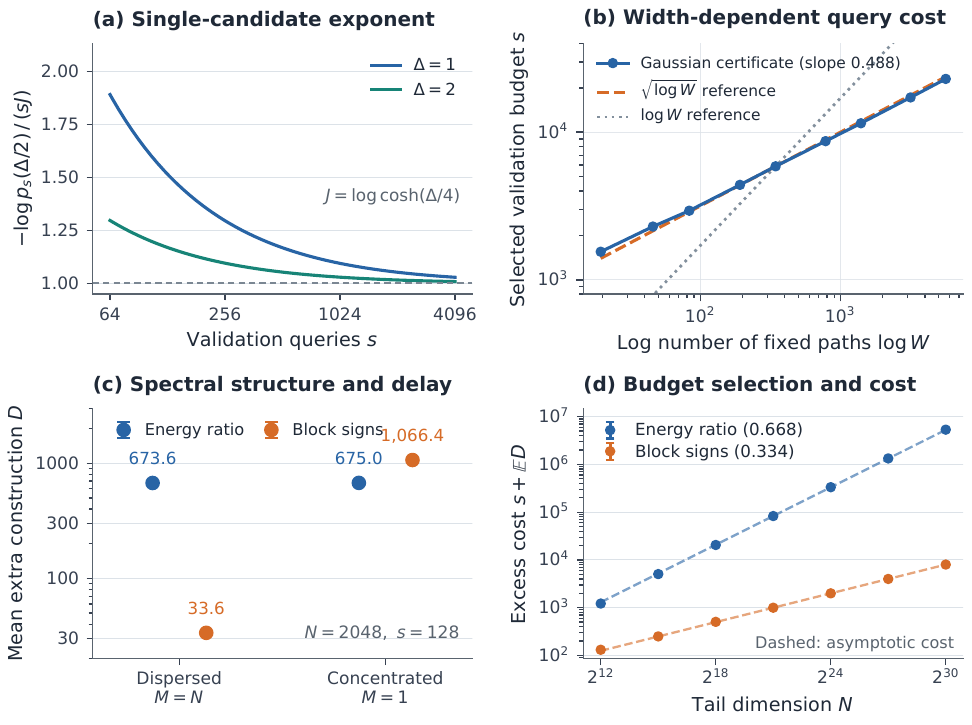}
    \caption{\small
    Numerical illustrations of the certification theory.
    (a) The normalized F-tail exponent approaches one.
    (b) Selected Gaussian validation budgets versus $\log W$;
    dashed and dotted curves indicate square-root logarithmic
    and logarithmic reference slopes, respectively.
    (c) Spectral dispersion changes the relative stopping delays
    of the two certificates.
    (d) Independently evaluated costs at selected validation budgets;
    dashed curves show the asymptotic minimum costs within
    the two rule families.
    Legend exponents in (b,d) are descriptive log--log fits.
    Error bars in (c,d) indicate $\pm2$ Monte Carlo standard errors;
    those in (d) are conditional on the selected budgets.
    }
    \label{fig:certification_experiments}
\end{figure*}

\paragraph{Single-candidate confidence exponent.}
To examine the finite-query behavior of
Theorem~\ref{thm:single_certification}, we numerically evaluate
the F-tail lower bound $p_s(\Delta/2)$ for $\Delta\in\{1,2\}$.
Figure~\ref{fig:certification_experiments}(a) plots
$-\log p_s(\Delta/2)/(sJ)$, where $J=\log\cosh(\Delta/4)$.
At $s=4096$, the normalized exponents are $1.029$ and $1.009$,
respectively. Their approach to one illustrates how the lower
and upper bounds acquire the same leading confidence exponent.

\paragraph{Shared validation and the width mechanism.}
To examine the $\sqrt{\log W}$ contribution in
Theorem~\ref{thm:simultaneous_certification}, we construct fixed
families of noncommuting residual energy matrices sharing one
positive-semidefinite envelope.
The effective dimension ranges from $d=8$ to $128$, with
$W=2^{d(d-1)/2}$ ordered paths.
Each path contains candidates at both decision boundaries,
$\rho=0.5$ and $\rho-g=0.49$.
Increasing dimension together with width avoids the saturation
inherent in a fixed-dimensional experiment.
The construction permits the maximum validation error over the
entire finite family to be computed without enumeration;
its attainment is checked for every simulated sketch.

Budgets are selected by screening toward a simultaneous error
rate of $0.1$, followed by $8192$ fresh evaluation trials.
Over the largest five dimensions, the selected budgets have
log--log slope $0.488$ against $\log W$, close to the predicted
$1/2$ (Figure~\ref{fig:certification_experiments}(b)).
At fixed $s=4096$, the mean maximum error has slope $0.256$,
consistent with the underlying
$(\log W)^{1/4}/\sqrt{s}$ scale.
Independent evaluation rates range from $2.32\%$ to $10.72\%$;
these are screening-selected budgets, not exact minimum budgets
at a common error probability.
The results illustrate the width mechanism for Gaussian validation;
the lower bound for arbitrary adaptive queries is established
by theory.
The appendix also reports repeated-check experiments in which
retaining both decision boundaries makes the failure event
unchanged when intermediate checks are added.

\paragraph{Spectral structure and stopping cost.}
We compare the energy-ratio and block-sign certificates on the
diagonal family from Section~\ref{sec:certification_delay},
using $\rho=0.1$, $\delta=0.05$, and $L=32$ sign blocks.
Both certificates follow the same coordinate expansion path.
At $N=2048$ and $s=128$, dispersing the captured energy over
$M=N$ directions yields mean additional construction costs
of $673.6$ steps for the ratio certificate and $33.6$ for block signs
(Figure~\ref{fig:certification_experiments}(c)).
Concentrating that energy in one direction ($M=1$) reverses
the ordering, demonstrating the instance dependence of the gain.

To assess whether the separation persists after budget selection,
we vary $M=N$ from $2^{12}$ to $2^{30}$.
Each rule selects its budget from the same prespecified geometric
grid using screening samples, followed by $8192$ fresh evaluation
trials at the selected budget.
The cost is $s+\mathbb{E}D$, excluding construction before
the true target.
Over the largest four dimensions, the fitted cost exponents are
$0.334$ for block signs and $0.668$ for the ratio certificate
(Figure~\ref{fig:certification_experiments}(d)).
These grid-selected costs agree with the predicted
$N^{1/3}$ and $N^{2/3}$ growth within the two specified rule families.

\section{Conclusion and Limitations}
We characterize the information needed to certify relative error
in iterative low-rank approximation. The results identify an exact
single-candidate confidence constant, an optimal width dependence
for shared validation, and a separation between reliable decisions
and prompt stopping. The guarantees require candidate families
independent of validation, and query complexity excludes the
arithmetic cost of evaluating candidates. The stopping-cost
separation concerns excess work within two specified rule families,
rather than total runtime or optimality over all certificates.
Extending these guarantees to candidates constructed using the
validation responses remains an open direction.

\section*{AI Use Statement}
Generative AI tools were used to check for typographical errors and to provide an auxiliary review for reference. They were also used to draft and check the section \textit{SAMPLING CHECKS AND REPRODUCIBILITY} in the Appendix. The authors take full responsibility for the content.

\bibliography{references}
\bibliographystyle{icml2026}
\clearpage

\appendix
\clearpage
\onecolumn
\numberwithin{equation}{section}

\section{Proofs for Single-Approximation Certification}
\label{app:single}

This appendix proves Theorem~\ref{thm:single_certification} and
Corollary~\ref{cor:output_certification}. We give the oracle model explicitly,
prove the joint Gaussian bound used by the certificate, and then establish
the matching information bound. The Gaussian innovation argument below
allows an arbitrary positive definite output covariance, so it also
applies to the multiple-hypothesis construction for simultaneous
certification.

\subsection{Oracle model and dimension-uniform risk}
\label{app:single:model}

Fix \(0<g<\rho<1\). The unknown matrix
\(A\in\mathbb R^{m\times n}\) is nonzero, and the supplied matrix
\(P\in\mathbb R^{m\times m}\) is an orthogonal projection:
\(P=P^\top=P^2\). Write
\(q(A,P)=\|(I_m-P)A\|_F^2/\|A\|_F^2\), and define
\[
\mathcal H_0(P)=\{A\ne0:q(A,P)\ge\rho\},
~~
\mathcal H_1(P)=\{A\ne0:q(A,P)\le\rho-g\}.
\]
The decision \(1\) means acceptance and the decision \(0\) means
rejection. Neither decision is prescribed outside
\(\mathcal H_0(P)\cup\mathcal H_1(P)\).

A procedure receives \((m,n,P,\rho,g)\), has private randomness
independent of \(A\), and observes \(A\) only through calls
\(v\mapsto Av\). Each call returns the full output vector. Query
directions, stopping decisions, and the final decision are measurable
functions of the supplied information, private randomness, and preceding
query--response pairs. The procedure makes at most \(s\) calls on every
run; equivalently, a procedure that stops earlier can be padded with
unused zero queries. All oracle calls made by the procedure count,
including any construction responses supplied for use in its decision.
There is no additional free transcript depending on \(A\).

For a procedure \(\mathcal M_{m,n,P}\), define
\[
\begin{aligned}
\alpha(\mathcal M_{m,n,P})
&=\sup_{A\in\mathcal H_0(P)}
  \mathbb P_A\{\mathcal M_{m,n,P}=1\},\\
\beta(\mathcal M_{m,n,P})
&=\sup_{A\in\mathcal H_1(P)}
  \mathbb P_A\{\mathcal M_{m,n,P}=0\}.
\end{aligned}
\]
Here the probability concerns the procedure's randomness for a fixed
matrix, and a supremum over an empty hypothesis class is defined as
zero. A family of procedures may depend on the known dimensions and
projection. The dimension-uniform minimax risk is
\begin{equation}
R_s(\rho,g)
=\inf_{\{\mathcal M_{m,n,P}\}}
  \sup_{\substack{m,n\ge1\\P=P^\top=P^2}}
  \max\{\alpha(\mathcal M_{m,n,P}),
          \beta(\mathcal M_{m,n,P})\}.
\label{app:single:risk}
\end{equation}
All procedures in the infimum obey the same budget \(s\). The optimal
budget is
\(s^\star(\delta;\rho,g)=\min\{s\ge1:R_s(\rho,g)\le\delta\}\).

For upper bounds, \(P\) may instead result from an earlier construction
stage: conditional on that stage, the proof applies to fresh independent
validation probes. This conditional statement counts additional
validation calls. It does not give a lower bound on additional calls
when arbitrary construction observations have been supplied for free.

\subsection{A joint exponential moment bound}
\label{app:single:joint_moment}

For independent \(g_1,\ldots,g_s\sim\mathcal N(0,I_n)\), set
\[
Z(M)=\frac1s\sum_{i=1}^s g_i^\top M g_i.
\]
All matrices in a comparison use these same probes. If \(M\succeq0\)
and \(\operatorname{tr}(M)>0\), then \(Z(M)>0\) almost surely.

\begin{lemma}[An exponential moment determined by two traces]
\label{app:single:moment}
Let \(U,V\succeq0\), with
\(\operatorname{tr}(U)=a>0\) and \(\operatorname{tr}(V)=b>0\).
For \(z\sim\mathcal N(0,I_n)\) and \(0\le t<1/(2b)\),
\begin{equation*}
\mathbb E\exp\{-t z^\top(U-V)z\}
\le [(1+2ta)(1-2tb)]^{-1/2}.
\end{equation*}
When \(n\ge2\), equality is attained by
\(U=a uu^\top\) and \(V=b vv^\top\) for orthogonal unit vectors
\(u,v\).
\end{lemma}

\begin{proof}
Let \(D=U-V\), and denote its positive eigenvalues by
\(\lambda_i\) and its negative eigenvalues by \(-\mu_j\), with
\(\lambda_i,\mu_j>0\). Define
\[
p=\sum_i\lambda_i,~~ r=\sum_j\mu_j.
\]
Zero eigenvalues are omitted, and an empty sum is zero. Taking the trace
gives \(p-r=a-b\). If \(\Pi_-\) is the negative spectral projection
of \(D\), then
\[
r=-\operatorname{tr}(\Pi_-D)
 =\operatorname{tr}(\Pi_-V)-\operatorname{tr}(\Pi_-U)
\le\operatorname{tr}(V)=b.
\]
Indeed, the trace of a product of positive semidefinite matrices is
nonnegative, and \(0\preceq\Pi_-\preceq I_n\). The positive spectral
projection gives \(p\le a\) in the same way. Thus \(pr\le ab\), and
every factor \(1-2t\mu_j\) is positive.

For nonnegative \(x_i\), expansion gives
\(\prod_i(1+x_i)\ge1+\sum_i x_i\). For nonnegative \(y_j\) with
\(\sum_jy_j<1\), induction gives
\(\prod_j(1-y_j)\ge1-\sum_jy_j\). Applying the two inequalities
separately yields
\begin{align}
\det(I_n+2tD)
&=\prod_i(1+2t\lambda_i)\prod_j(1-2t\mu_j)\notag\\
&\ge(1+2tp)(1-2tr)\notag\\
&=1+2t(a-b)-4t^2pr\notag\\
&\ge1+2t(a-b)-4t^2ab
 =(1+2ta)(1-2tb).
\label{app:single:determinant}
\end{align}
In particular, \(I_n+2tD\) is positive definite. Orthogonal
diagonalization preserves independent standard Gaussian coordinates.
The one-dimensional Gaussian integral
\(\mathbb E e^{-t\lambda Z^2}=(1+2t\lambda)^{-1/2}\), valid when
\(1+2t\lambda>0\), therefore gives
\[
\mathbb E e^{-t z^\top Dz}=\det(I_n+2tD)^{-1/2}.
\]
Combining this identity with \eqref{app:single:determinant} proves
the bound. For the stated orthogonal rank-one pair, the only nonzero
eigenvalues are \(a\) and \(-b\), so every inequality is an equality.
\end{proof}

\begin{lemma}[Two-sided log-ratio concentration]
\label{app:single:ratio_tail}
Let \(B,C\succeq0\) have positive traces. Define
\[
L_s=\log\frac{Z(B)}{Z(C)}
    -\log\frac{\operatorname{tr}(B)}{\operatorname{tr}(C)},
~~
I(u)=\log\cosh(u/2).
\]
For every \(u>0\),
\begin{equation}
\mathbb P\{L_s\le-u\}\le e^{-sI(u)},
~~
\mathbb P\{L_s\ge u\}\le e^{-sI(u)}.
\label{app:single:ratio_tails}
\end{equation}
No independence between \(Z(B)\) and \(Z(C)\) is required.
\end{lemma}

\begin{proof}
Write \(B_0=B/\operatorname{tr}(B)\),
\(C_0=C/\operatorname{tr}(C)\), and \(r=e^{-u}\in(0,1)\).
The lower-tail event is
\(\sum_{i=1}^s g_i^\top(B_0-rC_0)g_i\le0\).
For \(0<t<1/(2r)\), exponential Markov and independence across probes
give
\[
\mathbb P\{L_s\le-u\}
\le\prod_{i=1}^s
 \mathbb E e^{-t g_i^\top(B_0-rC_0)g_i}
\le[(1+2t)(1-2tr)]^{-s/2},
\]
where the second step uses Lemma~\ref{app:single:moment}. The
quadratic product is maximized at
\(t=(1-r)/(4r)=(e^u-1)/4\), which belongs to the allowed interval.
At this value,
\[
(1+2t)(1-2tr)=\frac{(1+r)^2}{4r}
             =\cosh^2(u/2).
\]
This proves the lower-tail bound. Interchanging \(B\) and \(C\)
changes \(L_s\) into \(-L_s\) and proves the other bound.
\end{proof}

\subsection{The Gaussian certificate}
\label{app:single:certificate}

For the supplied projection, set
\[
T=A^\top A,~~
E=A^\top(I_m-P)A,~~
C=A^\top PA=T-E.
\]
Orthogonality of \(P\) implies \(E,C\succeq0\). Cache
\(Y=[Ag_1,\ldots,Ag_s]\), so that
\[
\widehat q
=\frac{\|(I_m-P)Y\|_F^2}{\|Y\|_F^2}
=\frac{Z(E)}{Z(T)},~~
q=\frac{\operatorname{tr}(E)}{\operatorname{tr}(T)}.
\]
Although the residual and captured output vectors are orthogonal for
each probe, their squared norms need not be statistically independent.
Their common Gaussian inputs are handled by
Lemma~\ref{app:single:ratio_tail}.

For \(q\in(0,1)\), both \(E\) and \(C\) have positive trace, and
\[
\ell(\widehat q)-\ell(q)
=\log\frac{Z(E)}{Z(C)}
 -\log\frac{\operatorname{tr}(E)}{\operatorname{tr}(C)},
~~
\ell(x)=\log\frac{x}{1-x}.
\]
Consequently, both tails of this log-odds difference satisfy
\eqref{app:single:ratio_tails}. If \(q=0\), positive semidefiniteness
gives \(E=0\), and \(\widehat q=0\) almost surely. If \(q=1\), then
\(C=0\), and \(\widehat q=1\) almost surely.

Let
\[
\Delta=\ell(\rho)-\ell(\rho-g)>0,~~
\mu=\frac{\ell(\rho)+\ell(\rho-g)}2.
\]
The Gaussian certificate accepts exactly when
\begin{equation}
\sum_{i=1}^s\|(I_m-P)Ag_i\|_2^2
\le e^\mu\sum_{i=1}^s\|PAg_i\|_2^2.
\label{app:single:midpoint_rule}
\end{equation}
This formulation also covers \(q=0,1\) without taking logarithms
at an endpoint. For an interior null instance \(q\ge\rho\), acceptance
implies
\(\ell(\widehat q)-\ell(q)\le-\Delta/2\).
For an interior alternative \(q\le\rho-g\), rejection implies
\(\ell(\widehat q)-\ell(q)>\Delta/2\). Hence both errors are at most
\[
e^{-sI(\Delta/2)}
=\exp[-s\log\cosh(\Delta/4)].
\]
The endpoint instances are classified correctly almost surely. This
proves the upper bound uniformly over the dimensions, matrix, and
projection.

\subsection{The exact F tail and its prefactor}
\label{app:single:f_tail}

Let \(F_{s,s}\) be the ratio of two independent \(\chi_s^2\) random
variables, and write \(p_s(u)=\mathbb P\{F_{s,s}\le e^{-u}\}\) for
\(u>0\). The division by the common degrees of freedom cancels, so
this is also the usual \(F\) distribution with degrees of freedom
\((s,s)\).

\begin{lemma}[Exact density and endpoint asymptotics]
\label{app:single:f_asymptotics}
The random variable \(\log F_{s,s}\) has density
\begin{equation}
f_s(z)=c_s e^{-sI(z)},~~
c_s=\frac{\Gamma((s+1)/2)}
          {2\sqrt{\pi}\,\Gamma(s/2)},~~
I(z)=\log\cosh(z/2).
\label{app:single:logf_density}
\end{equation}
In particular,
\begin{equation}
p_s(u)=c_s\int_u^\infty e^{-sI(t)}\,dt.
\label{app:single:f_integral}
\end{equation}
For every integer \(s\ge1\) and \(u>0\),
\begin{equation}
\frac{2c_s}{s}e^{-sI(u)}
\le p_s(u)
\le\frac{2c_s}{s\tanh(u/2)}e^{-sI(u)}.
\label{app:single:f_bounds}
\end{equation}
For each fixed \(u>0\),
\begin{equation}
p_s(u)\sim
\frac{e^{-sI(u)}}{\sqrt{2\pi s}\,\tanh(u/2)}
~\text{as }s\to\infty.
\label{app:single:f_prefactor}
\end{equation}
\end{lemma}

\begin{proof}
Let \(X,W\) be independent \(\chi_s^2\) variables. Substitute
\(X=xw\), \(W=w\) into their joint gamma density. The Jacobian is
\(w\), and integration over \(w>0\) gives
\[
f_F(x)
=\frac{\Gamma(s)}{\Gamma(s/2)^2}
  \frac{x^{s/2-1}}{(1+x)^s},~~ x>0.
\]
With \(x=e^z\), the density of the logarithm is therefore
\[
f_s(z)
=\frac{\Gamma(s)}{2^s\Gamma(s/2)^2}\cosh(z/2)^{-s}.
\]
The duplication identity
\(\Gamma(s)=2^{s-1}\Gamma(s/2)\Gamma((s+1)/2)/\sqrt{\pi}\)
gives \eqref{app:single:logf_density}. Its symmetry implies
\(F_{s,s}^{-1}\overset d=F_{s,s}\) and
\eqref{app:single:f_integral}. The density is positive and continuous,
so \(p_s(u)>0\) and \(p_s\) is continuous for every finite \(u>0\).

For \(t\ge u>0\),
\(I'(t)=\tfrac12\tanh(t/2)\) is positive, nondecreasing, and at most
\(1/2\). Thus
\[
I(u)+I'(u)(t-u)\le I(t)\le I(u)+\tfrac12(t-u).
\]
Integrating the two exponential bounds proves
\eqref{app:single:f_bounds}.

For the sharp prefactor, set \(t=u+y/s\) in
\eqref{app:single:f_integral}:
\[
p_s(u)
=\frac{c_s}{s}e^{-sI(u)}
 \int_0^\infty
 \exp\{-s[I(u+y/s)-I(u)]\}\,dy.
\]
For each \(y\ge0\), the integrand tends to \(e^{-I'(u)y}\).
Convexity bounds it by this same integrable function for every \(s\).
Dominated convergence therefore gives the integral limit
\(1/I'(u)=2/\tanh(u/2)\). Stirling's formula gives
\[
\frac{\Gamma((s+1)/2)}{\Gamma(s/2)}
=\sqrt{s/2}\,[1+O(s^{-1})],
~~
c_s\sim\frac{\sqrt{s}}{2\sqrt{2\pi}}.
\]
Substitution proves \eqref{app:single:f_prefactor}.
\end{proof}

For two orthogonal rank-one energy forms, their normalized estimates
are independent \(\chi_s^2/s\) variables. Hence their log-ratio has
exactly the density above. It follows that \(I(u)\) is the sharp
exponential rate in \eqref{app:single:ratio_tails}. The general
finite-sample calibration remains the exponential bound; the exact
\(F\) tail is used below as a testing lower bound.

\subsection{Gaussian innovations under adaptive queries}
\label{app:single:adaptive}

\begin{lemma}[A covariance-independent transcript simulator]
\label{app:single:innovation}
Fix \(m,n,s\), and let \(\Sigma\in\mathbb R^{m\times m}\) be positive
definite. Suppose the columns of \(A\) are independent with law
\(\mathcal N(0,\Sigma/n)\). Consider an adaptive forward-query
procedure with at most \(s\) queries. Its supplied information is fixed
and independent of \(\Sigma\) and \(A\); its private random seed is
independent of \(A\) and has the same law for every \(\Sigma\).
Its complete transcript can be generated from the seed and \(s\)
independent vectors with law \(\mathcal N(0,\Sigma)\) through a
measurable map that does not depend on \(\Sigma\). The transcript
includes queries, responses, stopping, and the final output.
\end{lemma}

\begin{proof}
Draw all private randomness as a seed \(R\) before the first query;
it is independent of \(A\). The query rule is then deterministic
conditional on \(R\) and the observed history. After \(t\) queries
define
\[
V_t=[v_1,\ldots,v_t],~~
Y_t=[Av_1,\ldots,Av_t],~~
M_t=Y_tV_t^\dagger,~~
\Pi_t=I_n-V_tV_t^\dagger,
\]
where \(\dagger\) denotes the Moore--Penrose pseudoinverse.
The matrix \(\Pi_t\) projects onto the orthogonal complement of the
previous query span. Since \(Y_t=AV_t\), one has
\(M_t=A(I_n-\Pi_t)\), so \(M_t\) is the known action on that span.
Set \(M_0=0\), \(\Pi_0=I_n\), and let
\(\mathcal F_t=\sigma(R,V_t,Y_t)\).

We prove inductively that, under each fixed covariance, the regular
conditional law is
\begin{equation}
\mathcal L(A\mid\mathcal F_t)
=\mathcal L\left(
 M_t+\frac1{\sqrt n}\Sigma^{1/2}Z\Pi_t
 \,\middle|\,\mathcal F_t\right),
\label{app:single:conditional_law}
\end{equation}
where the entries of the auxiliary matrix \(Z\in\mathbb R^{m\times n}\)
are independent standard normals, independent of \(\mathcal F_t\).
The equality denotes conditional distributions, rather than a single
auxiliary matrix shared across different values of \(t\). At \(t=0\),
it is the assumed matrix prior.

Conditional on \(\mathcal F_t\), the next query \(v\) is fixed.
If \(d=\|\Pi_tv\|_2=0\), then \(Av=M_tv\) is already determined.
This query supplies no additional condition on the residual law,
and \eqref{app:single:conditional_law} remains valid.

Otherwise put \(u=\Pi_tv/d\). Under the conditional representation,
\[
A v=M_tv+\frac d{\sqrt n}\xi,
~~
\xi=\Sigma^{1/2}Zu\sim\mathcal N(0,\Sigma).
\]
The vector \(\xi\) and the matrix
\(\Sigma^{1/2}Z(\Pi_t-uu^\top)\) are jointly Gaussian and independent.
To verify this, for every vector \(w\) in the range of
\(\Pi_t-uu^\top\), their component covariances are determined by
\(\operatorname{Cov}(\Sigma^{1/2}Zu,\Sigma^{1/2}Zw)
=\Sigma(u^\top w)=0\).
Thus observing the new response reveals \(\xi\) and leaves the
orthogonal Gaussian component unchanged. The posterior mean and
residual projection become
\begin{equation}
M_{t+1}=M_t+\frac1{\sqrt n}\xi u^\top,
~~
\Pi_{t+1}=\Pi_t-uu^\top.
\label{app:single:innovation_update}
\end{equation}
These are precisely the known action on the enlarged query span and
its orthogonal complement. The independent remaining Gaussian
component proves \eqref{app:single:conditional_law} at step \(t+1\).
Adaptive selection imposes no additional constraint here: conditional
on the current history and seed, the selected query is a known vector.

A simulator now draws \(\xi_1,\ldots,\xi_s\) independently from
\(\mathcal N(0,\Sigma)\), together with \(R\). It follows the
procedure's query rule, returns \(M_tv\) for a redundant query, and
for a new direction returns
\(M_tv+d\xi_j/\sqrt n\), using the next unused vector and updating
the known action by \eqref{app:single:innovation_update}. Early
stopping leaves unused vectors. The induction proves equality in
law of the full transcripts. Every step of the simulation uses only
the supplied information, the seed, the generated history, and the
input vectors \(\xi_j\); it never uses \(\Sigma\).
\end{proof}

In particular, if \(\Sigma\) is one of finitely many unknown
covariances, this lemma gives a single hypothesis-independent
simulation kernel. Any decision from the transcript is consequently
a decision based on at most \(s\) independent covariance observations
and an independent seed. It cannot have smaller Bayes error than the
optimal test based on those observations. This is the full-vector
information reduction underlying \eqref{eq:gaussian_innovation};
it does not restrict the original procedure to scalar energy
statistics.

\subsection{Two covariance hypotheses and finite-dimensional leakage}
\label{app:single:lower}

\begin{proposition}[Finite-dimensional lower bound]
\label{app:single:finite_lower}
Fix \(m=2\), \(P=\operatorname{diag}(0,1)\), and integers \(n\ge s\ge1\).
Every procedure using at most \(s\) adaptive queries satisfies, for
each \(\eta>0\),
\begin{equation}
\max\{\alpha(\mathcal M_{2,n,P}),\beta(\mathcal M_{2,n,P})\}
\ge
\left[
\frac{p_s(\Delta/2+\eta)-p_n(\eta)}
     {1-p_n(\eta)}
\right]_+,
\label{app:single:finite_lower_bound}
\end{equation}
where \([x]_+=\max\{x,0\}\) and
\(\Delta=\ell(\rho)-\ell(\rho-g)\).
\end{proposition}

\begin{proof}
Set
\[
\lambda_0=\ell(\rho)+\eta,~~
\lambda_1=\ell(\rho-g)-\eta,~~
\Sigma_h=\operatorname{diag}
           (e^{\lambda_h/2},e^{-\lambda_h/2}).
\]
Under prior \(\mathcal Q_h\), draw the columns of \(A\) independently
from \(\mathcal N(0,\Sigma_h/n)\). The supplied dimensions and
projection are identical under the two priors. Applying
Lemma~\ref{app:single:innovation} under both hypotheses gives the
same simulation map from \(s\) independent vectors
\(\xi_i\sim\mathcal N(0,\Sigma_h)\).

We first compute the optimal equal-prior error for these vectors.
Write
\[
\mu=\frac{\lambda_0+\lambda_1}{2},~~
\omega=\frac{\lambda_0-\lambda_1}{2}=\Delta/2+\eta,~~
S_1=\sum_{i=1}^s\xi_{i1}^2,~~
S_2=\sum_{i=1}^s\xi_{i2}^2.
\]
Both covariances have determinant one. If \(f_h\) is the joint
density of the observations under hypothesis \(h\), direct
substitution into the Gaussian densities gives
\begin{align*}
\log\frac{f_0}{f_1}
&=-\frac12\left[
(e^{-\lambda_0/2}-e^{-\lambda_1/2})S_1
+(e^{\lambda_0/2}-e^{\lambda_1/2})S_2
\right]\\
&=\sinh(\omega/2)
  \left(e^{-\mu/2}S_1-e^{\mu/2}S_2\right).
\end{align*}
Under equal prior probabilities, selecting the larger density
minimizes the average error pointwise. The test therefore chooses
hypothesis \(1\), corresponding to acceptance, when
\(S_1/S_2\le e^\mu\). Ties have probability zero. Since the
coordinates are independent under either covariance,
\[
\frac{S_1}{S_2}\overset d=e^{\lambda_h}F_{s,s}.
\]
The error under hypothesis \(0\) is
\(\mathbb P\{F_{s,s}\le e^{-\omega}\}=p_s(\omega)\).
The error under hypothesis \(1\) is
\(\mathbb P\{F_{s,s}>e^\omega\}=p_s(\omega)\), by reciprocal
symmetry and continuity. Thus the optimal average error equals
\begin{equation}
p_s(\Delta/2+\eta).
\label{app:single:bayes_error}
\end{equation}
An independent private seed cannot improve this value: conditional
on each seed, the same pointwise density comparison remains optimal.
The transcript simulation therefore bounds the prior-averaged error
of every oracle procedure below by
\eqref{app:single:bayes_error}.

The Gaussian matrix priors have not been conditioned to lie in the
certification regions. Under \(\mathcal Q_h\), their realized error
odds satisfy
\[
\frac{q(A,P)}{1-q(A,P)}
=\frac{\|A_{1,:}\|_2^2}{\|A_{2,:}\|_2^2}
\overset d=e^{\lambda_h}F_{n,n}.
\]
Consequently, the two leakage probabilities are exactly
\begin{equation}
\mathcal Q_0\{q<\rho\}=p_n(\eta),
~~
\mathcal Q_1\{q>\rho-g\}=p_n(\eta).
\label{app:single:leakage}
\end{equation}
Let \(\varepsilon=\max\{\alpha,\beta\}\) be the procedure's
worst error on the required regions for this dimension and
projection. Its error for each fixed matrix is at most
\(\varepsilon\) inside the intended region and at most one
outside. Each prior-averaged error is therefore at most
\((1-p_n(\eta))\varepsilon+p_n(\eta)\). Combining this with
\eqref{app:single:bayes_error}, rearranging, and using
\(\varepsilon\ge0\) proves \eqref{app:single:finite_lower_bound}.
The denominator is positive because \(p_n(\eta)<1/2\).
\end{proof}

Keeping the priors Gaussian until after the simulation is essential
to this argument. Conditioning them on the certification regions
at the outset would change their Gaussian conditional laws.
Equation~\eqref{app:single:leakage} instead quantifies the discrepancy
after the information reduction has been applied.

\subsection{Minimax risk, confidence asymptotics, and the small gap}
\label{app:single:minimax}

\begin{proof}[Proof of Theorem~\ref{thm:single_certification}]
The certificate \eqref{app:single:midpoint_rule} gives the upper
bound in every dimension. For the lower bound, fix \(s\ge1\) and
\(\eta>0\). Proposition~\ref{app:single:finite_lower} applies
uniformly to every procedure in each dimension \(n\ge s\).
Lemma~\ref{app:single:f_asymptotics} gives \(p_n(\eta)\to0\) as
\(n\to\infty\). Taking the supremum over input dimensions in
\eqref{app:single:risk}, followed by the infimum over procedure
families, therefore yields
\[
R_s(\rho,g)\ge p_s(\Delta/2+\eta).
\]
Continuity of \(p_s\) permits \(\eta\downarrow0\). We obtain
\begin{equation}
p_s(\Delta/2)\le R_s(\rho,g)
\le e^{-sJ},~~
J=\log\cosh(\Delta/4)>0.
\label{app:single:sandwich}
\end{equation}

Put \(L=\log(1/\delta)\). The constructive upper bound gives
\(s^\star(\delta;\rho,g)\le\lceil L/J\rceil\).
The lower bound is positive at every fixed \(s\), so
\(s^\star(\delta;\rho,g)\to\infty\) as \(\delta\downarrow0\).
If \(R_s(\rho,g)\le\delta\), the lower estimate in
\eqref{app:single:f_bounds} implies
\[
sJ
\ge L+\log\frac{2c_s}{s}
=L-\frac12\log s-\frac12\log(2\pi)+O(s^{-1}).
\]
At \(s=s^\star\), the constructive upper bound implies
\(\log s^\star=O(\log L)=o(L)\). Dividing by \(L\) and combining
the upper and lower inequalities proves
\[
\lim_{\delta\downarrow0}
\frac{s^\star(\delta;\rho,g)}{\log(1/\delta)}
=\frac1{\log\cosh(\Delta/4)}.
\]
\end{proof}

For a fixed interior target \(\rho\), Taylor expansion gives
\[
\Delta
=\frac{g}{\rho(1-\rho)}+O_\rho(g^2),
~~
\log\cosh(\Delta/4)
=\frac{\Delta^2}{32}+O(\Delta^4).
\]
The precise small-gap consequence is the iterated limit
\begin{equation*}
\lim_{g\downarrow0}g^2
\left[
\lim_{\delta\downarrow0}
\frac{s^\star(\delta;\rho,g)}{\log(1/\delta)}
\right]
=32\rho^2(1-\rho)^2.
\end{equation*}
Thus the expression
\(32\rho^2(1-\rho)^2g^{-2}\log(1/\delta)\) describes the leading
cost when the confidence limit is taken before the gap limit.
No uniform assertion over arbitrary joint limits of \(g,\delta\)
is needed for this conclusion.

\begin{remark}[Why the dimension supremum matters]
\label{app:single:fixed_dimension}
For a fixed input dimension \(n\), querying all \(n\) coordinate
vectors recovers \(A\) exactly. With uncharged arithmetic, every
supplied projection can then be certified without statistical
error. The positive risk lower bound
\eqref{app:single:sandwich} consequently concerns a guarantee
uniform over input dimensions. In its proof, \(n\to\infty\) is
taken at fixed \(s,\eta\), and only afterward is
\(\eta\downarrow0\) taken. It is not a positive lower bound for
each fixed \(n\) once \(s\ge n\).
\end{remark}

\subsection{Mean-square output error}
\label{app:single:functional}

\begin{proof}[Proof of Corollary~\ref{cor:output_certification}]
Let \(K=\mathbb E[xx^\top]\succeq0\) be known and finite, and assume
\(\operatorname{tr}(AKA^\top)>0\). For every fixed matrix \(M\),
linearity of expectation gives
\[
\mathbb E\|Mx\|_2^2
=\mathbb E\operatorname{tr}(Mxx^\top M^\top)
=\operatorname{tr}(MKM^\top)
=\|MK^{1/2}\|_F^2.
\]
Taking \(M=(I_m-P)A\) and \(M=A\) yields, with \(B=AK^{1/2}\ne0\),
\begin{equation}
q_K(A,P)
=\frac{\mathbb E\|(I_m-P)Ax\|_2^2}
       {\mathbb E\|Ax\|_2^2}
=\frac{\|(I_m-P)B\|_F^2}{\|B\|_F^2}.
\label{app:single:functional_identity}
\end{equation}
The expectation is over the input distribution, whereas the
certification error probability is over the procedure's probes.
In particular, \eqref{app:single:functional_identity} is a ratio
of expectations, rather than a pointwise relative-error guarantee.

Every forward query to \(B\) can be implemented as
\(Bv=A(K^{1/2}v)\) with one forward query to \(A\). Computing with
the known \(K\) is part of the uncharged arithmetic in this query
model. Therefore any Frobenius certification procedure can be
simulated for the functional problem with the same hard budget.
In particular, Gaussian queries \(v=g_i\) to \(B\) become
\(K^{1/2}g_i\) queries to \(A\), proving the upper bounds for
every given \(K\).

For the minimax statement, define the functional risk by the same
suprema as in \eqref{app:single:risk}, with an additional supremum
over known positive semidefinite \(K\) satisfying the positive
denominator condition. The procedure receives \(K\). The query
simulation shows that this risk is no larger than the Frobenius
risk: apply procedures approaching its infimum to \(B\) and let
their excess risk tend to zero. Conversely, restricting the
functional instances to \(K=I_n\) recovers exactly the Frobenius
problem with the same supplied information and query model.
The functional risk is therefore no smaller than
\(R_s(\rho,g)\), proving equality and the common high-confidence
constant. No Gaussian assumption on the actual inputs \(x\)
is used.
\end{proof}

\begin{remark}[A prescribed second moment]
\label{app:single:fixed_second_moment}
The minimax equality above takes the worst case over known second
moments, not over matrices for one fixed second moment.
For example, if a prescribed \(K\) has rank \(r\), choose a known
orthonormal basis \(w_1,\ldots,w_r\) of its range. The \(r\) queries
\(Aw_1,\ldots,Aw_r\) determine \(AK^{1/2}\) exactly, and hence
determine \eqref{app:single:functional_identity} for every known
projection. A rank-one second moment can thus be certified exactly
with one query. This is consistent with the uniform result,
whose instance class contains \(K=I_n\) for arbitrarily large \(n\).
\end{remark}

\section{Proofs for Validation Reuse and the Shared-Energy Upper Bound}
\label{app:reuse-upper}

This appendix proves the pathwise reductions and the upper bound in
Theorem~\ref{thm:simultaneous_certification}. The matching information
lower bound is established in Appendix~\ref{app:reuse-lower}.
All matrices below are real. For a nonzero positive semidefinite matrix
$T$ and residual energies $0\preceq E\preceq T$, write
\begin{equation*}
Z(B)=\frac1s\sum_{j=1}^s g_j^\top B g_j,~~
q_E=\frac{\operatorname{tr}E}{\operatorname{tr}T},~~
\widehat q_E=\frac{Z(E)}{Z(T)},
~ g_j\stackrel{\mathrm{iid}}{\sim}N(0,I_n).
\end{equation*}
The definition of $Z$ applies to any symmetric matrix. Since $T\ne0$,
$Z(T)>0$ almost surely. Both $q_E$ and $\widehat q_E$ belong to $[0,1]$.
The notation $\|\cdot\|_*$ denotes the nuclear norm.

The complete candidate family and $T$ are independent of the validation
probes. They may depend on the operator and on independent construction
randomness. All statements can first be proved conditionally on these
quantities and then averaged. Throughout the pathwise arguments,
$\rho\in(0,1)$ and $g\in(0,\rho)$ are fixed before observing the sketch.
A required decision is rejection at $q_E\ge\rho$ and acceptance at
$q_E\le\rho-g$; no decision is prescribed between these regions.

\subsection{Two boundary candidates per chain}
\label{app:reuse:boundaries}

\begin{lemma}[Finite-family boundary reduction]
\label{app:reuse:boundary-lemma}
Let a finite candidate family be covered by $W$ Loewner chains, and
apply the same threshold rule $\widehat q_E\le\theta$ to every candidate,
where $\theta$ is fixed independently of the sketch. There is a set of
at most $2W$ candidates, independent of the sketch, such that every
required decision is correct if and only if every required decision
at these candidates is correct.
\end{lemma}
\begin{proof}
Fix a chain. If it contains any member with $q_E\ge\rho$, choose a
Loewner-minimal such member $E^{\mathrm{bad}}$. If it contains any member
with $q_E\le\rho-g$, choose a Loewner-maximal such member
$E^{\mathrm{good}}$. These extrema exist because the chain is finite.
They depend only on the fixed matrices and the true trace ratios.
Their identities need not be known to the procedure.

If a bad member $E$ is accepted, then
$E^{\mathrm{bad}}\preceq E$, hence
$Z(E^{\mathrm{bad}})\le Z(E)$ and
$\widehat q_{E^{\mathrm{bad}}}\le\widehat q_E\le\theta$.
Thus the bad boundary is also accepted. Conversely, an accepted bad
boundary is itself a false acceptance. Likewise, if a good member $E$
is rejected, then $E\preceq E^{\mathrm{good}}$, giving
$\widehat q_{E^{\mathrm{good}}}\ge\widehat q_E>\theta$;
the good boundary is rejected as well. The converse again holds because
the boundary is a member of the good set.
Collecting the boundaries of the $W$ chains proves the claim. Empty
bad or good sets contribute no boundary, and repeated boundaries can
be removed.
\end{proof}

On a decreasing path, these are the last unsafe member of its prefix
and the first sufficiently accurate member of its suffix. If only false
acceptance is controlled, one boundary per chain suffices. Loewner
ordering is essential to the argument: a decrease in the trace alone
does not imply a decrease in $Z(E)$ for each realized sketch.

For a countable family, choose increasing finite subsets whose union
is the whole family. Their intersections with the original chain cover
still have width at most $W$. If an error bound with a fixed budget and
threshold is uniform in the finite-subset size, it extends to the whole
family: the events that some required decision in a finite subset is
wrong increase to the event that some required decision in the full
family is wrong. Continuity of probability from below therefore gives
exactly the same bound. This argument neither requires global boundary
extrema to exist nor introduces a sum over subset sizes.

\subsection{The exact confidence constant for a fixed number of paths}
\label{app:reuse:fixed-width-confidence}

Let
\[
\ell(x)=\log\frac{x}{1-x},~~
\Delta=\ell(\rho)-\ell(\rho-g)>0,~~
\theta_{\log}
=\left[1+\exp\!\left(-\frac{\ell(\rho)+\ell(\rho-g)}2\right)\right]^{-1}.
\]
This is the midpoint in log-odds coordinates. It generally differs
from the arithmetic midpoint $\rho-g/2$ used below.

\begin{proposition}[Fixed-width high-confidence calibration]
\label{app:reuse:fixed-width-proposition}
For any finite or countable family of width at most $W$, accepting when
$\widehat q_E\le\theta_{\log}$ gives
\begin{equation}
\mathbb P\{\text{at least one required decision is incorrect}\}
\le2W\exp\{-s\log\cosh(\Delta/4)\}.
\label{app:reuse:fixed-width-risk}
\end{equation}
Consequently, for fixed $W,\rho,g$,
\begin{equation}
\lim_{\delta\downarrow0}
\frac{s^\star_{\mathrm{sim}}(\delta;\rho,g,W)}{\log(1/\delta)}
=\frac1{\log\cosh(\Delta/4)}.
\label{app:reuse:fixed-width-limit}
\end{equation}
\end{proposition}
\begin{proof}
For a fixed bad boundary with $0<q_E<1$, acceptance requires
\[
\ell(\widehat q_E)-\ell(q_E)
\le\ell(\theta_{\log})-\ell(\rho)=-\Delta/2.
\]
For a fixed good boundary with $0<q_E<1$, rejection requires
\[
\ell(\widehat q_E)-\ell(q_E)
>\ell(\theta_{\log})-\ell(\rho-g)=\Delta/2.
\]
The single-candidate Gaussian bound from
Section~\ref{sec:single_certification} controls each of these events by
$\exp\{-s\log\cosh(\Delta/4)\}$. At $q_E=0$ or $q_E=1$, the sampled
ratio equals the true ratio almost surely, so the corresponding error
probability is zero. Lemma~\ref{app:reuse:boundary-lemma} and a union
bound over at most $2W$ boundaries prove
\eqref{app:reuse:fixed-width-risk} for finite families. The finite-subset
argument above proves the countable case.

In particular, the integer budget
\[
\left\lceil
\frac{\log(2W/\delta)}{\log\cosh(\Delta/4)}
\right\rceil
\]
is sufficient. The single-candidate minimax lower bound of
Theorem~\ref{thm:single_certification} applies because singleton families
are included among families of width at most $W$. For fixed $W$, these
upper and lower bounds have the same leading coefficient of
$\log(1/\delta)$, proving \eqref{app:reuse:fixed-width-limit}.
\end{proof}

\subsection{Gaussian concentration facts}
\label{app:reuse:facts}

We first record the quadratic-form bounds used below, including their
normalizations.

\begin{lemma}[Gaussian quadratic forms]
\label{app:reuse:lemma-quadratic}
For a fixed real symmetric matrix $B$ and $u>0$,
\begin{equation}
\mathbb P\left\{
 |Z(B)-\operatorname{tr} B|>
 2\|B\|_F\sqrt{u/s}+2\|B\|_{\mathrm{op}}u/s
\right\}\le 2e^{-u}.
\label{app:reuse:quadratic-two-sided}
\end{equation}
If $B\succeq0$, its lower tail also satisfies
\begin{equation}
\mathbb P\left\{Z(B)<\operatorname{tr} B-2\|B\|_F\sqrt{u/s}\right\}
\le e^{-u}.
\label{app:reuse:quadratic-lower}
\end{equation}
\end{lemma}

\begin{proof}
Let $\lambda_1,\ldots,\lambda_n$ be the eigenvalues of $B$.
For $0\le\theta<s/(2\|B\|_{\mathrm{op}})$, independence and the Gaussian
moment-generating function give
\[
\log\mathbb E\exp\{\theta(Z(B)-\operatorname{tr} B)\}
=\frac s2\sum_{k=1}^n
\left[-\frac{2\theta\lambda_k}{s}
      -\log\left(1-\frac{2\theta\lambda_k}{s}\right)\right].
\]
Using $-x-\log(1-x)\le x^2/[2(1-|x|)]$ for $|x|<1$ yields
\begin{equation*}
\log\mathbb E\exp\{\theta(Z(B)-\operatorname{tr} B)\}
\le
\frac{\theta^2\|B\|_F^2/s}
     {1-2\theta\|B\|_{\mathrm{op}}/s}.
\end{equation*}
The same inequality applies to $-B$.  The exponential Markov bound,
with
\[
\theta=
\frac{\sqrt u}
     {\|B\|_F/\sqrt s+2\|B\|_{\mathrm{op}}\sqrt u/s},
\]
gives each one-sided bound with threshold
$2\|B\|_F\sqrt{u/s}+2\|B\|_{\mathrm{op}}u/s$ and probability $e^{-u}$.
The case $B=0$ is immediate.  Summing the two tails proves
\eqref{app:reuse:quadratic-two-sided}.

When $B\succeq0$, the inequality
$x-\log(1+x)\le x^2/2$ for $x\ge0$ instead gives
\[
\log\mathbb E\exp\{-\theta(Z(B)-\operatorname{tr} B)\}
\le\theta^2\|B\|_F^2/s,
~~ \theta\ge0.
\]
Optimizing the exponential Markov bound proves
\eqref{app:reuse:quadratic-lower}.
\end{proof}

We also use the Gaussian singular-value inequalities in
\cite[Corollary~5.35 and its proof]{vershynin2012nonasymptotic}: if
$G_0\in\mathbb R^{h\times s}$ has independent standard Gaussian entries, then,
for every $t>0$,
\begin{align}
\mathbb P\{\|G_0\|_{\mathrm{op}}>\sqrt s+\sqrt h+\sqrt{2t}\}
&\le e^{-t},
\label{app:reuse:gaussian-upper}\\
\mathbb P\{\sigma_{\min}(G_0)<\sqrt s-\sqrt h-\sqrt{2t}\}
&\le e^{-t}~~ (s\ge h).
\label{app:reuse:gaussian-lower}
\end{align}
In \eqref{app:reuse:gaussian-lower}, $\sigma_{\min}$ is the smallest
of the $h$ singular values. The cited proof combines the Gaussian
comparison bound of Theorem~5.32 with the one-sided Gaussian
concentration inequality of Proposition~5.34, yielding the individual
tail probabilities displayed above. In the proof below the
lower bound is invoked only when its right-hand threshold is positive;
otherwise positive semidefiniteness alone controls the lower covariance
deviation.

\subsection{Uniform concentration under a shared energy envelope}
\label{app:reuse:proof}

\begin{proposition}[Uniform concentration for a finite family]
\label{app:reuse:uniform}
Let $M,s\ge1$ be integers and let $0<\alpha<1$. Suppose that
$T\succeq0$, $\operatorname{tr}T>0$, and $0\preceq E_i\preceq T$
for $1\le i\le M$, with all matrices independent of the Gaussian
validation probes. Set
\[
H_M(\alpha)=\sqrt{\log(2M)}+\log(5/\alpha).
\]
Then
\begin{equation*}
\mathbb P\left\{
\max_{i\le M}|\widehat q_i-q_i|
\le\min\{1,16\sqrt{H_M(\alpha)/s}\}\right\}\ge1-\alpha.
\end{equation*}
In particular, $s\ge256H_M(\alpha)/\epsilon^2$ guarantees simultaneous
additive accuracy $\epsilon$. This proves
\eqref{eq:shared_envelope_concentration} with the absolute constant
$C=16$.
\end{proposition}

\begin{proof}[Proof of Proposition~\ref{app:reuse:uniform}]
Both ratios are unchanged by a common scaling of all matrices, so assume
$\operatorname{tr} T=1$.  As $T\ne0$, the denominator $Z(T)$ is strictly positive
almost surely.  The inequality $0\preceq E_i\preceq T$ also gives
$0\le q_i,\widehat q_i\le1$.

On the support of $T$, write
\[
E_i=T^{1/2}K_iT^{1/2},~~ 0\preceq K_i\preceq I,
~~ q_i=\operatorname{tr} E_i.
\]
This factorization remains valid for singular $T$: a vector in the kernel
of $T$ lies in the kernel of every $E_i$, and $K_i$ can be set to zero on
that kernel.  Define
\begin{equation*}
F_i=E_i-q_iT=T^{1/2}(K_i-q_iI)T^{1/2}.
\end{equation*}
Then $\operatorname{tr} F_i=0$, $\|K_i-q_iI\|_{\mathrm{op}}\le1$, and
\begin{equation*}
\widehat q_i-q_i=\frac{Z(F_i)}{Z(T)}.
\end{equation*}

\paragraph{Deterministic spectral decomposition.}
Choose an integer $r\ge1$, diagonalize $T$, and put its largest
$h=\min(r,n)$ eigenvalues in a head block $T_H$ and the remaining
eigenvalues in a tail block $T_L$.  The orthogonal change of basis preserves
the distribution of the Gaussian probes.  If $r<n$, the eigenvalue ordering
and $\operatorname{tr} T=1$ imply
\[
\|T_L\|_{\mathrm{op}}\le\frac1{r+1}\le\frac1r.
\]
If $r\ge n$, the tail and all cross blocks are empty; all bounds for
those blocks below are interpreted as zero.  This convention allows us to
choose $r$ without assuming that the ambient dimension is large.

For any symmetric contraction $L$ and PSD matrix $D$,
\begin{align*}
\|D^{1/2}LD^{1/2}\|_*
&\le\|D^{1/2}\|_F\|LD^{1/2}\|_F
\le\operatorname{tr} D,
\\
\|D^{1/2}LD^{1/2}\|_F^2
&\le\|D\|_{\mathrm{op}}\operatorname{tr}(DL^2)
\le\|D\|_{\mathrm{op}}\operatorname{tr} D.
\end{align*}
Applying these inequalities to the principal blocks of
$K_i-q_iI$ gives
\begin{equation}
\|F_{i,HH}\|_*\le1,~~
\|F_{i,LL}\|_F^2\le1/r,~~
\|F_{i,LL}\|_{\mathrm{op}}\le1/r.
\label{app:reuse:block-bounds}
\end{equation}

The cross block has a sharper bound because $K_i$ is a positive
contraction.  Taking the head principal block of $K_i-K_i^2\succeq0$
shows that
\[
K_{i,HL}K_{i,LH}
\preceq K_{i,HH}-K_{i,HH}^2
\preceq\tfrac14I.
\]
Since $T$ is block diagonal, $F_{i,HL}=E_{i,HL}$, and therefore
\begin{align*}
\|F_{i,HL}\|_F^2
&=\operatorname{tr}\bigl(T_H^{1/2}K_{i,HL}T_LK_{i,LH}T_H^{1/2}\bigr)
\\
&\le\|T_L\|_{\mathrm{op}}
      \operatorname{tr}\bigl(T_HK_{i,HL}K_{i,LH}\bigr)
\le\frac1{4r}.
\end{align*}
None of these estimates requires $E_i$ and $T$ to commute.

\paragraph{Uniform control of the head.}
Let $G_H$ and $G_L$ contain the corresponding rows of the probe matrix.
For $t>0$, set
\begin{equation*}
L_0=\log(2M)+t,~~
a=\frac{\sqrt r+\sqrt{2t}}{\sqrt s},~~
b=\frac{L_0}{rs}.
\end{equation*}
By \eqref{app:reuse:gaussian-upper}--\eqref{app:reuse:gaussian-lower}
and $h\le r$, an event $\mathcal H$ of probability at least
$1-2e^{-t}$ satisfies
\begin{equation}
\|G_H\|_{\mathrm{op}}\le\sqrt s(1+a),~~
\left\|G_HG_H^\top/s-I_h\right\|_{\mathrm{op}}\le2a+a^2.
\label{app:reuse:head-event}
\end{equation}
For completeness, if $a<1$, then $s>r\ge h$ and the lower
singular-value bound gives
$\lambda_{\min}(G_HG_H^\top/s)\ge(1-a)^2$.
If $a\ge1$, its lower covariance deviation is at most $1$, which is
already bounded by $2a+a^2$.  Thus \eqref{app:reuse:head-event} does
not require $s\ge r$.
On $\mathcal H$, \eqref{app:reuse:block-bounds} implies, simultaneously
for all candidates,
\begin{equation*}
\left|
\operatorname{tr}\bigl(F_{i,HH}(G_HG_H^\top/s-I_h)\bigr)
\right|\le2a+a^2.
\end{equation*}

\paragraph{Finite-family control of the tail.}
Apply Lemma~\ref{app:reuse:lemma-quadratic} to each $F_{i,LL}$, with
$u=L_0$.  Equation~\eqref{app:reuse:block-bounds} and a union bound
show that
\begin{equation*}
\max_{i\le M}
\left|
\frac1s\operatorname{tr}(G_L^\top F_{i,LL}G_L)-\operatorname{tr} F_{i,LL}
\right|
\le2\sqrt b+2b
\end{equation*}
except on an event of probability at most
$2Me^{-L_0}=e^{-t}$.

\paragraph{Conditional control of the cross term.}
Conditional on $G_H$, each cross contribution
\[
C_i=\frac2s\operatorname{tr}(G_H^\top F_{i,HL}G_L)
\]
is a centered Gaussian random variable, because $G_L$ is independent
of $G_H$.  Its conditional variance satisfies
\begin{align*}
\operatorname{Var}(C_i\mid G_H)
&=\frac4{s^2}\|F_{i,LH}G_H\|_F^2
\\
&\le\frac4{s^2}\|F_{i,HL}\|_F^2\|G_H\|_{\mathrm{op}}^2
\le\frac{(1+a)^2}{rs}
~~\text{on }\mathcal H.
\end{align*}
Consequently, for every realization of $G_H$ in $\mathcal H$, a
conditional Gaussian tail bound and a union bound imply
\[
\mathbb P\left\{
\max_{i\le M}|C_i|>(1+a)\sqrt{2b}\;\middle|\;G_H
\right\}\le2Me^{-L_0}=e^{-t}.
\]
Integrating this conditional bound shows that the cross event can fail
inside $\mathcal H$ with probability at most $e^{-t}$.
No independence across candidates is used.

Since $\operatorname{tr} F_i=0$, the centered head, centered tail, and cross
contributions sum to $Z(F_i)$.  Combining the preceding events gives
\begin{equation}
\max_{i\le M}|Z(F_i)|\le
B:=2a+a^2+2\sqrt b+2b+(1+a)\sqrt{2b}
\label{app:reuse:numerator-bound}
\end{equation}
with probability at least $1-4e^{-t}$.

\paragraph{The denominator.}
Because $T\succeq0$ and $\|T\|_F\le\operatorname{tr} T=1$,
\eqref{app:reuse:quadratic-lower} gives
\begin{equation*}
Z(T)\ge1-2\sqrt{t/s}
\end{equation*}
except on an event of probability at most $e^{-t}$.
Thus, whenever $s>4t$,
\begin{equation}
\max_{i\le M}|\widehat q_i-q_i|
\le\frac{B}{1-2\sqrt{t/s}}
\label{app:reuse:exact-bound}
\end{equation}
with probability at least $1-5e^{-t}$.

\paragraph{Choice of the split and the explicit constant.}
Set
\[
m=\log(2M),~~ r=\lceil\sqrt m\rceil,~~
t=\log(5/\alpha),~~ H=\sqrt m+t=H_M(\alpha).
\]
In particular, $r\ge1$ and $t>1$.  Define
\[
x=\frac{r+2t+(m+t)/r}{s}.
\]
The ceiling causes no dimension-dependent remainder: since
$r\le\sqrt m+1$, $m/r\le\sqrt m$, and $t/r\le t$,
\begin{equation*}
x\le\frac{2\sqrt m+1+3t}{s}\le\frac{4H}{s}.
\end{equation*}
The empty-tail convention covers $r\ge n$.

Let $h_0=(r+2t)/s$, so that $x=h_0+b$ and $a^2\le2h_0$.
The square-root terms in \eqref{app:reuse:numerator-bound} are at most
$(2+3\sqrt2)\sqrt x$.  Moreover,
$a\sqrt{2b}\le2\sqrt{h_0b}\le x$, while
$a^2\le2x$ and $2b\le2x$.  Hence
\begin{equation}
B\le(2+3\sqrt2)\sqrt x+5x
\le(4+6\sqrt2)\sqrt{H/s}+20H/s.
\label{app:reuse:simplified-numerator}
\end{equation}
Put $z=\sqrt{H/s}$.  If $s\ge256H$, then $z\le1/16$ and
$\sqrt{t/s}\le z$.  Equations~\eqref{app:reuse:exact-bound} and
\eqref{app:reuse:simplified-numerator} yield
\[
\max_i|\widehat q_i-q_i|
\le\frac{(4+6\sqrt2)z+20z^2}{1-2z}
\le
\frac{4+6\sqrt2+20/16}{1-2/16}\,z
<16z.
\]
The failure probability is at most $5e^{-t}=\alpha$.
If $s<256H$, then $16\sqrt{H/s}>1$ and the desired bound is
deterministic, since both ratios lie in $[0,1]$.
Together these cases prove the stated bound with
$\min\{1,16\sqrt{H_M(\alpha)/s}\}$.
\end{proof}

\subsection{Applying uniform concentration to the boundary family}
\label{app:reuse:width-application}

We now use the arithmetic midpoint
\[
\theta_{\mathrm{add}}=\rho-g/2.
\]
This choice maximizes the smaller additive distance to the two promise
endpoints. It is used for the additive-gap query bound, whereas
$\theta_{\log}$ was used for the exact high-confidence coefficient.
Both rules solve the same certification task.

\begin{proposition}[An explicit simultaneous-certification budget]
\label{app:reuse:width-upper}
Let $0<\rho<1$, $0<g<\rho$, $0<\delta<1$, and let the independent
candidate family have width at most $W$. Accepting when
$\widehat q_E\le\rho-g/2$ makes every required decision correct with
probability at least $1-\delta$ whenever
\begin{equation}
s\ge\frac{2304}{g^2}
\left[\sqrt{\log(4W)}+\log(5/\delta)\right].
\label{app:reuse:explicit-width-budget}
\end{equation}
This proves the upper bound in Theorem~\ref{thm:simultaneous_certification}; the smaller gap range and \(\rho\)-dependent constants in that theorem concern the matching lower bound, not this sufficient budget. The constant \(2304=48^2\) comes from requiring the estimation error \(16\sqrt{H_M(\delta)/s}\) to be at most \(g/3\).
\end{proposition}
\begin{proof}
For a finite family, form the boundary set in
Lemma~\ref{app:reuse:boundary-lemma}. If it is empty, there are no required
decisions. Otherwise let its cardinality be $M\le2W$. Apply
Proposition~\ref{app:reuse:uniform} once to the entire boundary set,
with $\alpha=\delta$. Since
\[
H_M(\delta)\le\sqrt{\log(4W)}+\log(5/\delta),
\]
condition \eqref{app:reuse:explicit-width-budget} implies
$16\sqrt{H_M(\delta)/s}\le g/3$.
Thus, with probability at least $1-\delta$, every boundary estimate has
absolute error at most $g/3$. For each existing bad boundary,
\[
\widehat q_{E^{\mathrm{bad}}}
\ge q_{E^{\mathrm{bad}}}-g/3
\ge\rho-g/3>\rho-g/2,
\]
so it is rejected. For each existing good boundary,
\[
\widehat q_{E^{\mathrm{good}}}
\le q_{E^{\mathrm{good}}}+g/3
\le\rho-2g/3<\rho-g/2,
\]
so it is accepted. The strict slack covers both promise endpoints
without any argument about ties. The boundary reduction then proves
simultaneous correctness for the finite family.

For a countable family, apply this same budget and threshold to every
finite subset in an increasing exhaustion. Each subset still has at
most $2W$ boundary candidates, so its error probability is at most
$\delta$, independently of subset size. The increasing error events
and continuity of probability give the full-family guarantee.

The uniform concentration result has already accounted for all the
boundary matrices jointly. No further union bound over the $2W$
boundaries is taken here. Such an additional step would discard the
improved dependence on width that the shared-energy analysis provides.
\end{proof}

\subsection{Independent families and validation-dependent visits}
\label{app:reuse:independence}

On the simultaneous event above, every required decision in the fixed
candidate family is correct. Therefore the visiting order, skipped
checks, revisits, switches between paths, and the final selection may
be arbitrary measurable functions of the validation observations and
other available information. No independence is required of the final
selected index. Independence is imposed on the complete candidate
universe before that selection.

It is not enough that each realized run happens to visit a short or
ordered path. To see this, take $A=I_n$, draw $G\in\mathbb R^{n\times s}$
with $s<n$, and let $P_G$ be the orthogonal projection onto its column
space. The Gaussian matrix has rank $s$ almost surely, and
\[
(I-P_G)G=0,~~
\widehat q(P_G)=0,~~
q(P_G)=\frac{n-s}{n}.
\]
For any target $\rho<(n-s)/n$, both midpoint rules would accept this
unsafe candidate almost surely. The realized singleton is trivially
ordered, but its construction depends on the validation sketch and
falls outside the independent-family assumptions.

Finally, the boundary reductions above hold for the fixed target
$\rho$ and gap $g$. Their boundary matrices generally change when the
target changes. The proof therefore does not grant simultaneous
validity for an arbitrarily selected target after observing the sketch.
The finite-family concentration proposition itself controls all the
ratios at once, but replacing the full family size by its chain width
uses boundaries for the specified, fixed classification problem.

\subsection{Reuse for mean-square output error}
\label{app:reuse:mean-square}

Let the input distribution have a known second moment
$K=\mathbb E[xx^\top]\succeq0$, and suppose
$\mathbb E\|Ax\|_2^2>0$. Set $B=AK^{1/2}$ and
\[
T_K=B^\top B=K^{1/2}A^\top AK^{1/2},~~
E_i^K=B^\top(I-P_i)B
=K^{1/2}A^\top(I-P_i)AK^{1/2}.
\]
For orthogonal $P_i$, these matrices satisfy
$0\preceq E_i^K\preceq T_K$ and
\[
\frac{\mathbb E\|(I-P_i)Ax\|_2^2}{\mathbb E\|Ax\|_2^2}
=\frac{\operatorname{tr}E_i^K}{\operatorname{tr}T_K}.
\]
Congruence preserves the Loewner order: every chain of original
residual energies maps to a chain of the transformed energies. Hence
a width-$W$ cover remains a cover by at most $W$ chains after this
transformation. Singular $K$ causes no difficulty because
$\operatorname{tr}T_K>0$ is assumed.

Each Gaussian response needed for the transformed problem is
$Bg_j=A(K^{1/2}g_j)$ and therefore costs exactly one forward query to
$A$. Once these responses are cached, all candidate projections use
the same cache. The finite- and countable-family reductions, the
fixed-width confidence bound, and the shared-energy upper bound above
therefore apply for every prescribed $K$, provided that $K$ and the
complete candidate family are independent of the validation probes.
The actual input distribution need not be Gaussian; only its known
second moment enters the target.

When the worst case also ranges over known second moments, the
subclass $K=I_n$ contains the Frobenius problem. Thus the matching
lower bounds of Appendix~\ref{app:reuse-lower} also apply to that joint
worst case. This inclusion does not assert the same lower bound for
every individual prescribed, possibly singular, $K$.

\section{Adaptive Lower Bounds for Simultaneous Certification}
\label{app:reuse-lower}

This appendix proves the lower bound in
Theorem~\ref{thm:simultaneous_certification}. The argument applies to
arbitrary adaptive forward-query procedures, rather than only to the
Gaussian ratio estimator. All logarithms are natural.

\subsection{The task and the dimension-uniform quantifiers}
\label{app:width:model}

Fix $\rho\in(0,1)$ and a gap $g\in(0,\rho)$. An unknown, nonzero
matrix $A\in\mathbb R^{m\times n}$ is accessed through responses $Av$.
The public candidate family consists of orthogonal residual projectors
$R_1,\ldots,R_M\in\mathbb R^{m\times m}$; the corresponding
approximation projectors are $I_m-R_j$. Define
\begin{equation}
 q_j(A)=\frac{\operatorname{tr}(A^\top R_jA)}
                    {\operatorname{tr}(A^\top A)},
 ~~ E_j(A)=A^\top R_jA,
 ~~ T(A)=A^\top A.
 \label{app:width:ratios}
\end{equation}
A decision $D_j=1$ means acceptance. Simultaneous safety and
completeness require, for every admissible deterministic $A$ and family,
\begin{equation}
 \mathbb P_A\!\left\{
 \begin{array}{ll}
 D_j=0&\text{for all }j\text{ with }q_j(A)\ge\rho,\\
 D_j=1&\text{for all }j\text{ with }q_j(A)\le\rho-g
 \end{array}\right\}\ge1-\delta.
 \label{app:width:success}
\end{equation}
Candidates in $(\rho-g,\rho)$ impose no decision requirement.
The family has width at most $W$ if the matrices $E_j(A)$ can be covered
by at most $W$ Loewner chains. In particular, every family with $M\le W$
is admissible, regardless of $A$.

The algorithm knows the dimensions, projectors, and testing parameters.
Its query and decision maps are measurable and may use private randomness
independent of $A$. Every response concerning $A$ counts toward a hard
budget $s\in\mathbb Z_{\ge0}$, including responses obtained before a
putative validation phase. Query directions may depend on the complete
preceding transcript, and early stopping is allowed. No construction
transcript, matrix factorization, exact energy, or other observation of
$A$ is supplied for free. A dimension-uniform budget is one bound $s$
valid for all finite $m,n$ and admissible pairs $(A,\{R_j\})$; the
algorithm itself may depend on these public inputs.

\begin{theorem}[Simultaneous certification requires width and confidence information]
\label{app:width:lower-bound}
Let $W\ge1$ be an integer and put $w=\min\{\rho,1-\rho\}$. If
$0<g\le w/32$ and $0<\delta\le1/4$, every dimension-uniform procedure
satisfying~\eqref{app:width:success} on all families of width at most
$W$ has
\begin{equation}
 s\ge\frac{w^2}{65536g^2}
       \left[\sqrt{\log(W+1)}+\log\frac1\delta\right].
 \label{app:width:final-bound}
\end{equation}
The numerical constant is not optimized. The lower bound already
holds on finite families of fixed orthogonal residual projectors.
\end{theorem}

We prove the theorem below. The width term is obtained by identifying a
hidden covariance label, and the confidence term by testing two
covariances. The hard families used for these two terms need not be the
same: a uniform procedure must succeed on both.

\subsection{A subspace packing with quadratic log cardinality}
\label{app:width:packing-section}

For completeness, we state the precise concentration ingredient and
derive the packing. The sole geometric input is the standard
logarithmic Sobolev inequality for Haar measure on $\mathrm{SO}(k)$,
with its constant specified below.

\begin{lemma}[Haar concentration used in the packing]
\label{app:width:haar}
Let $k\ge3$, let $U$ be Haar distributed on $\mathrm{SO}(k)$, and let
$F$ be a smooth function satisfying
$|F(U)-F(V)|\le L\lVert U-V\rVert_F$. For $L>0$ and $t>0$,
\begin{equation*}
 \mathbb P\{F(U)-\mathbb EF(U)\ge t\}
 \le\exp\left\{-\frac{(k-2)t^2}{24L^2}\right\}.
\end{equation*}
\end{lemma}

\begin{proof}
Write $\mu$ for Haar probability measure. With the metric locally
induced by the Frobenius norm, the logarithmic Sobolev inequality is
\begin{equation}
 \operatorname{Ent}_\mu(f^2)
 \le\frac{8}{k-2}\int\lVert\nabla f\rVert^2\,d\mu,
 ~~
 \operatorname{Ent}_\mu(X)
 =\mathbb E_\mu[X\log X]-\mathbb E_\mu X\log\mathbb E_\mu X.
 \label{app:width:haar-lsi}
\end{equation}
This is the $\mathrm{SO}(k)$ case of
\cite[Theorem~5.16]{meckes2019}, whose log-Sobolev convention is
$\operatorname{Ent}(f^2)\le2C\mathbb E\lVert\nabla f\rVert^2$ with
$C=4/(k-2)$.
Here is the concentration argument, to make the required normalization
explicit. Set $\psi(\lambda)=\log\mathbb E e^{\lambda F(U)}$.
The group is compact, so these exponential moments and their derivatives
exist. Apply~\eqref{app:width:haar-lsi} to
$f=e^{\lambda F/2}$ and use $\lVert\nabla F\rVert\le L$:
\[
 \lambda\psi'(\lambda)-\psi(\lambda)
 \le\frac{2\lambda^2L^2}{k-2}.
\]
Integrating the derivative of $\psi(\lambda)/\lambda$ from zero gives
$\psi(\lambda)-\lambda\mathbb EF\le2\lambda^2L^2/(k-2)$ for
$\lambda>0$. Exponential Markov inequality, optimized at
$\lambda=(k-2)t/(4L^2)$, gives the stronger exponent
$-(k-2)t^2/(8L^2)$. Weakening its denominator to $24$ proves the stated
form, which also agrees with the general concentration statement
in~\cite[Theorem~5.17]{meckes2019}.
\end{proof}

\begin{lemma}[A finite projector packing]
\label{app:width:packing}
For every integer $k\ge3$, put $d=2k$. There exist at least
$\lfloor e^{k(k-2)/96}\rfloor$ rank-$d/2$ orthogonal projectors
$P_i\in\mathbb R^{d\times d}$ with
\begin{equation}
 \operatorname{tr}(P_iP_j)\le\frac{3d}{8}
 ~\text{whenever }i\ne j.
 \label{app:width:overlap}
\end{equation}
In particular, for every $W\ge16$ one can choose exactly $M=W$
projectors with $d\le24\sqrt{\log W}$.
\end{lemma}

\begin{proof}
For $U\in\mathrm{SO}(k)$ define the graph projector
\begin{equation*}
 P(U)=\frac12\begin{pmatrix}I_k&U\\U^\top&I_k\end{pmatrix}.
\end{equation*}
Direct multiplication gives $P(U)^2=P(U)$ and
$\operatorname{tr}P(U)=k$. Also,
\[
 \operatorname{tr}(P(U)P(V))
 =\frac d4+\frac12\operatorname{tr}(U^\top V).
\]
Thus pairwise traces $\operatorname{tr}(U^\top V)\le k/2$ imply
\eqref{app:width:overlap}.

The smooth function $F(U)=\operatorname{tr}U$ is $\sqrt{k}$-Lipschitz
by Cauchy--Schwarz. Its Haar mean is zero: multiplying on the left by a
diagonal rotation that changes the signs of any prescribed row and one
other row changes the sign of the prescribed diagonal entry without
changing its distribution. Lemma~\ref{app:width:haar} now gives
\[
 \mathbb P\{\operatorname{tr}U>k/2\}
 \le e^{-k(k-2)/96}=:p_k.
\]
For each fixed center $U_0$, Haar invariance gives the same upper bound
for the set $\{V:\operatorname{tr}(U_0^\top V)>k/2\}$.
After choosing $j$ centers, their excluded sets have total measure at
most $jp_k$. Whenever $jp_k<1$, another center exists. This greedy
construction supplies at least $\lfloor p_k^{-1}\rfloor$ centers.
The trace relation is symmetric in its two arguments, so every selected
pair has the required overlap.

For the last assertion take
\begin{equation}
 k=\left\lceil\sqrt{96\log W}\right\rceil+2,
 ~~ d=2k.
 \label{app:width:choose-dimension}
\end{equation}
Then $k(k-2)\ge96\log W$, so at least $W$ members are available.
Retain exactly $M=W$. Since $W\ge16$,
$d\le2\sqrt{96\log W}+6\le24\sqrt{\log W}$.
\end{proof}

\subsection{A covariance family with exactly one acceptable candidate}
\label{app:width:covariance-section}

Fix a packing from Lemma~\ref{app:width:packing}, and define
$Q_i=2P_i-I_d$ and $a=16g/w$. The gap assumption gives $0<a\le1/2$.
The matrices $Q_i$ satisfy $Q_i^2=I_d$ and
$\operatorname{tr}Q_i=0$, with each eigenvalue $1$ and $-1$ occurring
$d/2$ times. In output dimension $m=d+2$, set
\begin{equation}
 \begin{aligned}
 \Sigma_i&=\operatorname{diag}\left(
       \frac wd(I_d-aQ_i),\beta,\zeta\right),\\
 \beta&=\rho-\frac w2+\frac{3wa}{8},
 ~~
 \zeta=1-\rho-\frac w2-\frac{3wa}{8},\\
 R_j&=\operatorname{diag}(P_j,1,0),~~ 1\le j\le M.
 \end{aligned}
 \label{app:width:covariances}
\end{equation}
The candidates $R_j$ are fixed and do not depend on the hidden label
$i$. All blocks of $\Sigma_i$ are positive definite. More explicitly,
\[
 \frac wd(I_d-aQ_i)\succeq\frac{w}{2d}I_d,
 ~~ \beta\ge\frac\rho2,
 ~~ \zeta\ge\frac{5(1-\rho)}{16}.
\]
The first bound uses $a\le1/2$. The second uses $w\le\rho$; the
third follows from $3wa/8\le3w/16$ and $w\le1-\rho$.
Furthermore,
\begin{equation}
 \operatorname{tr}\Sigma_i=1,
 ~~
 \det\Sigma_i=
 \left(\frac wd\right)^d(1-a^2)^{d/2}\beta\zeta,
 \label{app:width:trace-determinant}
\end{equation}
so every hypothesis has the same total population energy and determinant.

Write $\mu_{ij}=\operatorname{tr}(R_j\Sigma_i)$. From
\eqref{app:width:covariances},
\begin{equation*}
 \mu_{ij}=\rho+\frac{3wa}{8}
                -\frac{wa}{d}\operatorname{tr}(P_jQ_i).
\end{equation*}
Since $P_iQ_i=P_i$, the diagonal case has
$\operatorname{tr}(P_iQ_i)=d/2$. For $j\ne i$,
\eqref{app:width:overlap} gives
\[
 \operatorname{tr}(P_jQ_i)
 =2\operatorname{tr}(P_jP_i)-d/2\le d/4.
\]
Consequently,
\begin{equation}
 \mu_{ii}=\rho-\frac{wa}{8}=\rho-2g,
 ~~
 \mu_{ij}\ge\rho+\frac{wa}{8}=\rho+2g~(j\ne i).
 \label{app:width:population-separation}
\end{equation}
This verifies the separation used in~\eqref{eq:hard_family_separation}:
only candidate $i$ lies below $\rho-g$, while every other candidate
lies above $\rho$. Both sides have strict slack.

The Gaussian pairwise relative entropy can be computed exactly. For
positive definite covariances in dimension $m$,
\[
 \operatorname{KL}\bigl(N(0,\Sigma_i)\Vert N(0,\Sigma_j)\bigr)
 =\frac12\left[
   \operatorname{tr}(\Sigma_j^{-1}\Sigma_i)-m
   +\log\frac{\det\Sigma_j}{\det\Sigma_i}\right].
\]
The determinant term vanishes by~\eqref{app:width:trace-determinant}.
The two scalar blocks each contribute one to the trace. On the first
$d$ coordinates,
\[
 (I_d-aQ_j)^{-1}(I_d-aQ_i)
 =\frac{(I_d+aQ_j)(I_d-aQ_i)}{1-a^2},
\]
whose trace is $[d-a^2\operatorname{tr}(Q_jQ_i)]/(1-a^2)$.
It follows that
\begin{equation*}
 \begin{aligned}
 \operatorname{KL}\bigl(N(0,\Sigma_i)\Vert N(0,\Sigma_j)\bigr)
 =\frac{a^2[d-\operatorname{tr}(Q_iQ_j)]}{2(1-a^2)}
 \le\frac{a^2d}{1-a^2}.
 \end{aligned}
\end{equation*}
Here $|\operatorname{tr}(Q_iQ_j)|\le d$ because $Q_i,Q_j$ are
orthogonal matrices. This upper bound is of order $g^2d/w^2$.

\subsection{Transferring covariance information to adaptive queries}
\label{app:width:adaptive-section}

Under label $i$, draw the $n$ columns of $A$ independently from
$N(0,\Sigma_i/n)$. These Gaussian matrices are priors for a lower
bound on the deterministic-matrix problem; they do not change the
pointwise requirement~\eqref{app:width:success}.

Lemma~\ref{app:single:innovation} proves the Gaussian-innovation
representation in~\eqref{eq:gaussian_innovation} for an arbitrary
positive definite output covariance, including
\eqref{app:width:covariances}. In the present notation, after observing queries $V_t=[v_1,\ldots,v_t]$ and
responses $Y_t=AV_t$, let $M_t=Y_tV_t^\dagger$ and
$\Pi_t=I_n-V_tV_t^\dagger$. Conditional on the transcript and private
seed, the residual $A-M_t$ is centered matrix Gaussian with covariance
\[
 \operatorname{Cov}\bigl((A-M_t)_{a,u},(A-M_t)_{b,v}
                  \mid\text{transcript}\bigr)
 =\frac{(\Sigma_i)_{ab}}{n}(\Pi_t)_{uv}.
\]
This statement follows by the same Gaussian conditioning induction as
for~\eqref{eq:gaussian_innovation}. In particular, for the next query
$v$, with $v_\perp=\Pi_tv$,
\begin{equation*}
 Av=M_tv+\frac{\lVert v_\perp\rVert_2}{\sqrt n}\,\xi,
 ~~ \xi\sim N(0,\Sigma_i),
\end{equation*}
where the innovation is fresh and independent of the preceding
transcript whenever $v_\perp\ne0$. Gaussian components along
orthogonal unexplored input directions are independent even though
the output coordinates can be correlated. If $v_\perp=0$, the response
is determined by the past. After a nonredundant query with
$u=v_\perp/\lVert v_\perp\rVert_2$, the updates are
$M_{t+1}=M_t+\xi u^\top/\sqrt n$ and $\Pi_{t+1}=\Pi_t-uu^\top$.
They preserve the displayed conditional covariance.

Thus the entire transcript is a measurable function of the private
seed and at most $s$ independent $N(0,\Sigma_i)$ vectors. The function
is independent of $i$ and $\Sigma_i$: the simulator uses the supplied
vectors and the query rule, not the unknown covariance. Redundant
queries consume budget but no innovation; after early stopping, unused
vectors are discarded. The statement is also valid for $s=0$.

Let $\mathsf P_{i,n}$ be the resulting transcript law, including the
private seed and output. Data processing and product relative entropy
\citep{scarlett2019fano} give, for every pair $i,j$ and every input dimension $n$,
\begin{equation*}
 \operatorname{KL}(\mathsf P_{i,n}\Vert\mathsf P_{j,n})
 \le s\operatorname{KL}\bigl(N(0,\Sigma_i)\Vert N(0,\Sigma_j)\bigr)
 \le\frac{sa^2d}{1-a^2}.
\end{equation*}
The private seed contributes no relative entropy because it has the
same distribution under both labels. The bound remains valid for
procedures that depend on $n$.

\subsection{A uniform finite-dimensional leakage bound}
\label{app:width:leakage-section}

The population separation must be transferred to realized matrices.
The following elementary estimate makes this step uniform over the
finite set of labels and candidates.

\begin{lemma}[Finite-dimensional prior leakage]
\label{app:width:leakage}
Let $\Sigma\succ0$ have trace one, let the columns of $A$ be independent
$N(0,\Sigma/n)$ vectors, and let $R_1,\ldots,R_M$ be fixed orthogonal
projectors. Write $\mu_j=\operatorname{tr}(R_j\Sigma)$ and let
$q_j(A)$ be as in~\eqref{app:width:ratios}. For $0<g<1$,
\begin{equation*}
 \mathbb P\left\{\max_{j\le M}|q_j(A)-\mu_j|\ge g/2\right\}
 \le\min\left\{1,\frac{128(M+1)}{ng^2}\right\}.
\end{equation*}
\end{lemma}

\begin{proof}
Write the columns of $A$ as $z_\ell/\sqrt n$, where the $z_\ell$ are
independent $N(0,\Sigma)$ vectors, and set
\[
 D_n=\frac1n\sum_{\ell=1}^n\lVert z_\ell\rVert_2^2,
 ~~ N_{j,n}=\frac1n\sum_{\ell=1}^n z_\ell^\top R_jz_\ell.
\]
Then $q_j=N_{j,n}/D_n$, $\mathbb ED_n=1$, and
$\mathbb EN_{j,n}=\mu_j\in[0,1]$.
Diagonalizing the relevant Gaussian quadratic form and using
$\operatorname{Var}(\chi_1^2)=2$ gives
\[
 \operatorname{Var}(D_n)=\frac2n\operatorname{tr}(\Sigma^2)\le\frac2n,
 ~~
 \operatorname{Var}(N_{j,n})
 =\frac2n\operatorname{tr}\bigl((\Sigma^{1/2}R_j\Sigma^{1/2})^2\bigr)
 \le\frac2n.
\]
The last inequalities use positivity and trace at most one.
Apply Chebyshev's inequality with $h=g/8$ to these $M+1$ variables.
Except on an event of probability at most
$2(M+1)/(nh^2)=128(M+1)/(ng^2)$, all their deviations from their means
are at most $h$. On that event,
\[
 |q_j-\mu_j|
 \le\frac{|N_{j,n}-\mu_j|+\mu_j|D_n-1|}{D_n}
 \le\frac{2h}{1-h}<\frac g2,
\]
since $g<1$. This proves the claim.
\end{proof}

Apply the lemma under each $\Sigma_i$. Because its bound depends only
on the common trace and family size, it yields
\begin{equation*}
 \eta_n:=\max_i\mathbb P_i\left\{
       \max_j|q_j(A)-\mu_{ij}|\ge g/2\right\}
 \le\min\left\{1,\frac{128(M+1)}{ng^2}\right\}
 \longrightarrow0.
\end{equation*}
Outside this event, candidate $i$ has $q_i(A)<\rho-g$ and every other
candidate has $q_j(A)>\rho$, by
\eqref{app:width:population-separation}. The limit fixes $M,d,g$ and
$s$ first and then increases $n$; no joint limiting assumption is needed.

\subsection{Decoding and the width-dependent query lower bound}
\label{app:width:fano-section}

Choose a label $J$ uniformly from $\{1,\ldots,M\}$ and then draw the
matrix under the corresponding prior. Let $\mathcal T_n$ be the full
transcript. Decode $J$ by returning the unique accepted candidate, and
return a fixed index if the acceptance vector has no unique one.
On the complement of the leakage event, simultaneous correctness
forces exactly candidate $J$ to be accepted. Integrating the pointwise
guarantee over the prior therefore gives
\begin{equation}
 p_{e,n}:=\mathbb P\{\widehat J\ne J\}\le\delta+\eta_n.
 \label{app:width:decoder-upper-error}
\end{equation}
This is where completeness as well as safety is used.

For clarity, the needed Fano inequality
\citep{scarlett2019fano} follows from a short entropy calculation. The error indicator is determined by $(J,\mathcal T_n)$.
Conditioning on it gives
\[
 \operatorname{H}(J\mid\mathcal T_n)
 \le\log2+p_{e,n}\log(M-1)
 \le\log2+p_{e,n}\log M.
\]
Since $\operatorname{H}(J)=\log M$, for $M\ge2$ this implies
\begin{equation*}
 p_{e,n}\ge1-
       \frac{\operatorname{I}(J;\mathcal T_n)+\log2}{\log M}.
\end{equation*}
If $\overline{\mathsf P}_n=M^{-1}\sum_j\mathsf P_{j,n}$, convexity of
relative entropy in its second argument gives
\begin{align}
 \operatorname{I}(J;\mathcal T_n)
 &=\frac1M\sum_i
       \operatorname{KL}(\mathsf P_{i,n}\Vert\overline{\mathsf P}_n)
       \notag\\
 &\le\frac1{M^2}\sum_{i,j}
       \operatorname{KL}(\mathsf P_{i,n}\Vert\mathsf P_{j,n})
 \le\frac{sa^2d}{1-a^2}.
 \label{app:width:mutual-information}
\end{align}
Thus~\eqref{app:width:decoder-upper-error}--\eqref{app:width:mutual-information}
produce the finite-dimensional inequality
\begin{equation*}
 \delta+\eta_n
 \ge1-\frac{sa^2d/(1-a^2)+\log2}{\log M}.
\end{equation*}
It holds for every $n$ and the algorithm supplied for that $n$.
Consequently, a dimension-uniform budget must satisfy the same inequality
with $\eta_n$ removed by letting $n\to\infty$.
For $M\ge16$ and $\delta\le1/4$,
\[
 (1-\delta)\log M-\log2\ge\tfrac12\log M,
\]
so
\begin{equation*}
 s\ge\frac{1-a^2}{2a^2d}\log M
   \ge\frac{3}{8a^2d}\log M.
\end{equation*}

For $W\ge16$, take $M=W$ and $d$ from
\eqref{app:width:choose-dimension}. Substitution of $a=16g/w$ and
$d\le24\sqrt{\log W}$ yields the explicit width term
\begin{equation}
 s\ge\frac{w^2}{16384g^2}\sqrt{\log W}.
 \label{app:width:width-term}
\end{equation}
The $R_j$ are distinct projectors of the same rank $d/2+1$. They are
therefore pairwise incomparable: a nonzero positive semidefinite
difference cannot have trace zero. When $n\ge d+2$, the Gaussian
matrix $A$ has full row rank almost surely. If $B$ is any right inverse
with $AB=I_{d+2}$, then
\[
 A^\top R_iA\preceq A^\top R_jA
 ~\Longrightarrow~
 R_i=B^\top A^\top R_iAB\preceq B^\top A^\top R_jAB=R_j.
\]
The reverse implication is immediate by congruence. Thus the input
energy matrices are also an antichain and have width exactly $M=W$.
For the uniform lower bound, the weaker fact $M\le W$ already ensures
the width promise for every matrix, including rank-deficient ones.

\subsection{The confidence term and all integer widths}
\label{app:width:confidence-section}

Use a single residual projector $R=\operatorname{diag}(1,0)$ and
covariances
\begin{equation*}
 \Sigma_0=\operatorname{diag}(\rho+2g,1-\rho-2g),
 ~~
 \Sigma_1=\operatorname{diag}(\rho-2g,1-\rho+2g).
\end{equation*}
Their entries are positive, their traces equal one, and their population
residual ratios lie in the strict null and alternative regions,
respectively. The Gaussian priors and transcript simulation above
apply with $m=2$.

For the variance ratios appearing in either Gaussian relative entropy,
write the ratio as $1+t$. The gap assumption implies
\[
 |t|\le\frac{4g}{w-2g}\le\frac{2}{15}<\frac12.
\]
The identity
$t-\log(1+t)=\int_0^t u/(1+u)\,du$ gives
$t-\log(1+t)\le t^2$ on $|t|\le1/2$.
Using the two diagonal variance ratios in the Gaussian formula,
\begin{equation*}
 \operatorname{KL}\bigl(N(0,\Sigma_1)\Vert N(0,\Sigma_0)\bigr)
 \le\left(\frac{4g}{w-2g}\right)^2
 \le32\frac{g^2}{w^2}.
\end{equation*}
The middle expression follows because the factor $1/2$ in the Gaussian
formula cancels the number of diagonal terms.

Let $p_{h,n}$ be the probability of accepting the candidate under label
$h$. Lemma~\ref{app:width:leakage}, with $M=1$, gives a number
$\eta_n\to0$ such that pointwise correctness implies
\[
 p_{0,n}\le\delta+\eta_n,
 ~~ p_{1,n}\ge1-\delta-\eta_n.
\]
Data processing from the transcript to this binary decision yields
\begin{equation*}
 32s\frac{g^2}{w^2}
 \ge\operatorname{kl}(p_{1,n},p_{0,n}),
 ~~
 \operatorname{kl}(p,q)=p\log\frac pq+(1-p)\log\frac{1-p}{1-q}.
\end{equation*}
For all sufficiently large $n$, $\delta+\eta_n<1/2$.
Binary relative entropy is increasing in its first argument and
decreasing in its second whenever the first exceeds the second.
Therefore the right side is at least
$\operatorname{kl}(1-\delta-\eta_n,\delta+\eta_n)$.
Letting $n\to\infty$ gives
\begin{align*}
 32s\frac{g^2}{w^2}
 \ge\operatorname{kl}(1-\delta,\delta)
 =(1-2\delta)\log\frac{1-\delta}{\delta}
 \ge\frac14\log\frac1\delta,
  0<\delta\le\frac14.
\end{align*}
For the final inequality, $1-2\delta\ge1/2$ and
$1-\delta\ge\sqrt\delta$ on this interval.
It follows that
\begin{equation}
 s\ge\frac{w^2}{128g^2}\log\frac1\delta.
 \label{app:width:confidence-term}
\end{equation}

We now finish the proof of Theorem~\ref{app:width:lower-bound} with an
explicit constant. If $W\ge16$, then
$\sqrt{\log(W+1)}\le2\sqrt{\log W}$. Thus
\eqref{app:width:width-term} bounds $s$ below by
$w^2\sqrt{\log(W+1)}/(32768g^2)$, while
\eqref{app:width:confidence-term} bounds it below by
$w^2\log(1/\delta)/(128g^2)$. The maximum is at least half the sum
with the smaller coefficient, proving~\eqref{app:width:final-bound}.
For $1\le W<16$,
\[
 \sqrt{\log(W+1)}\le\sqrt{\log16}
 \le2\log4\le2\log(1/\delta).
\]
Hence the confidence term alone implies
\eqref{app:width:final-bound}, with substantial slack.
All inequalities apply to integer budgets $s\ge0$ without assuming
$s\to\infty$. In particular, $s=0$ cannot meet the required error
level; an integer version of~\eqref{app:width:final-bound} can be written
by taking the ceiling of its right side.

\paragraph{Scope of the minimax statement.}
The hard family is selected in the supremum over admissible matrices,
projectors, and dimensions. The result neither declares every
width-$W$ family difficult nor asserts this lower bound at each fixed
input dimension: $n$ coordinate queries can recover a fixed
$m\times n$ matrix exactly. It lower-bounds the total counted forward
information needed for simultaneous safety and completeness. A procedure
that receives a construction transcript without charging its responses
uses a different information model. At fixed interior $\rho$, combining
this lower bound with the common-envelope upper bound gives the stated
order in $W,g,\delta$.

\section{Proofs for Spectrum-Dependent Certification Delay}
\label{app:stopping}

This appendix proves the safety statement in
Proposition~\ref{prop:block_sign}, its spectrum-dependent power bound,
and all delay and allocation claims in Theorem~\ref{thm:diffuse_delay}.
The block count is denoted by $L$ throughout, and $\delta$ is the
error-probability level. In particular, $L$ is fixed in the stopping-time
asymptotics; it is not a matrix dimension or a covariance parameter.

\subsection{Notation, stopping rules, and probability space}
\label{app:stopping:setup}

Let $T\ne0$ be positive semidefinite and let $0\preceq E\preceq T$.
Define
\[
 q=\frac{\operatorname{tr}(E)}{\operatorname{tr}(T)},~~
 H=E-\rho T,~~ 0<\rho<1.
\]
Fix $0<\delta<1$. The validation probes $g_1,\ldots,g_s$ are independent
standard Gaussian vectors. Write
\[
 Z_s(B)=\frac1s\sum_{i=1}^s g_i^\top Bg_i,~~
 \widehat q=\frac{Z_s(E)}{Z_s(T)}.
\]
The denominator is positive almost surely. For projection residuals,
$T=A^\top A$ and $E=A^\top(I-P)A$, so
\[
 g_i^\top Hg_i=\|(I-P)Ag_i\|_2^2-\rho\|Ag_i\|_2^2.
\]
All block statistics below are therefore computable from the cached
forward responses $Ag_i$.

For an ordered path, the entire continuation $(E_r)_{r\ge0}$ and all
construction randomness are independent of the validation probes, and
$E_{r+1}\preceq E_r$. One can first condition on that construction
randomness, prove each probability bound for the resulting deterministic
path, and then average. Parameters, partitions, and thresholds are
chosen before observing validation data. The path may be inspected at
sketch-dependent indices, provided it remains the same independent
complete path.

Let
\[
 r_\star=\min\{r:q_r\le\rho\},~~
 D_s=(\tau_s-r_\star)_+,
\]
where $\tau_s$ is the first acceptance position of the specified rule.
All distributions and expectations of $D_s$ below are unconditional.
In particular, they are not conditioned on the absence of an early
false acceptance. The definition of $r_\star$ includes the boundary
$q=\rho$, while the safety calibration controls acceptance even on
that boundary.

For the ratio rule, put $\ell(x)=\log(x/(1-x))$ and
\begin{equation}
 u_s(\delta)=2\operatorname{arcosh}(\delta^{-1/s}),~~
 \theta_s(\delta)
 =\frac{\rho}{\rho+e^{u_s(\delta)}(1-\rho)}.
 \label{app:stopping:ratio-threshold}
\end{equation}
The rule accepts when $\widehat q\le\theta_s(\delta)$. Since
$0<\theta_s(\delta)<\rho$, the log-ratio bound from
Section~\ref{sec:single_certification} gives acceptance probability at
most $\delta$ for every fixed candidate with $q\ge\rho$. For
$q\in\{0,1\}$ the sampled ratio equals $q$ almost surely, so these
endpoints cause no undefined-logarithm issue. Along an independent
ordered path, $\widehat q_r$ decreases and the unsafe-prefix argument
below extends the same probability bound to false stopping anywhere
on the path.

\subsection{Log-concavity and the exact block-sign calibration}
\label{app:stopping:sign-calibration}

For integers $L\ge1$ and $b\ge2$, let $s=Lb$ and partition the probes
into $L$ disjoint blocks $\mathcal I_1,\ldots,\mathcal I_L$ of size $b$.
Set
\begin{equation*}
 X_\ell(H)=\frac1b\sum_{i\in\mathcal I_\ell}g_i^\top Hg_i,
 ~~
 V_L(H)=\sum_{\ell=1}^L\mathbf1\{X_\ell(H)\ge0\}.
\end{equation*}
The cutoff and acceptance rule are
\begin{equation*}
 \begin{split}
 k_\delta(L)
 &=\max\left\{k\in\{0,\ldots,L-1\}:
       \mathbb P\{\operatorname{Bin}(L,e^{-1})\le k\}\le\delta\right\},\\
 \text{accept}&~\Longleftrightarrow~ V_L(H)\le k_\delta(L).
 \end{split}
\end{equation*}
If the defining set is empty, put $k_\delta(L)=-1$ and reject
identically. This convention is useful for safety but is excluded from
the stopping-time and allocation results, which require eventual
acceptance.

We use the standard marginalization consequence of the
Pr\'ekopa--Leindler inequality: integrating a log-concave function over
some coordinates preserves log-concavity \citep{saumard2014}. Applied to $f(x-y)h(y)$,
this shows that convolution preserves log-concavity of densities.
Applied to $f(x+t)\mathbf1\{t\ge0\}$, it shows that a log-concave
density has a log-concave survival function.

\begin{lemma}[Mass above the mean]
\label{app:stopping:mean-halfspace}
If $X$ has a nondegenerate log-concave density and a finite mean, then
$\mathbb P(X\ge\mathbb E X)\ge e^{-1}$.
\end{lemma}
\begin{proof}
Let $S(x)=\mathbb P(X\ge x)$. The function $S$ is log-concave. The
distribution of $X$ is continuous, so the probability integral transform
gives that $S(X)$ is uniform on $(0,1)$. Therefore
\[
 \mathbb E\log S(X)=\int_0^1\log t\,dt=-1.
\]
The mean is in the interior of the interval supporting a nondegenerate
integrable log-concave distribution. Jensen's inequality for the concave
function $\log S$ consequently gives
\[
 \log S(\mathbb E X)\ge\mathbb E\log S(X)=-1.
\]
Exponentiating proves the result.
\end{proof}

\begin{lemma}[Sign probability of a Gaussian block]
\label{app:stopping:quadratic-sign}
Let $H$ be a real symmetric matrix. For $b\ge2$ independent standard
Gaussian vectors, define $X_b=b^{-1}\sum_{i=1}^b g_i^\top Hg_i$.
If $\operatorname{tr}(H)\ge0$, then
$\mathbb P(X_b\ge0)\ge e^{-1}$. The constant cannot be increased
uniformly over matrix dimensions, symmetric $H$, and integers $b\ge2$.
\end{lemma}
\begin{proof}
If $H=0$, then $X_b=0$ and the probability is one. Otherwise,
orthogonal diagonalization and Gaussian rotation give
\[
 X_b\ \overset{d}=\ \sum_j\lambda_j\frac{Y_j}{b},
 ~~ Y_j\sim\chi_b^2\ \text{independently},
\]
where zero eigenvalues are omitted. A $\chi_b^2$ density is Gamma with
shape $b/2\ge1$, hence is log-concave. Scaling by a nonzero coefficient
preserves log-concavity; a negative coefficient merely reflects the
density. Convolution therefore shows that $X_b$ has a nondegenerate
log-concave density. Since $\mathbb E X_b=\operatorname{tr}(H)\ge0$,
Lemma~\ref{app:stopping:mean-halfspace} gives
\[
 \mathbb P(X_b\ge0)
 \ge\mathbb P(X_b\ge\mathbb E X_b)\ge e^{-1}.
\]
For uniform sharpness, take $b=2$ and
$H_n=\operatorname{diag}(1,-1/n,\ldots,-1/n)$ with $n$ negative
coordinates. Its trace is zero, and
\[
 X_2(H_n)\overset{d}=
 \frac{\chi_2^2}{2}-\frac{\chi_{2n}^2}{2n}
 \ \Longrightarrow\ \operatorname{Exp}(1)-1.
\]
The limiting distribution has no atom at zero and has probability
$e^{-1}$ of being nonnegative. This proves the claimed uniform
sharpness, without asserting sharpness for each fixed $b>2$.
\end{proof}

\begin{proof}[Proof of the safety statement in Proposition~\ref{prop:block_sign}]
For a fixed $q\ge\rho$,
$\operatorname{tr}(H)=\operatorname{tr}(T)(q-\rho)\ge0$.
The block statistics are independent and identically distributed. By
Lemma~\ref{app:stopping:quadratic-sign}, the probability $p$ of a
nonnegative block is at least $e^{-1}$. Thus
$V_L(H)\sim\operatorname{Bin}(L,p)$ stochastically dominates
$\operatorname{Bin}(L,e^{-1})$. The definition of $k_\delta(L)$ gives
\[
 \mathbb P\{V_L(H)\le k_\delta(L)\}\le\delta.
\]
The empty-cutoff case rejects deterministically. If $H=0$, all $L$
blocks are nonnegative, and every allowable cutoff is less than $L$,
so the rule also rejects.

For pathwise safety, condition on the independent construction
randomness. The contrast matrices $H_r=E_r-\rho T$ decrease in the
Loewner order. Hence every $X_\ell(H_r)$ and their count $V_L(H_r)$
are nonincreasing in $r$. The true errors $q_r$ are nonincreasing as
well. On any finite prefix $0,\ldots,J$, let $r_J$ be the last index
with $q_{r_J}\ge\rho$, if one exists. Acceptance at any such index is
equivalent to acceptance at $r_J$: the acceptance events are nested
along this unsafe prefix. The single-candidate bound is therefore a
bound of $\delta$ on the union of those events. If no unsafe index
exists, the union is empty.

As $J\to\infty$, these union events increase to acceptance at any
$q_r\ge\rho$ on the countable path. Continuity of probability from
below preserves the bound $\delta$. This argument does not assume a
last unsafe index on the entire infinite path. Averaging over the
conditioned construction randomness completes the proof. The same
argument proves pathwise safety of the ratio rule because its sampled
ratios decrease along the same path.
\end{proof}

\subsection{Finite-sample power without spectral inputs}
\label{app:stopping:power}

\begin{proposition}[Block-sign power]
\label{app:stopping:power-bound}
For a fixed candidate with $\operatorname{tr}(H)=-d<0$, put
$v_H^2=2\operatorname{tr}(H^2)$. Then
\begin{equation}
 \mathbb P\{\mathrm{reject}\}
 \le\mathbb P\!\left\{
   \operatorname{Bin}\!\left(L,
       \frac{v_H^2}{v_H^2+bd^2}\right)>k_\delta(L)\right\}.
 \label{app:stopping:power-binomial}
\end{equation}
For $0<\delta\le1/2$, there are choices of $L$ and $b$ giving both
errors at most $\delta$ with sufficient total budget
\begin{equation}
 s=O\!\left(\log(1/\delta)
       \left[1+\frac{\operatorname{tr}(H^2)}{
                          (\operatorname{tr}H)^2}\right]\right).
 \label{app:stopping:power-complexity}
\end{equation}
The quantities $d$ and $v_H$ characterize when a budget gives power;
they are not inputs to the block-sign decision rule.
\end{proposition}
\begin{proof}
For a symmetric Gaussian quadratic form,
$\operatorname{Var}(g^\top Hg)=2\operatorname{tr}(H^2)$. Consequently
\[
 \mathbb E X_\ell=-d,~~
 \operatorname{Var}(X_\ell)=v_H^2/b.
\]
Cantelli's inequality gives
\[
 p:=\mathbb P(X_\ell\ge0)
 \le\frac{v_H^2/b}{v_H^2/b+d^2}
 =\frac{v_H^2}{v_H^2+bd^2}.
\]
Stochastic monotonicity of binomial distributions proves
\eqref{app:stopping:power-binomial}.

For explicit sufficient choices, write
\[
 \operatorname{kl}(a\|p)
 =a\log(a/p)+(1-a)\log((1-a)/(1-p))
\]
for Bernoulli relative entropy and define the positive constants
\[
 d_0=\operatorname{kl}(1/4\|e^{-1}),~~
 d_1=\operatorname{kl}(1/4\|1/8).
\]
Choose integers satisfying
\begin{equation}
 L\ge\frac{\log(1/\delta)}{\min\{d_0,d_1\}},~~
 b\ge\max\{2,7v_H^2/d^2\}.
 \label{app:stopping:sufficient-blocks}
\end{equation}
A binomial Chernoff bound yields
\[
 \mathbb P\{\operatorname{Bin}(L,e^{-1})\le\lfloor L/4\rfloor\}
 \le e^{-Ld_0}\le\delta.
\]
Thus $k_\delta(L)\ge\lfloor L/4\rfloor$, in particular the cutoff is
nonnegative. The condition on $b$ gives $p\le1/8$. Rejection implies
$V_L>\lfloor L/4\rfloor$, hence $V_L/L>1/4$, and a second Chernoff
bound gives rejection probability at most $e^{-Ld_1}\le\delta$.
Safety holds by Proposition~\ref{prop:block_sign}.

Taking ceilings in~\eqref{app:stopping:sufficient-blocks} proves
\eqref{app:stopping:power-complexity}, with a universal implied
constant. Operationally, one may choose $L$ solely from $\delta$ and,
for a given budget $s$, use $b=\lfloor s/L\rfloor\ge2$, discarding any
remainder. The second inequality in
\eqref{app:stopping:sufficient-blocks} describes when that fixed budget
has the stated power. It does not require estimating either spectral
quantity to set the acceptance threshold.
\end{proof}

\subsection{The common spectral family and its exact block representation}
\label{app:stopping:family-setup}

For positive integers $M,N$, consider
\begin{equation*}
 A_{M,N}=\operatorname{diag}\left(
   \sqrt{\frac{1-\rho}{M}}\,I_M,
   \sqrt{\frac\rho N}\,I_N\right).
\end{equation*}
Its squared Frobenius norm is one. Let $P_r$ project onto the first
$r$ coordinates. This path first captures the initial $M$ directions
and then the $N$ tail directions individually. For $0\le r<M$,
\[
 q(P_r)=1-\frac{r(1-\rho)}M>\rho,
\]
and for $0\le j\le N$,
\begin{equation*}
 q(P_{M+j})=\rho(1-j/N).
\end{equation*}
Thus $r_\star=M$ exactly under the definition $q\le\rho$.
The path is an ordered coordinate path for all $M,N$. It is also the
truncated singular-vector path only when
$(1-\rho)/M\ge\rho/N$; no such additional ordering of singular values
is needed for the results below.

The delay results assume that every coordinate position is inspected
in order, or equivalently that $\tau_s$ is the first position on the
complete path satisfying the acceptance condition. In particular, all
tail positions $j=0,\ldots,N$ are inspected unless acceptance has already
occurred. Inspecting a subset retains the preceding safety guarantee,
but the exact delay constants below do not apply to arbitrary skipped
positions.

Throughout the remaining subsections, fix $\rho\in(0,1)$,
$\delta\in(0,1/2)$, and an integer $L$ for which
$k=k_\delta(L)\ge0$. In particular, $0\le k<L$. Both $L$ and $k$
remain fixed as $M,N$ and the block size vary. The sign rule is the
pure block-sign rule with this cutoff. The ratio rule uses
\eqref{app:stopping:ratio-threshold} at the same level $\delta$.
At the full-rank position, $E=0$, so every sign block equals a strictly
negative multiple of a positive Gaussian energy almost surely. Hence
$V_L=0\le k$. Also $\widehat q=0<\theta_s(\delta)$ for the ratio rule.
Both stopping times are therefore finite, with positive delays at most
$N$.

Number tail positions by $j=0,\ldots,N$ and write
$H_j=A_{M,N}^\top(I-P_{M+j})A_{M,N}-\rho A_{M,N}^\top A_{M,N}$.
Its eigenvalues are
\[
 -\frac{\rho(1-\rho)}M~(M\text{ times}),~~
 -\frac{\rho^2}N~(j\text{ times}),~~
 \frac{\rho(1-\rho)}N~(N-j\text{ times}).
\]
Consequently
\begin{align}
 \operatorname{tr}(H_j)&=-\rho j/N,
 \label{app:stopping:contrast-mean}\\
 \operatorname{tr}(H_j^2)
 &=\frac{\rho^2(1-\rho)^2}M
  +\frac{\rho^2\{(N-j)(1-\rho)^2+j\rho^2\}}{N^2}
 \le\rho^2\left(\frac{(1-\rho)^2}M+\frac1N\right).
 \label{app:stopping:contrast-second-moment}
\end{align}
In particular,
\begin{equation*}
 \operatorname{Var}(X_\ell(H_0))
 =\frac{2\rho^2(1-\rho)^2}{b}\left(\frac1M+\frac1N\right).
\end{equation*}

For each block $\ell$, introduce independent random variables
\[
 C_\ell\sim\frac{\chi_{Mb}^2}{Mb},~~
 Y_{\ell i}\sim\frac{\chi_b^2}{b},~ 1\le i\le N,
 ~~ R_\ell=\frac1N\sum_{i=1}^N Y_{\ell i}.
\]
Different blocks are independent. The dependence of $R_\ell$ on the
$Y_{\ell i}$ is retained. The exact Gaussian block distribution gives
\begin{equation}
 X_\ell(H_j)
 =\rho(1-\rho)(R_\ell-C_\ell)
  -\frac\rho N\sum_{i=1}^jY_{\ell i}.
 \label{app:stopping:block-process}
\end{equation}
The block statistic decreases strictly with $j$ almost surely. Define
its first strict-negative time by
\[
 T_{\ell,N}=\min\{0\le j\le N:X_\ell(H_j)<0\}.
\]
Acceptance requires at least $L-k$ strict-negative blocks, so the first
acceptance time when the tail is inspected is the $(L-k)$th smallest
of $T_{1,N},\ldots,T_{L,N}$. This time equals the unconditional positive
delay $(\tau_s-M)_+$ even if the rule is run before $M$: an acceptance
before $M$ forces acceptance at $M$ by monotonicity, and both quantities
are then zero. No conditioning on successful certification is used in
this identity.

\subsection{The pure sign rule: crossing limit and convergence of means}
\label{app:stopping:sign-limit-proof}

The following strengthened version of the sign part of
Theorem~\ref{thm:diffuse_delay} permits bounded block sizes. This extension
will also be needed to rule out bounded optimal budgets.

\begin{proposition}[Exact block-sign delay limit]
\label{app:stopping:sign-limit}
Assume $M/N\to\kappa\in(0,\infty)$ and let $b=b_N\ge2$ be integers
with $b_N=o(N)$. Put $s=Lb$, $r=L-k$, and let $Z_{(r)}$ be the $r$th
smallest of $L$ independent standard normal variables. Then
\begin{equation}
 \frac{D_s^{\mathrm{sign}}}{\sqrt{N/b}}
 \ \Longrightarrow\
 A_\kappa(Z_{(r)})_+,
 ~~ A_\kappa=(1-\rho)\sqrt{2(1+\kappa^{-1})}.
 \label{app:stopping:sign-distribution}
\end{equation}
The expectations converge as well. Equivalently,
\begin{equation}
 \mathbb E D_s^{\mathrm{sign}}
 \sim C_{\mathrm{sign}}\sqrt{N/s},~~
 C_{\mathrm{sign}}
 =(1-\rho)\sqrt{2L(1+\kappa^{-1})}
       \,\mathbb E[(Z_{(L-k)})_+].
 \label{app:stopping:sign-mean}
\end{equation}
The constant is finite and strictly positive.
\end{proposition}
\begin{proof}
Let $a_N=\sqrt{N/b}\to\infty$. Since $Nb\to\infty$ and
$Mb\to\infty$, the chi-square central limit theorem gives
\[
 \sqrt{Nb}(R_\ell-1)\Longrightarrow\mathcal N(0,2),
 ~~
 \sqrt{Nb}(C_\ell-1)\Longrightarrow\mathcal N(0,2/\kappa).
\]
These two variables are independent for each $\ell$, and different
blocks are independent. Therefore, jointly over the fixed $L$ blocks,
\begin{equation*}
 U_{\ell,N}:=\frac{\sqrt{Nb}}\rho X_\ell(H_0)
 =\sqrt{Nb}(1-\rho)(R_\ell-C_\ell)
 \ \Longrightarrow\ A_\kappa Z_\ell,
\end{equation*}
where the $Z_\ell$ are independent standard normals. This conclusion
does not require $b\to\infty$ because both degrees of freedom already
diverge through $M,N$.

By~\eqref{app:stopping:block-process},
\begin{equation*}
 \frac{\sqrt{Nb}}\rho X_\ell(H_j)
 =U_{\ell,N}-\sqrt{b/N}\sum_{i=1}^jY_{\ell i}.
\end{equation*}
The centered prefix sums are square-integrable martingales and
$\operatorname{Var}(Y_{\ell i})=2/b$. For fixed $R>0$, Doob's $L^2$
maximal inequality yields
\begin{equation}
\label{app:stopping:prefix-negligible}
 \mathbb E\max_{0\le j\le\lfloor Ra_N\rfloor}
 \left|\sqrt{b/N}\sum_{i=1}^j(Y_{\ell i}-1)\right|^2
 \nonumber
 \le
 4\frac bN\lfloor Ra_N\rfloor\frac2b
 \le\frac{8R}{\sqrt{Nb}}\longrightarrow0.
\end{equation}
For sufficiently large $N$ the displayed prefix stays within the tail,
since $a_N/N=1/\sqrt{Nb}\to0$. Also,
$\sqrt{b/N}\lfloor a_Nt\rfloor\to t$ uniformly on compact
$t$-intervals. Thus the scaled signed block process has a Gaussian
initial offset followed by deterministic drift $-t$. No independence
between the initial offset and the removed tail energies is needed:
\eqref{app:stopping:prefix-negligible} makes the centered prefix
negligible even with their exact dependence retained.

For an explicit hitting-time argument, fix $t>0$. For large $N$,
monotonicity gives
\[
 \{T_{\ell,N}/a_N\le t\}
 =\{X_\ell(H_{\lfloor a_Nt\rfloor})<0\}.
\]
The scaled statistic on the right converges to $A_\kappa Z_\ell-t$,
which has no atom at zero. The corresponding probability converges to
$\mathbb P(A_\kappa Z_\ell<t)$. Applying this argument jointly at
positive times $t_1,\ldots,t_L$, or using block independence, proves
\[
 (T_{1,N},\ldots,T_{L,N})/a_N
 \Longrightarrow A_\kappa((Z_1)_+,\ldots,(Z_L)_+).
\]
Positive continuity points suffice to determine this distribution on
the nonnegative orthant; its possible atom at zero is retained by the
positive part. Sorting is a continuous map on $\mathbb R^L$ and
commutes with a nondecreasing scalar function such as the positive
part. The exact order-statistic representation consequently proves
\eqref{app:stopping:sign-distribution}.

To pass to means, let $1\le j\le N$. By
\eqref{app:stopping:contrast-mean} and
\eqref{app:stopping:contrast-second-moment}, Cantelli's inequality gives
\begin{equation*}
 p_j:=\mathbb P\{X_\ell(H_j)\ge0\}
 \le\frac{2N^2}{bj^2}
       \left(\frac{(1-\rho)^2}M+\frac1N\right).
\end{equation*}
Here the Cantelli expression is further bounded by variance divided by
squared mean. Rejection at
position $j$ requires at least $k+1$ nonnegative blocks. Markov's
inequality on their count therefore gives
\begin{equation*}
 \mathbb P(D_s^{\mathrm{sign}}>j)
 \le\min\left\{1,\frac{Lp_j}{k+1}\right\}
 \le\min\left\{1,\frac{C_1N}{bj^2}\right\},
\end{equation*}
for $1\le j\le N$, with $C_1$ independent of $N,b,j$ for large $N$,
because $M/N$ is bounded away from zero. For $j>N$ the probability is
zero. Thus for $t\ge2$ and sufficiently large $N$,
$\lfloor a_Nt\rfloor\ge a_Nt/2$ and
\[
 \mathbb P(D_s^{\mathrm{sign}}/a_N>t)\le4C_1/t^2.
\]
For $R\ge2$, integrating this tail gives the explicit bound
\begin{align*}
 \mathbb E\left[
 \frac{D_s^{\mathrm{sign}}}{a_N}
 \mathbf1\{D_s^{\mathrm{sign}}/a_N>R\}\right]
 =
 R\,\mathbb P(D_s^{\mathrm{sign}}/a_N>R)
 +\int_R^\infty\mathbb P(D_s^{\mathrm{sign}}/a_N>t)\,dt
 \le\frac{8C_1}{R}.
\end{align*}
This proves uniform integrability, rather than inferring it from weak
convergence alone. Consequently the expectations in
\eqref{app:stopping:sign-distribution} converge. Since
$a_N=\sqrt L\sqrt{N/s}$, the resulting constant is exactly
$C_{\mathrm{sign}}$ in~\eqref{app:stopping:sign-mean}.
It is finite because $L$ is fixed and normal variables are integrable.
It is positive, for example, because all $L$ normals have a positive
probability of exceeding one.
\end{proof}

\subsection{The ratio rule on the same diffuse family}
\label{app:stopping:ratio-limit-proof}

\begin{proposition}[Exact ratio-rule delay limit]
\label{app:stopping:ratio-limit}
Assume $M/N\to\kappa\in(0,\infty)$ and use the ratio threshold
\eqref{app:stopping:ratio-threshold}. If $s=s_N\to\infty$ and
$s_N=o(N^2)$, then
\begin{equation}
 \frac{\sqrt s}{N}D_s^{\mathrm{ratio}}
 \longrightarrow C_{\mathrm{ratio}}
 =(1-\rho)\sqrt{8\log(1/\delta)}
 ~\text{in }L^2.
 \label{app:stopping:ratio-distribution}
\end{equation}
In particular,
$\mathbb E D_s^{\mathrm{ratio}}
 \sim C_{\mathrm{ratio}}N/\sqrt s$.
For every fixed integer $s\ge1$,
\begin{equation}
 D_s^{\mathrm{ratio}}/N
 \longrightarrow1-\theta_s(\delta)/\rho>0
 ~\text{in }L^2.
 \label{app:stopping:ratio-fixed-budget}
\end{equation}
\end{proposition}
\begin{proof}
Let $Y_1,\ldots,Y_N\sim\chi_s^2/s$ be independent and let
$C_N\sim\chi_{Ms}^2/(Ms)$ be independent of them. Set
\[
 B_j=\frac1N\sum_{i=1}^jY_i,~~ B_0=0.
\]
At tail position $j$, the exact Gaussian ratio is
\begin{equation*}
 \widehat q(P_{M+j})
 =\frac{\rho(B_N-B_j)}{(1-\rho)C_N+\rho B_N}.
\end{equation*}
Abbreviate $\theta_s=\theta_s(\delta)$. Acceptance is equivalent to
$B_j\ge h_N$, where
\begin{equation*}
 h_N=(1-\theta_s)B_N-
       \frac{\theta_s(1-\rho)}\rho C_N.
\end{equation*}
Since
$B_N-h_N=\theta_s\{B_N+(1-\rho)C_N/\rho\}>0$ almost surely,
the first crossing
\[
 j_N=\min\{0\le j\le N:B_j\ge h_N\}
\]
exists. As explained for the sign rule, monotonicity makes this tail
crossing time equal to the unconditional positive delay
$D_s^{\mathrm{ratio}}$, including sample paths that accepted before
$M$.

Let $\Delta_N=\max_{0\le j\le N}|B_j-j/N|$. The following bound is
deterministic:
\begin{equation}
 \left|\frac{j_N}N-(h_N)_+\right|
 \le\Delta_N+\frac1N.
 \label{app:stopping:ratio-crossing-error}
\end{equation}
If $h_N\le0$, both $j_N$ and $(h_N)_+$ are zero. Otherwise, minimality
gives $B_{j_N-1}<h_N\le B_{j_N}$, and therefore
\[
 h_N-\Delta_N\le j_N/N<h_N+\Delta_N+1/N.
\]
For the centered prefix martingale, Doob's inequality gives
\begin{equation}
 \mathbb E\Delta_N^2
 \le4\operatorname{Var}(B_N)=\frac8{Ns}.
 \label{app:stopping:ratio-prefix-bound}
\end{equation}
The crossing level has mean and variance
\begin{align}
 \mathbb E h_N&=1-\theta_s/\rho,
 \label{app:stopping:ratio-level-mean}\\
 s\operatorname{Var}(h_N)
 &=\frac{2(1-\theta_s)^2}N
   +\frac2M\left(\frac{\theta_s(1-\rho)}\rho\right)^2
 \longrightarrow0.
 \label{app:stopping:ratio-level-variance}
\end{align}
The last limit is uniform over $s\ge1$ because $0<\theta_s<\rho$ and
$M/N\to\kappa>0$.

Put $\lambda=\log(1/\delta)>0$. For $s\to\infty$,
\[
 \sqrt{s}\,u_s(\delta)\longrightarrow\sqrt{8\lambda}.
\]
Expanding the error-scale threshold at $u=0$ gives
\[
 1-\theta_s/\rho=(1-\rho)u_s(\delta)+O(u_s(\delta)^2),
\]
so $\sqrt{s}\,\mathbb E h_N\to C_{\mathrm{ratio}}$. Together with
\eqref{app:stopping:ratio-level-variance}, this proves
$\sqrt{s}h_N\to C_{\mathrm{ratio}}$ in $L^2$.
The positive-part map is $1$-Lipschitz and
$C_{\mathrm{ratio}}>0$, so the same is true of
$\sqrt{s}(h_N)_+$. Multiplying
\eqref{app:stopping:ratio-crossing-error} by $\sqrt s$, using
\eqref{app:stopping:ratio-prefix-bound}, and using $\sqrt{s}/N\to0$
now proves~\eqref{app:stopping:ratio-distribution}. In this case
convergence of means follows from the stronger $L^2$ convergence.

Finally, for a fixed $s$, equations
\eqref{app:stopping:ratio-level-mean} and
\eqref{app:stopping:ratio-level-variance} give
$h_N\to1-\theta_s/\rho>0$ in $L^2$, while
\eqref{app:stopping:ratio-prefix-bound} gives $\Delta_N\to0$ in $L^2$.
The same crossing bound proves
\eqref{app:stopping:ratio-fixed-budget}.
\end{proof}

When the two rules are compared at a common budget $s=Lb\to\infty$
with fixed $L$ and $b=o(N)$, all conditions of both delay propositions
hold. Their means are precisely those stated in
\eqref{eq:diffuse_delay_asymptotics}. The sign statement itself permits
bounded $b$, whereas the ratio statement permits the larger growing
budget range $s=o(N^2)$.

\subsection{Global optimization over integer validation budgets}
\label{app:stopping:allocation-proof}

For fixed $c_{\mathrm v},c_{\mathrm b}>0$, define
\[
 \mathcal C_N^{\mathrm{sign}}(s)
 =c_{\mathrm v}s+c_{\mathrm b}\mathbb E D_s^{\mathrm{sign}},
 ~~
 \mathcal C_N^{\mathrm{ratio}}(s)
 =c_{\mathrm v}s+c_{\mathrm b}\mathbb E D_s^{\mathrm{ratio}}.
\]
The sign optimization is over
$\mathcal S_L=\{Lb:b\in\mathbb N,\ b\ge2\}$, with $L$ and
$k=k_\delta(L)$ fixed. The ratio optimization is over all positive
integers $s$. A minimum exists in each case: the cost is finite at
every allowed budget because $D_s\le N$, and it is at least
$c_{\mathrm v}s$, which diverges with $s$.

\begin{proposition}[Exact global allocations]
\label{app:stopping:allocation}
Let $C_{\mathrm{sign}}$ and $C_{\mathrm{ratio}}$ be as above and put
\[
 a_{\mathrm{sign}}=
 \left(\frac{c_{\mathrm b}C_{\mathrm{sign}}}{2c_{\mathrm v}}\right)^{2/3},
 ~~
 a_{\mathrm{ratio}}=
 \left(\frac{c_{\mathrm b}C_{\mathrm{ratio}}}{2c_{\mathrm v}}\right)^{2/3}.
\]
Every sequence of global integer minimizers satisfies
\begin{align}
 s_N^{\mathrm{sign},*}&\sim a_{\mathrm{sign}}N^{1/3},&
 \min_{s\in\mathcal S_L}\mathcal C_N^{\mathrm{sign}}(s)
 &\sim3c_{\mathrm v}a_{\mathrm{sign}}N^{1/3},
 \nonumber\\
 s_N^{\mathrm{ratio},*}&\sim a_{\mathrm{ratio}}N^{2/3},&
 \min_{s\in\mathbb N,\ s\ge1}\mathcal C_N^{\mathrm{ratio}}(s)
 &\sim3c_{\mathrm v}a_{\mathrm{ratio}}N^{2/3}.
 \label{app:stopping:ratio-optimum}
\end{align}
These statements keep the matrix family, coordinate path, target,
confidence level, and the respective rule fixed. They concern the
excess costs defined above, excluding construction of the initial $M$
directions, and are not minimax lower bounds over all possible
certificates or jointly adaptive budget-allocation schemes.
\end{proposition}
\begin{proof}
First consider the sign rule. For any fixed $x>0$, select
$s_N\in\mathcal S_L$ with $s_N/N^{1/3}\to x$; rounding to a multiple
of fixed $L$ has asymptotically negligible effect. Then
$b_N=s_N/L\to\infty$ and $b_N=o(N)$, so
Proposition~\ref{app:stopping:sign-limit} gives
\begin{equation}
 \frac{\mathcal C_N^{\mathrm{sign}}(s_N)}{N^{1/3}}
 \longrightarrow f_{\mathrm{sign}}(x)
 :=c_{\mathrm v}x+c_{\mathrm b}C_{\mathrm{sign}}x^{-1/2}.
 \label{app:stopping:sign-limiting-cost}
\end{equation}
In particular, the global minimum is $O(N^{1/3})$. For any minimizing
sequence, its validation cost implies
$s_N^{\mathrm{sign},*}=O(N^{1/3})$, hence
$b_N^*=s_N^{\mathrm{sign},*}/L=o(N)$.

The block sizes must also diverge. Otherwise a bounded subsequence
of the integer $b_N^*$ has a constant further subsequence. The
bounded-$b$ case of Proposition~\ref{app:stopping:sign-limit} gives
$\mathbb E D_{Lb_N^*}^{\mathrm{sign}}=\Theta(\sqrt N)$ on that
subsequence, with a positive constant. This contradicts the feasible
$O(N^{1/3})$ cost. The mean asymptotic thus applies along every
minimizing sequence.

Write $x_N=s_N^{\mathrm{sign},*}/N^{1/3}$. It is bounded above. If a
subsequence had $x_N\to0$, the mean asymptotic would imply
\[
 \frac{\mathcal C_N^{\mathrm{sign}}(s_N^{\mathrm{sign},*})}{N^{1/3}}
 \ge c_{\mathrm b}(C_{\mathrm{sign}}+o(1))x_N^{-1/2}
 \longrightarrow\infty,
\]
again contradicting the feasible bound. Thus $x_N$ stays in a compact
subset of $(0,\infty)$. Along any subsequence with $x_N\to x$, the
normalized cost converges to $f_{\mathrm{sign}}(x)$. Comparing with
the feasible sequences in~\eqref{app:stopping:sign-limiting-cost}
forces $x$ to minimize $f_{\mathrm{sign}}$. This limiting function has
the unique minimizer
\[
 x=a_{\mathrm{sign}}
 =\left(\frac{c_{\mathrm b}C_{\mathrm{sign}}}{2c_{\mathrm v}}\right)^{2/3}.
\]
Every subsequential limit is therefore $a_{\mathrm{sign}}$, proving
the budget asymptotic. Since
$f_{\mathrm{sign}}(a_{\mathrm{sign}})
 =3c_{\mathrm v}a_{\mathrm{sign}}$, the cost asymptotic follows.

For the ratio rule, use instead feasible integers $s_N\sim xN^{2/3}$.
Proposition~\ref{app:stopping:ratio-limit} gives
\begin{equation*}
 \frac{\mathcal C_N^{\mathrm{ratio}}(s_N)}{N^{2/3}}
 \longrightarrow f_{\mathrm{ratio}}(x)
 :=c_{\mathrm v}x+c_{\mathrm b}C_{\mathrm{ratio}}x^{-1/2}.
\end{equation*}
The minimum is $O(N^{2/3})$, and every minimizing sequence obeys
$s_N^{\mathrm{ratio},*}=O(N^{2/3})=o(N^2)$.
It cannot remain bounded on a subsequence: a constant integer further
subsequence would have $\mathbb E D_s^{\mathrm{ratio}}=\Theta(N)$ by
\eqref{app:stopping:ratio-fixed-budget}, contradicting that feasible
cost. Thus $s_N^{\mathrm{ratio},*}\to\infty$, and the mean asymptotic
in Proposition~\ref{app:stopping:ratio-limit} applies to every minimizing
sequence.

The normalized widths $y_N=s_N^{\mathrm{ratio},*}/N^{2/3}$ are bounded
above and cannot approach zero, since their normalized delay costs
would then diverge as
$c_{\mathrm b}(C_{\mathrm{ratio}}+o(1))/\sqrt{y_N}$.
The same compact-subsequence comparison now forces every limit to be
the unique minimizer
\[
 y=a_{\mathrm{ratio}}
 =\left(\frac{c_{\mathrm b}C_{\mathrm{ratio}}}{2c_{\mathrm v}}\right)^{2/3}
\]
of $f_{\mathrm{ratio}}$. Evaluating this function there gives
$3c_{\mathrm v}a_{\mathrm{ratio}}$, completing
\eqref{app:stopping:ratio-optimum}.
\end{proof}

The two global allocation conclusions establish the cost and budget
orders in Theorem~\ref{thm:diffuse_delay}. Their proofs first locate all
minimizers in the appropriate asymptotic regime; simply differentiating
a local delay approximation would not establish these global claims.
When $M\asymp N$, constructing the initial $M$ directions can dominate
the reported excess cost and must be restored when discussing total
runtime or total construction complexity.

\subsection{Transfer to mean-square output error}
\label{app:stopping:functional}

Let the input second moment $K\succeq0$ be known and put
$B=AK^{1/2}\ne0$, as in Corollary~\ref{cor:output_certification}.
For each candidate define $T_K=B^\top B$, $E_{K,r}=B^\top(I-P_r)B$,
and $H_{K,r}=E_{K,r}-\rho T_K$. The block statistic is computed by
replacing each cached response $Ag_i$ with $A(K^{1/2}g_i)=Bg_i$.
Each such response still costs one forward query to $A$.
The safety and power proofs above apply to these matrices directly;
their assumptions use positivity, ordering, and Gaussian probes for
the effective operator $B$. They impose no Gaussian assumption on
the actual input distribution. The delay and allocation conclusions
apply when this effective operator and its projection path have the
spectral-family form analyzed above. This is a condition on $B$, and
does not assert the same delay on every matrix for each prescribed
$K$. The choice $K=I$ recovers the stated diffuse family exactly.

\section{Experimental Details and Additional Results}
\label{app:experiments}

This appendix specifies the synthetic operators, certification rules,
sampling procedures, and statistical conventions used in the numerical
experiments. All random candidates are fixed independently of validation.
The simulations evaluate the relative squared Frobenius criterion using
the exact distribution of Gaussian response statistics. Large operators
need not be materialized, and computational timings are not used as a
measure of certification cost. The width experiment varies effective
dimension together with family size; a separate repeated-check
experiment isolates reuse along fixed paths. All logarithms are natural
unless a base is explicitly indicated.

\subsection{Single-candidate confidence exponent}
\label{app:exp_single}

For $\Delta\in\{1,2\}$, we evaluate the lower bound
$p_s(\Delta/2)=\mathbb P\{F_{s,s}\le e^{-\Delta/2}\}$ and the upper bound
$\exp(-sJ)$, where $J=\log\cosh(\Delta/4)$.
The query grid contains 81 geometrically spaced values from 8 to 4096,
rounded to integers, together with $2^j$ for $j=3,\ldots,12$.
Removing duplicates leaves 87 budgets per value of $\Delta$.
These calculations are deterministic and require no Monte Carlo trials.

For a numerical cross-check, the same tail probability is computed from
the regularized incomplete beta function:
\[
    p_s(\Delta/2)
    =I_{(1+e^{\Delta/2})^{-1}}(s/2,s/2).
\]
The maximum discrepancy between the two log-probability evaluations is
$3.98\times10^{-13}$ over the full grid.
Figure~\ref{fig:exp_single_supp} displays the finite-query bounds and
their normalized exponents. At $s=4096$, the normalized exponents are
$1.029007$ for $\Delta=1$ and $1.008754$ for $\Delta=2$.
Their convergence to one concerns the leading exponential rate;
it does not imply that the two finite-query risk bounds have ratio
converging to one.

\begin{figure}[htbp]
    \centering
    \includegraphics[width=0.8\linewidth]{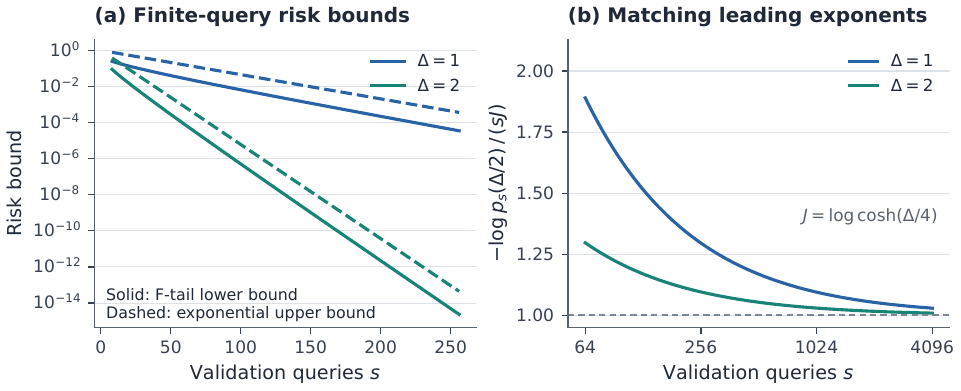}
    \caption{Single-candidate bounds.
    Left: the F-tail lower bound (solid) and exponential upper bound
    (dashed). Right: $-\log p_s(\Delta/2)/(sJ)$, with reference value one.
    The right panel displays budgets $s\ge64$ for readability.}
    \label{fig:exp_single_supp}
\end{figure}

\subsection{The square-root logarithmic width mechanism}
\label{app:exp_width_mechanism}

This experiment varies the effective dimension together with the
number of fixed candidates. It measures simultaneous decisions made by
the Gaussian ratio certificate; it does not experimentally establish a
lower bound for arbitrary adaptive query algorithms.

\paragraph{A fixed finite family with exponentially many paths.}
Let $\mathcal B_d$ be the set of all symmetric $d\times d$ matrices
with zero diagonal and off-diagonal entries in $\{-1,1\}$. The set
contains $W=2^{d(d-1)/2}$ matrices and is fixed before validation.
For every $B\in\mathcal B_d$, define
\[
 D_B=\frac{B}{\max\{c\sqrt d,\|B\|_{\mathrm{op}}\}},~~
 K_{B,0}=\rho I_d+aD_B,~~ K_{B,1}=K_{B,0}-gI_d,
\]
where $c=3$, $a=1/4$, $\rho=1/2$, and $g=0.01$.
Set $T=I_d/d$ and $E_{B,j}=K_{B,j}/d$.
Since $\|D_B\|_{\mathrm{op}}\le1$ and $\operatorname{tr}D_B=0$,
$0\preceq E_{B,1}\preceq E_{B,0}\preceq T$ and the true ratios
are $q_{B,0}=\rho$ and $q_{B,1}=\rho-g$.
Each $B$ therefore defines a two-candidate ordered path containing
both required decision boundaries. Distinct $B$ give distinct
$E_{B,0}$ of the same trace, hence an antichain of size $W$.
Together with the displayed chain cover, this proves that the full
family has width exactly $W$.

\paragraph{Realization by one operator and nested projections.}
The energy construction corresponds to actual orthogonal output
projections, rather than to arbitrary residual quadratic forms.
Let $J=[I_d,0,0]^\top\in\mathbb R^{3d\times d}$ and
$A=J/\sqrt d$. For each $B$, the matrix
\[
 V_B=\begin{bmatrix}
 K_{B,1}^{1/2}\\ \sqrt g I_d\\ (I_d-K_{B,0})^{1/2}
 \end{bmatrix}
\]
is an isometry. Extend it to an orthogonal matrix $O_B$ with
$O_BJ=V_B$, and put
$R_{B,0}=O_B^\top\operatorname{diag}(I_d,I_d,0)O_B$ and
$R_{B,1}=O_B^\top\operatorname{diag}(I_d,0,0)O_B$.
Then $R_{B,1}\preceq R_{B,0}$ and
$A^\top R_{B,j}A=E_{B,j}$.
The approximation projectors are $P_{B,j}=I_{3d}-R_{B,j}$.
An arbitrarily large ambient input dimension can be introduced by
replacing $A$ with $AV^\top$, where $V^\top V=I_d$ is fixed
independently of validation; this leaves all simulated statistics
unchanged. The construction illustrates the behavior of Gaussian
validation, not oracle hardness of this particular structured operator.

\paragraph{An exactly checkable maximum without enumeration.}
Generate the actual Gaussian sample covariance
$S=GG^\top/s$, with $G\in\mathbb R^{d\times s}$ standard Gaussian.
Both positions of path $B$ have the same estimation offset,
\[
 \widehat q_{B,j}-q_{B,j}
 =\frac{a\operatorname{tr}(D_BS)}{\operatorname{tr}S}.
\]
Write $M(S)$ for the maximum absolute offset over $\mathcal B_d$.
Let $B_*(S)$ have off-diagonal entries
$\operatorname{sign}(S_{ij})$ and zero diagonal, assigning either
sign at zero, an event of probability zero. Deterministically,
\[
 U(S)\min\left\{1,\frac{c\sqrt d}{\|B_*(S)\|_{\mathrm{op}}}\right\}
 \le M(S)\le U(S),~~
 U(S)=\frac{a\sum_{i\ne j}|S_{ij}|}
              {c\sqrt d\operatorname{tr}S}.
\]
The upper bound follows by the triangle inequality and the lower
bound by evaluating the member $B_*(S)$ of the fixed family.
If $\|B_*(S)\|_{\mathrm{op}}\le c\sqrt d$, the bounds coincide
and give the exact maximum. Selecting its maximizing index after
validation does not change the pre-existing candidate family.
The family contains $B$ and $-B$. Thus, for the midpoint rule
$\widehat q\le\rho-g/2$, any incorrect required decision occurs
exactly when $M(S)\ge g/2$, up to zero-probability ties.

Every draw retains both bounds and checks the attainment condition.
If attainment were to fail, the implementation would retain an
interval for the error event rather than discard or replace that draw.
All 749568 draws across the experiments below satisfy attainment;
the largest observed ratio
$\|B_*(S)\|_{\mathrm{op}}/(c\sqrt d)$ is $0.934712$.
An additional exhaustive check at $d=4$, with all 64 matrices and
100 direct Gaussian sketches, verifies the maximum formula.
The same implementation explicitly verifies the projection dilation.

\paragraph{Why the construction exposes the width mechanism.}
There are order $d^2$ off-diagonal entries. Their typical individual
fluctuation is of order $s^{-1/2}$, while the denominator in $U(S)$
has scale $d^{3/2}$. This gives the error scale
$\sqrt{d/s}\asymp(\log W)^{1/4}/\sqrt s$.
Meeting a fixed additive gap therefore suggests a budget of order
$d/g^2\asymp\sqrt{\log W}/g^2$.
As a distributional check, the upper envelope has the exact mean
\[
 \mathbb E U(S)=
 \frac{2a(d-1)}{c\sqrt{\pi d}\,s}
 \frac{\Gamma((s+1)/2)}{\Gamma(s/2)}.
\]
Indeed, $\operatorname{tr}(GG^\top)$ is independent of the normalized
Gaussian direction, and conditioning on one row gives
$\mathbb E|(GG^\top)_{ij}|=
2\Gamma((s+1)/2)/(\sqrt\pi\Gamma(s/2))$.
The mean identity concerns $U(S)$ unconditionally; it is not a claim
that the attainment condition holds for every possible sketch.

\paragraph{Sampling and budget selection.}
Use $d\in\{8,12,16,24,32,48,64,96,128\}$.
The fixed-budget experiment uses $s=4096$ and 4096 independent
trials per dimension. The covariance is generated by the exact
Bartlett factorization: for lower triangular $L$, its subdiagonal
entries are independent standard normals and
$L_{ii}^2\sim\chi^2_{s-i+1}$ independently; set $S=LL^\top/s$.
Every budget satisfies $s\ge d$. Neither asymptotic error nor delay
laws are used to generate observations.

An initial dimension-independent budget grid consists of the distinct
integers $\lceil64(1.18)^j\rceil$, $j=0,\ldots,37$, retaining $s\ge d$.
Each budget receives 768 screening trials. Select the first whose
empirical upper failure rate is at most $0.1$ and evaluate it and
its preceding grid point on 4096 fresh trials.
This coarse grid revealed a transition too sharp to resolve at factor
$1.18$, so a documented refinement uses 17 logarithmically spaced,
rounded budgets between those two grid points.
All original observations are retained.
Each refined budget receives 1536 fresh screening trials; select the
smallest whose pointwise Wilson 95\% upper endpoint is at most $0.1$.
That budget and its immediate predecessor receive 8192 new evaluation
trials each. The refinement does not use these evaluation outcomes.
Pointwise screening intervals are not simultaneous guarantees for
the selected budget.

\begin{table}[t]
\centering
\caption{Width experiment: independently evaluated, screening-selected
Gaussian budgets. Error intervals are pointwise Wilson 95\% intervals
conditional on the selected budget. Here $\log_2 W=d(d-1)/2$.}
\label{tab:width_mechanism}
\small
\begin{tabular}{rrrrr}
\toprule
$d$ & $\log_2 W$ & Selected $s$ & Error (\%) & Interval (\%)\\
\midrule
8   & 28   & 1549  & 8.997  & [8.396, 9.636] \\
12  & 66   & 2295  & 7.715  & [7.157, 8.313] \\
16  & 120  & 2941  & 10.718 & [10.066, 11.406] \\
24  & 276  & 4402  & 8.533  & [7.947, 9.157] \\
32  & 496  & 5881  & 6.348  & [5.840, 6.896] \\
48  & 1128 & 8714  & 6.519  & [6.004, 7.074] \\
64  & 2016 & 11521 & 6.775  & [6.251, 7.340] \\
96  & 4560 & 17247 & 4.626  & [4.192, 5.103] \\
128 & 8128 & 23040 & 2.319  & [2.015, 2.668] \\
\bottomrule
\end{tabular}
\end{table}

The reported budgets are selected from a finite grid, rather than
being exact minimum budgets at risk $0.1$.
Independent evaluation risks range from $2.32\%$ to $10.72\%$;
in particular, the point estimate at $d=16$ exceeds the screening
target and is retained. The steepening transition makes even the
refined grid conservative at larger dimensions.
Over the largest five dimensions, ordinary least-squares slopes
against $\log W$ in log--log coordinates are $0.488268$ for selected
budgets, $0.256384$ for mean fixed-budget maximum error, and
$0.244699$ for its 90th percentile.
The corresponding slopes over the largest four dimensions are
$0.492623$, $0.254907$, and $0.245554$; fits over all dimensions
are also retained.
These are descriptive exponents, not confidence intervals.
The reference power curves are anchored at $d=32$ and are not
theoretical finite-sample thresholds.

\begin{figure}[t]
    \centering
    \includegraphics[width=\linewidth]{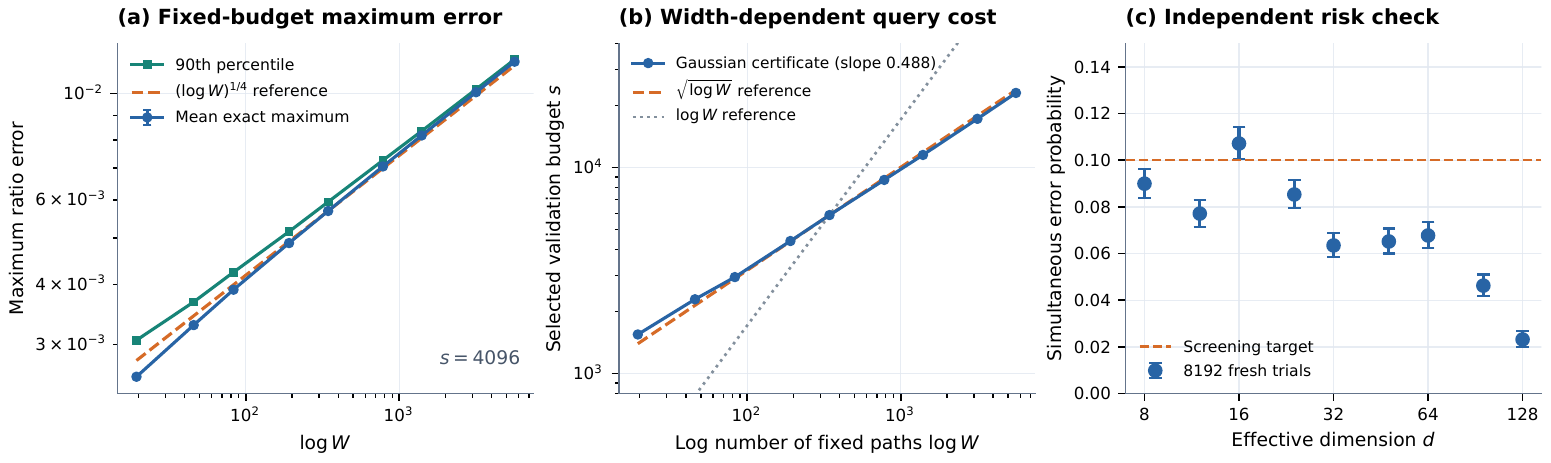}
    \caption{The width mechanism on fixed finite families.
    (a) Mean and 90th-percentile maximum error at $s=4096$;
    mean error bars show $\pm2$ Monte Carlo standard errors.
    (b) Selected query budgets compared with square-root logarithmic
    and logarithmic reference slopes.
    (c) Fresh evaluation of the selected budgets, with pointwise Wilson
    95\% intervals and the screening target marked.
    The family width is $W=2^{d(d-1)/2}$ throughout.}
    \label{fig:width_mechanism}
\end{figure}

\subsection{Validation reuse across nested paths}
\label{app:exp_reuse}

\paragraph{Candidate families.}
Fix effective dimension $d=24$, ambient input dimension $n=65536$,
$\beta=0.4$, $\rho=0.5$, and $g=0.01$.
For each family, generate 4096 independent rank-$d/2$ orthogonal
projections $Q_i=U_iU_i^\top$, where $U_i$ is obtained by reduced QR
factorization of a $d\times(d/2)$ standard Gaussian matrix.
All projections are generated before any validation samples.
Widths $W\in\{1,4,16,64,256,1024,4096\}$ use prefixes of this fixed family.

For $J_{\mathrm{chk}}\in\{4,32,256\}$, the checked error levels consist of
$J_{\mathrm{chk}}-2$ equally spaced points in $[0.25,0.75]$, together
with $\rho=0.5$ and $\rho-g=0.49$, sorted in decreasing order.
Their ordered union has $K=284$ levels $q_1>\cdots>q_K$.
Set $a_k=q_k-\beta/2$ and define the effective energy matrices
\[
    T_{\mathrm{eff}}=I_d/d,
    ~~
    E_{i,k}^{\mathrm{eff}}=(a_kI_d+\beta Q_i)/d.
\]
They satisfy $0\preceq E_{i,k}^{\mathrm{eff}}\preceq T_{\mathrm{eff}}$
and have true relative error $q_k$.

\paragraph{Realization by nested output projections.}
The energy matrices above arise from one fixed operator and actual
orthogonal output projections. Let $a_1=0.55$, $a_K=0.05$, and form
the nonnegative block weights
\[
    w=(\beta,a_K,a_1-a_2,\ldots,a_{K-1}-a_K,1-\beta-a_1).
\]
These weights sum to one.
Construct $A_{\mathrm{eff}}$ by vertically stacking
$\sqrt{w_t/d}\,I_d$ for $t=1,\ldots,K+2$.
Fix any $V\in\mathbb R^{n\times d}$ with $V^\top V=I_d$, independently
of validation, and set $A=A_{\mathrm{eff}}V^\top$.
The output dimension is $(K+2)d=6864$.

On path $i$, the residual projection in the first block is $Q_i$;
the second block remains residual, and the last block is always captured.
The intermediate difference blocks are captured successively.
At level $k$, the blocks with weights $a_t-a_{t+1}$ for $t<k$ have
been captured and those for $t\ge k$ remain residual.
The resulting approximation projections $P_{i,k}$ are orthogonal
and increase along each path, with
\[
    A^\top A=VT_{\mathrm{eff}}V^\top,
    ~~
    A^\top(I-P_{i,k})A=VE_{i,k}^{\mathrm{eff}}V^\top.
\]
Different paths generally do not commute. At a common level, distinct
residual energy matrices have equal trace and are therefore incomparable
in Loewner order. Together with the $W$-path cover, this gives width
exactly $W$ almost surely.

\paragraph{Shared Gaussian responses.}
Since $V^\top g\sim\mathcal N(0,I_d)$ for an ambient standard Gaussian
probe, all required statistics can be generated in effective coordinates.
For each trial, draw $S=GG^\top/s$, where $G\in\mathbb R^{d\times s}$
has independent standard Gaussian entries, using the equivalent Bartlett
factorization $S=BB^\top/s$.
Here $B$ is lower triangular, $B_{aa}^2\sim\chi^2_{s-a+1}$ for
$a=1,\ldots,d$, and its strictly lower triangular entries are independent
standard Gaussians; all these entries are mutually independent.
Every tested budget satisfies $s\ge d$.
The same $S$ is used for all candidates and widths within a trial.

The empirical ratio retains its random denominator and satisfies
\[
    \widehat q_{i,k}=q_k+\eta_i,
    ~~
    \eta_i=\beta\left(
      \frac{\operatorname{tr}(Q_iS)}{\operatorname{tr}S}-\frac12
    \right).
\]
Acceptance uses $\widehat q\le\rho-g/2=0.495$.
Because each grid contains both decision frontiers, a simultaneous error
is exactly the event
\[
    \left\{\min_{i\le W}\eta_i\le-g/2\right\}
    \ \cup\
    \left\{\max_{i\le W}\eta_i>g/2\right\}.
\]
Almost surely, this equals $\{\max_{i\le W}|\eta_i|>g/2\}$.
The offset is constant along a path, so adding intermediate checks while
retaining the frontiers leaves the error event unchanged. This additive
offset structure is specific to the experimental family; the theoretical
boundary reduction requires only the Loewner ordering.

For all trials, widths, and check counts, the grid-specific frontier
computation agrees with the common frontier event, giving 258048
consistent comparisons from 12288 independent validation draws.
We additionally evaluate every decision on the complete grids for the
first eight trials and first sixteen paths of each family--budget
setting; these checks also agree. Repeated checks are not additional
independent trials.

\paragraph{Settings and results.}
The budgets are $s\in\{512,1024,2048,4096\}$.
Three independently generated candidate families are retained.
The primary family uses 2048 trials per budget, and the other two use
512 trials each. Table~\ref{tab:exp_reuse_counts} reports all error counts
for the primary family; the three check counts produce identical entries.
Table~\ref{tab:exp_reuse_families} compares the families at
$s=2048$, $W=4096$.
Figure~\ref{fig:exp_reuse_supp} additionally reports the empirical 95th
percentile of $\max_{i\le W}|\eta_i|$.
This percentile describes the sampled error distribution and is not a
confidence upper bound.

\begin{table}[t]
\centering
\caption{Primary-family simultaneous error counts out of 2048 trials
per budget, identical for all three check grids. The last column is
the empirical 95th percentile of maximum absolute error at $s=2048$.}
\label{tab:exp_reuse_counts}
\small
\setlength{\tabcolsep}{5pt}
\begin{tabular}{r rrrr r}
\toprule
$W$ & $s=512$ & $s=1024$ & $s=2048$ & $s=4096$ & Error percentile\\
\midrule
1    & 82   & 15   & 0   & 0 & 0.002509 \\
4    & 366  & 43   & 1   & 0 & 0.003137 \\
16   & 1123 & 187  & 4   & 0 & 0.003767 \\
64   & 1944 & 622  & 11  & 0 & 0.004267 \\
256  & 2048 & 1517 & 43  & 0 & 0.004746 \\
1024 & 2048 & 2025 & 154 & 0 & 0.005131 \\
4096 & 2048 & 2048 & 568 & 0 & 0.005550 \\
\bottomrule
\end{tabular}
\end{table}

\begin{table}[t]
\centering
\caption{Independent fixed families at $s=2048$, $W=4096$.
Error probabilities and pointwise Wilson 95\% intervals are percentages.}
\label{tab:exp_reuse_families}
\small
\setlength{\tabcolsep}{5pt}
\begin{tabular}{r rrrr}
\toprule
Family & Trials & Failures & Error (\%) & Interval (\%)\\
\midrule
1 & 2048 & 568 & 27.734 & [25.839, 29.714] \\
2 & 512  & 156 & 30.469 & [26.640, 34.589] \\
3 & 512  & 139 & 27.148 & [23.477, 31.160] \\
\bottomrule
\end{tabular}
\end{table}

\begin{figure}[t]
    \centering
    \includegraphics[width=0.8\linewidth]{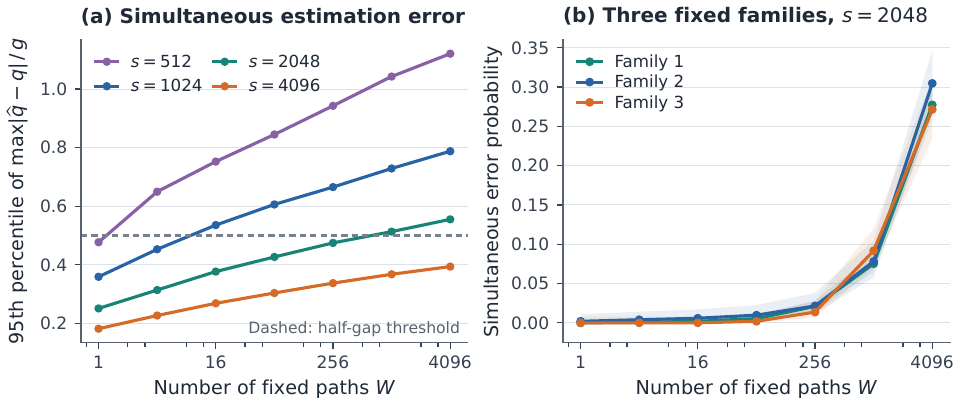}
    \caption{Additional validation-reuse results.
    Left: empirical 95th percentile of the maximum absolute ratio error,
    divided by $g$, for the primary family; the dashed line marks the
    half-gap decision margin. No interval for the percentile is shown.
    Right: simultaneous error probabilities for all three fixed families
    at $s=2048$, with pointwise Wilson 95\% intervals.}
    \label{fig:exp_reuse_supp}
\end{figure}

Intervals are pointwise Wilson 95\% intervals conditional on the fixed
candidate family. Families are reported separately rather than pooled.
In particular, zero failures among 2048 trials have Wilson upper endpoint
$0.001872$ and do not establish zero risk.
As a finite-dimensional diagnostic, the maximum error over all rank-$d/2$
projections is computed from the eigenvalues of the same $S$:
\[
    \sup_{\substack{Q=Q^\top=Q^2\\\operatorname{rank}(Q)=d/2}}
    \beta\left|\frac{\operatorname{tr}(QS)}{\operatorname{tr}S}-\frac12\right|
    =\beta\left(
      \frac{\sum_{a=1}^{d/2}\lambda_a^\downarrow(S)}{\operatorname{tr}S}
      -\frac12\right).
\]
At $W=4096$, the median ratio of the finite-family maximum to this
supremum is approximately $0.254$--$0.256$ across the tested conditions.
The observed families are thus well below this saturation limit.
Nevertheless, fixed effective dimension eventually imposes saturation;
these experiments do not estimate the dimension-uniform minimax width
exponent. The dimension-varying experiment in
Section~\ref{app:exp_width_mechanism} instead illustrates the
square-root logarithmic mechanism for Gaussian validation.
Neither ambient embedding is a hardness construction for arbitrary
adaptive queries.

\subsection{Spectral dependence and validation-budget selection}
\label{app:exp_stopping}

\paragraph{Matrices, paths, and cost.}
Set $\rho=0.1$, $\delta=0.05$, and use
\[
    A_{M,N}=\operatorname{diag}\left(
       \sqrt{\frac{1-\rho}{M}}I_M,
       \sqrt{\frac{\rho}{N}}I_N\right).
\]
The coordinate path captures the first $M$ directions and then removes
the $N$ tail directions individually. The first position with error at
most $\rho$ is $r_\star=M$, and
$q(P_{M+j})=\rho(1-j/N)$.
Both certificates use the same path, and each validation batch is reused
throughout it. If $\tau_s$ is the first acceptance position, the measured
delay is $D_s=(\tau_s-M)_+$ and the cost is
$\mathcal C_N(s)=s+\mathbb E D_s$.
Thus validation and each additional direction have unit cost, and the
initial $M$ construction steps are excluded.

All delay expectations are unconditional. By monotonicity, simulating
acceptance from rank $M$ onward gives the exact positive delay even
when the full-path rule would have accepted earlier.
A zero delay includes acceptance at the null boundary $q=\rho$ and
does not distinguish acceptance at $M$ from acceptance before $M$.
These simulations therefore do not estimate the probability of an
unsafe pre-target stop.

\paragraph{Certification rules.}
The energy-ratio rule accepts when $\widehat q\le\theta_s$, where
\[
    \theta_s=\ell^{-1}\left(
       \ell(\rho)-2\operatorname{arcosh}(\delta^{-1/s})\right),
    ~~
    \ell(t)=\log\frac{t}{1-t}.
\]
Block signs use $L=32$ predetermined blocks of size $b=s/L\ge2$.
The cutoff is $k=6$, since
$\mathbb P\{\operatorname{Bin}(32,e^{-1})\le6\}=0.022779$,
whereas the corresponding probability at 7 is $0.054918>\delta$.
Both rules are calibrated at the same nominal level $\delta=0.05$;
the smaller binomial bound reflects the discrete cutoff.

Write $C_\ell$ for captured energy in block $\ell$ and $U_{\ell i}$
for its tail-coordinate energies. With the shape--scale convention,
independent Gaussian probing gives independent variables
\[
    C_\ell\sim\operatorname{Gamma}\left(
         \frac{Mb}{2},\frac{2(1-\rho)}{Mb}\right),
    ~~
    U_{\ell i}\sim\operatorname{Gamma}\left(
         \frac b2,\frac{2\rho}{Nb}\right).
\]
Let $T_\ell=\sum_{i=1}^N U_{\ell i}$ and
$R_\ell(j)=T_\ell-\sum_{i=1}^jU_{\ell i}$.
At tail position $j$, the two decisions are
\[
    \frac{\sum_{\ell=1}^L R_\ell(j)}
         {\sum_{\ell=1}^L(C_\ell+T_\ell)}\le\theta_s,
    ~~
    \sum_{\ell=1}^L
       \mathbf1\{R_\ell(j)-\rho(C_\ell+T_\ell)\ge0\}\le6.
\]
The ratio uses pooled energies, rather than an average of block ratios.

\paragraph{Fixed-budget spectral comparison.}
For $N\in\{512,2048\}$, $M\in\{1,N\}$, and
$s\in\{128,256,512\}$, all coordinate energies are sampled explicitly.
Each of the 12 settings has 512 trials, with the two certificates sharing
the same energies within each trial.
Table~\ref{tab:exp_spectral} reports every mean delay and its Monte Carlo
standard error. The curves in Figure~\ref{fig:exp_stopping_supp}(a,b)
show the same data.
For $M=N$, both energy blocks are dispersed; for $M=1$, the captured
energy is concentrated. The reversal of the ordering at $M=1$ is
retained and illustrates that the advantage of block signs depends on
the spectrum.

\begin{table}[t]
\centering
\caption{Complete fixed-budget spectral results: mean positive delay
$\pm$ one Monte Carlo standard error from 512 paired trials per setting.
Adding $s$ gives the mean excess cost without changing its standard error.}
\label{tab:exp_spectral}
\small
\setlength{\tabcolsep}{5pt}
\begin{tabular}{rrr rr}
\toprule
$N$ & $M$ & $s$ & Energy ratio & Block signs\\
\midrule
512 & 1 & 128 & $167.21\pm1.72$ & $263.91\pm2.04$ \\
512 & 1 & 256 & $125.88\pm1.35$ & $190.37\pm1.85$ \\
512 & 1 & 512 & $91.19\pm1.04$ & $135.18\pm1.41$ \\
\addlinespace[2pt]
512 & 512 & 128 & $168.72\pm0.13$ & $17.11\pm0.22$ \\
512 & 512 & 256 & $125.54\pm0.09$ & $12.50\pm0.16$ \\
512 & 512 & 512 & $91.97\pm0.07$ & $8.91\pm0.11$ \\
\addlinespace[2pt]
2048 & 1 & 128 & $674.96\pm7.21$ & $1066.37\pm7.90$ \\
2048 & 1 & 256 & $494.98\pm5.14$ & $768.80\pm6.84$ \\
2048 & 1 & 512 & $373.09\pm4.27$ & $557.42\pm5.69$ \\
\addlinespace[2pt]
2048 & 2048 & 128 & $673.62\pm0.24$ & $33.61\pm0.45$ \\
2048 & 2048 & 256 & $500.01\pm0.19$ & $23.79\pm0.31$ \\
2048 & 2048 & 512 & $366.63\pm0.15$ & $16.98\pm0.22$ \\
\bottomrule
\end{tabular}
\end{table}

\paragraph{Exact large-scale delay sampling.}
The allocation experiment fixes $M=N$ and uses
$N=2^j$ for $j\in\{12,15,18,21,24,27,30\}$.
Conditional on their total, independent equal-scale Gamma tail energies
have Dirichlet proportions. For one sign block, draw
\[
    G_c,G_t\ \stackrel{\mathrm{iid}}{\sim}
      \operatorname{Gamma}\left(
         \frac{Nb}{2},\frac{2}{Nb}\right),
    ~~
    C=(1-\rho)G_c,~ T=\rho G_t.
\]
The normalized tail vector is independent of these totals and has law
$\operatorname{Dirichlet}(b/2,\ldots,b/2)$.
If $F_j$ is its cumulative mass through coordinate $j$, the block
becomes negative when
$F_j>h_{\mathrm{sign}}$, where
$h_{\mathrm{sign}}=(1-\rho)(1-G_c/G_t)$.
Negative thresholds give zero crossings.
Across $L$ independent blocks, the certificate delay is the
$(L-k)$th smallest crossing time, here the 26th.

For the ratio rule, pool all $s$ probes before sampling:
$G_c,G_t$ have shape $Ns/2$ and scale $2/(Ns)$, and the Dirichlet
coordinate parameter is $s/2$.
Acceptance occurs when $F_j\ge h_{\mathrm{ratio}}$, with
\[
    h_{\mathrm{ratio}}
    =1-\theta_s-\frac{\theta_s(1-\rho)G_c}{\rho G_t}
    =1-\frac{\theta_s}{\rho}
      +\frac{\theta_s(1-\rho)}{\rho}\frac{G_t-G_c}{G_t}.
\]
The implementation uses the latter, more stable expression.
Strict versus non-strict crossings agree almost surely for these
continuous distributions.

For a segment containing $n_0$ coordinates with Dirichlet parameter
$\alpha$, its left-half mass, conditional on the segment total, is
$Z\sim\operatorname{Beta}(n_0\alpha/2,n_0\alpha/2)$.
If the threshold lies in the left half, update it to $h/Z$; otherwise
subtract the left half, add $n_0/2$ to the coordinate offset, and update
the threshold to $(h-Z)/(1-Z)$.
Repeating this conditional split until a single coordinate remains
samples the crossing exactly in $O(\log N)$ steps.
Consequently, the large-$N$ observations are generated from the finite
Gaussian model, without using the asymptotic delay laws as a simulator.
Sign and ratio trials in this experiment use independent streams;
they are not paired realizations of one cached response matrix.

\paragraph{Budget screening and independent evaluation.}
For each $N$, both rules use a common geometric grid of budgets $s=32b$.
Its block sizes are the distinct values
$\operatorname{round}(2^{j/2})$ for integers $j\ge2$ satisfying
$2^{j/2}\le(N+64)/32$, together with the endpoint $(N+64)/32$.
Thus the smallest budget is 64 and the largest is $N+64$.
There are respectively $14,20,26,32,38,44,50$ budgets at the seven
dimensions, for 224 grid points in total.
The upper endpoint follows from $D_s\le N$:
any $s>N+64$ costs more than the feasible choice $s=64$.
The grid does not use the theoretical optimal validation budget.
The ratio search is also restricted to this grid; budgets below 64
and intermediate integers are not searched.

Each rule receives 1024 screening trials at every grid point.
The budget minimizing its screening mean cost is selected, with ties
resolved toward the smaller budget. At that budget, 8192 fresh trials
estimate the cost reported in Table~\ref{tab:exp_allocation}.
No further selection uses the evaluation outcomes.
These are independently evaluated grid-selected costs, not exact
finite-$N$ global minima. Their standard errors are conditional on the
selected budget and do not include screening variability.

\begin{table}[t]
\centering
\caption{Independent evaluation at screening-selected budgets.
Costs are $s+\mathbb E D_s$; entries show means $\pm$ one Monte Carlo
standard error from 8192 fresh trials.}
\label{tab:exp_allocation}
\small
\setlength{\tabcolsep}{4pt}
\begin{tabular}{r rr rr}
\toprule
& \multicolumn{2}{c}{Block signs}
& \multicolumn{2}{c}{Energy ratio}\\
\cmidrule(lr){2-3}\cmidrule(lr){4-5}
$N$ & Selected $s$ & Cost & Selected $s$ & Cost\\
\midrule
$2^{12}$ & 64 & $131.40\pm0.22$
         & 352 & $1220.21\pm0.06$ \\
$2^{15}$ & 96 & $250.26\pm0.51$
         & 1440 & $5054.76\pm0.09$ \\
$2^{18}$ & 192 & $504.70\pm1.05$
         & 5792 & $20584.19\pm0.13$ \\
$2^{21}$ & 352 & $998.34\pm2.18$
         & 23168 & $83098.69\pm0.19$ \\
$2^{24}$ & 736 & $1996.60\pm4.22$
         & 92672 & $333937.20\pm0.27$ \\
$2^{27}$ & 1440 & $3997.61\pm8.51$
         & 524288 & $1338785.98\pm0.32$ \\
$2^{30}$ & 2912 & $7999.00\pm16.96$
         & 2097152 & $5359548.83\pm0.45$ \\
\bottomrule
\end{tabular}
\end{table}

\paragraph{Comparison with the asymptotic laws.}
For the present parameters, the delay constants are
$C_{\mathrm{sign}}=8.386229$ and $C_{\mathrm{ratio}}=4.405944$.
The corresponding continuous asymptotic minimum costs are
\[
    3(C_{\mathrm{sign}}/2)^{2/3}N^{1/3},
    ~~
    3(C_{\mathrm{ratio}}/2)^{2/3}N^{2/3}.
\]
These expressions provide the reference curves and normalizations;
they are not fitted to the observed costs.
The cost slopes are obtained by ordinary least squares of log mean
evaluation cost against log $N$ over the largest four dimensions.
They are $0.333608$ for signs and $0.667890$ for ratios.
Using all seven dimensions gives $0.330509$ and $0.671671$.
Other fitting windows are retained in the result files.

The discrete budget grid produces visible finite-scale effects.
At $N=4096$, the selected sign budget is the lower endpoint 64.
Over the largest four dimensions, the selected ratio \emph{budget}
has slope $0.733354$, even though its evaluated \emph{cost} has slope
$0.667890$. Thus the results support the cost separation but do not
provide an equally precise empirical estimate of the budget exponent.
These descriptive regressions are not confidence intervals or
information-theoretic lower bounds.

\begin{figure}[t]
    \centering
    \includegraphics[width=0.8\linewidth]{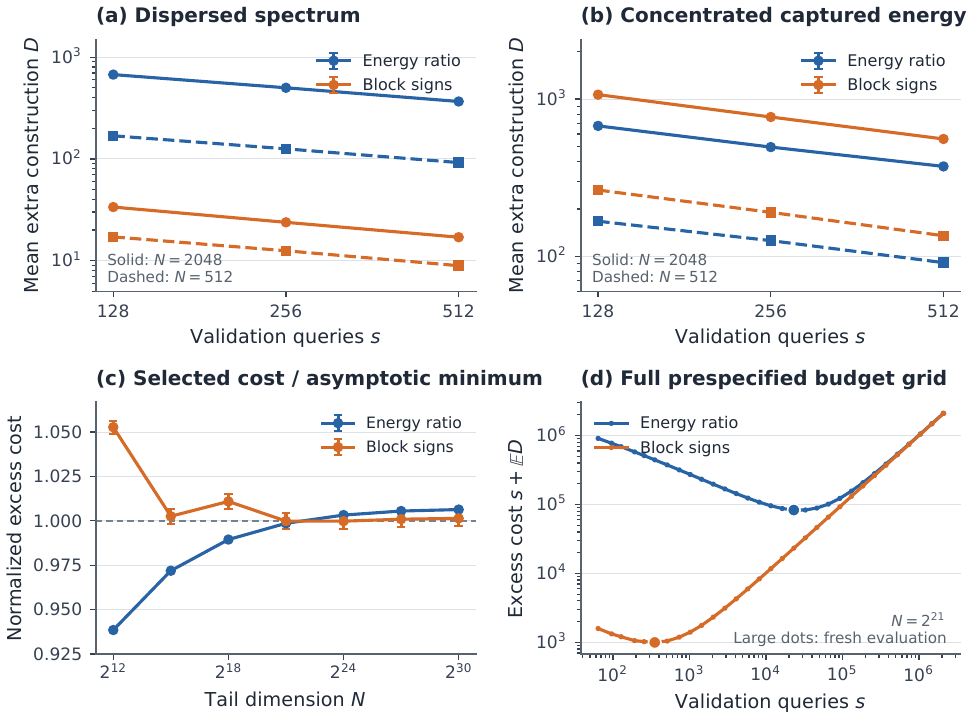}
    \caption{Additional stopping-cost results.
    (a,b) All fixed-budget spectral settings, with 512 paired trials per
    setting. Solid curves use $N=2048$ and dashed curves use $N=512$.
    (c) Independent evaluation costs divided by the continuous asymptotic
    minimum within each rule family.
    (d) The complete screening grid at $N=2^{21}$; large markers show
    fresh evaluation at the selected budgets.
    Error bars in (a--c) are $\pm2$ Monte Carlo standard errors.
    All costs exclude construction before the true target.}
    \label{fig:exp_stopping_supp}
\end{figure}

\subsection{Sampling checks and reproducibility}
\label{app:exp_reproducibility}

\paragraph{Implementation checks.}
For the Dirichlet crossing sampler, the identity
$\mathbb P\{D\le j\}=1-I_h(j\alpha,(N-j)\alpha)$, $1\le j<N$, is checked at
$N=64$, $\alpha=2$, and $h=0.37$, using 100000 samples.
The maximum CDF discrepancy is $0.002504$, below the prespecified
99.9\% Dvoretzky--Kiefer--Wolfowitz tolerance $0.006165$.
At $(N,b)=(32,2)$ and $(128,8)$, each rule is additionally compared
with direct coordinate-energy simulation using 8192 independent trials
per sampler. The largest absolute mean difference is $2.147$ combined
standard errors and the largest empirical CDF discrepancy is $0.018433$.
Across 113319012 Beta draws, no endpoint, nonfinite value, invalid
normalized threshold, clipping, or resampling was recorded.
Threshold updates use extended precision; invalid values cause the
sampler to stop rather than silently modify a draw.

For reuse, an explicit $n=64$, $d=4$ operator verifies projection
idempotence, nesting, and the energy and response identities, with
maximum algebraic error below $7.3\times10^{-16}$.
At $d=4$, $s=16$, Bartlett and direct Gaussian sampling are compared
using 20000 trials each. The four checked statistics are $S_{11}$,
$S_{12}$, $\operatorname{tr}S$, and the rank-half energy fraction.
One marginal two-sample Kolmogorov--Smirnov test has $p=0.00826$;
all original diagnostics are retained.
An additional implementation check, specified after the unchanged
main run, compares the maximum-error statistic at
$d=24$, $s=1024$, $W=64$, using 2048 trials per sampler.
It gives KS statistic $0.02246$ and $p=0.67997$.
The corresponding error-rate difference is approximately one combined
standard error. This supplementary check does not replace or filter
any main-experiment outcomes.

For a single fixed rank-half projection,
$\operatorname{tr}(Q_iS)/\operatorname{tr}S
\sim\operatorname{Beta}(ds/4,ds/4)$ supplies an exact probability
benchmark. Not every pointwise 95\% interval covers this benchmark;
the observed deviations are retained without changing seeds.
All saved means, standard errors, error counts, and selected budgets
for these experiments were independently reconciled with the raw arrays.

For the new width experiment, the exact upper-envelope mean provides
an additional distributional check: across the nine fixed-budget
settings, the largest absolute discrepancy between the empirical and
exact mean is $2.286$ Monte Carlo standard errors. The exhaustive
maximum and projection-dilation checks are described in
Section~\ref{app:exp_width_mechanism}.

\paragraph{Random streams and software.}
The fixed-budget spectral experiment uses NumPy seed 20260920,
iterating over $N$, then $M$, then $s$ in the orders listed above.
Reuse family seeds are 240901, 240902, and 240903.
Validation uses \texttt{SeedSequence([740901,f,s])}, with zero-based
family index $f$; candidate and validation streams are separate.
The explicit projection check uses seed 991002, and the two Gaussian
sampler checks use \texttt{SeedSequence([991001,j])}, with
$j=0,1$ for the original check and $j=2,3$ for the supplementary check.

Allocation streams use
\texttt{SeedSequence([202609214731,phase,N,s,method])},
where \texttt{method} is 1 for signs and 2 for ratios, and
\texttt{phase} is 10 for screening and 20 for evaluation.
Separate phases 1--3 are reserved for sampler checks.

Random streams for the width experiment use
\texttt{SeedSequence([202609220731,\allowbreak phase,\allowbreak d,\allowbreak s])}, with phases
10 and 20 for coarse screening and evaluation, 11 and 21 for refined
screening and evaluation, and 30 for fixed-budget sampling.
The algebraic checks use seed 202609220732.
The archive contains every sampled maximum bound and attainment
diagnostic, the unfiltered coarse and refined results, and plotting code.

The code uses Python 3.12.14, NumPy 2.3.5, SciPy 1.17.0,
Matplotlib 3.10.8, and threadpoolctl 3.6.0.
The reuse and width simulations use batches of 128 and four BLAS
threads.
\end{document}